\documentclass[12pt]{article}

\usepackage[a4paper, marginparwidth=2.5cm, left=3cm, right=3cm]{geometry}

\usepackage{amsmath, amssymb, amsthm}
\usepackage{empheq}
\usepackage{mathtools}
\usepackage{enumerate}

\newtheorem{theorem}{Theorem}[section]
\newtheorem{corollary}[theorem]{Corollary}
\newtheorem{lemma}[theorem]{Lemma}
\newtheorem{proposition}[theorem]{Proposition}
\newtheorem{remark}[theorem]{Remark}
\newtheorem*{remark*}{Remark}
\newtheorem{assumption}[theorem]{Assumption}
\newtheorem{definition}[theorem]{Definition}

\numberwithin{equation}{section}

\def\qp{\sharp}
\def\dmp{\flat}

\def\eps{\varepsilon}
\def\D{{\rm d}}
\def\bigO{\mathcal{O}}

\def\id{\operatorname{id}}
\def\End{\operatorname{End}}
\def\diag{\operatorname{diag}}
\def\Re{\operatorname{Re}}

\def\setR{{\mathbb{R}}}
\def\setC{{\mathbb{C}}}
\def\setN{{\mathbb{N}}}
\def\setZ{{\mathbb{Z}}}
\def\setT{{\mathbb{T}}}

\def\calE{{\mathcal{E}}}
\def\calN{{\mathcal{N}}}
\def\calT{{\mathcal{T}}}

\def\Lexp{{L^\infty_{\rm exp}}}
\newcommand{\rk}[1]{^{[#1]}}

\newcommand\mean[1]{\langle #1 \rangle}
\newcommand\Mean[1]{\big\langle #1 \big\rangle}
\newcommand\MEAN[1]{\left\langle #1 \right\rangle}

\title{\Large\scshape
Multi-frequency averaging and uniform accuracy towards numerical approximations for a Bloch model 
}

\author{Brigitte Bidegaray-Fesquet$^1$ \and Clément Jourdana$^1$ \and Léopold Trémant$^2$\\
\normalsize1. Univ. Grenoble Alpes, CNRS, Grenoble INP\footnote{Institute of Engineering Univ. Grenoble Alpes}, LJK, 38000 Grenoble, France\\
\normalsize 2.  Institut de Recherche Mathématique Avancée, UMR 7501, \\ 
\normalsize Université de Strasbourg et CNRS,  7 rue René Descartes, 67000 Strasbourg, France \\
\normalsize \& INRIA Nancy-Grand Est, TONUS Project, Strasbourg, France}
\date{}   

\begin{document}
\maketitle

\begin{abstract}
We are interested in numerically solving a transitional model derived from the Bloch model. 
The Bloch equation describes the time evolution of the density matrix of a quantum system forced by an electromagnetic wave. 
In a high frequency and low amplitude regime, it asymptotically reduces to a non-stiff rate equation. 
As a middle ground, the transitional model governs the diagonal part of the density matrix. 
It fits in a general setting of linear problems with a high-frequency quasi-periodic forcing and an exponentially decaying forcing. 
The numerical resolution of such problems is challenging. 
Adapting high-order averaging techniques to this setting, we separate the slow (rate) dynamics from the fast (oscillatory and decay) dynamics to derive a new micro-macro problem. 
We derive estimates for the size of the micro part of the decomposition, and of its time derivatives, showing that this new problem is non-stiff. 
As such, we may solve this micro-macro problem with uniform accuracy using standard numerical schemes. 
To validate this approach, we present numerical results first on a toy problem and then on the transitional Bloch model. 
\end{abstract}




\section{Introduction}

The Bloch model describes the time-evolution of the density matrix of a quantum system with a discrete number of energy levels, forced by an electromagnetic wave. 
Different strategies have been proposed in the literature to solve the Bloch equation. 
Let us mention for instance a splitting procedure that solves the different terms separately in an exact way, or a relaxation scheme where the diagonal and the off-diagonal parts of the density matrix are located on a staggered time grid (see \cite{Bidegaray2001} for details on these approaches). 
These schemes have been designed to preserve physical properties of interest or to numerically decouple the equations. 
However, they are not suitable in the case of stiff forcing coefficients.

In \cite{BidegarayFesquet2004a, BidegarayFesquet2004b}, the authors study some high frequency and low amplitude regime, and show that the model asymptotically behaves like a rate equation with averaged transition rates. 
Numerically, the original model is very stiff, meaning that using standard numerical methods requires costly computations. 
The rate equation, however, is a non-stiff autonomous equation which can be solved with standard methods at no additional cost.

In this paper, we are interested in a transitional model from which the rate equation is actually obtained. 
This transitional model, governing the diagonal part of the density matrix, can be seen as a middle-ground between the full original equation and the simpler rate equation. 
It still presents numerical challenges, due to the stiff time-dependence of the transition rates. 
Classical numerical methods may fail to tackle this problem at a reasonable computational cost. 
The present work addresses this issue.\\
 
The main numerical challenge at hand is that of order reduction, a well-known phenomenon documented e.g. in \cite[Sec. IV.15]{Hairer1996} or \cite{Verwer1998} and references therein. 
Here, this is due to the degeneracy of the second (and higher) derivative(s) of the solution in the asymptotic limit. This causes an increase of the error constant in standard estimates, to the point where the theoretical order may no longer be observed.

Formally, if we denote $\eps$ the characteristic time of the problem and $\Delta t$ the time-step of the numerical method, then the error $E_\eps(\Delta t)$ of a standard scheme of order $s$ may be bounded
\begin{equation*}
    E_\eps(\Delta t) \leq C \min \left( \frac{\Delta t^s}{\eps^r}, \Delta t^{s'} \right), 
\end{equation*}
for some positive constant $C$ independent of $\eps$, a uniform order $s' \leq s$ and  a degeneracy order $r \geq 0$. 
When $r > 0$ and $s' < s$, there is a so-called \emph{asymptotic regime} $\Delta t \gg \eps$ in which the behavior of the error does not match the order of the method.  
This is how the \emph{order reduction} phenomenon manifests itself in our context. 
In order to ensure a given error bound, one must use an $\eps$-dependent time-step $\Delta t = \bigO(\eps^{r/s})$, which increases the computational cost. 
To facilitate the discussion surrounding this interaction between the error bound, the characteristic time $\eps$ and the time-step $\Delta t$, we consider different notions of convergence beyond the standard $\Delta t \ll \eps$ paradigm.

If the parameter $\eps$ is small w.r.t. the desired error, then one may consider a non-stiff reduced asymptotic model instead of the full original model. 
In that case, the numerical error is assumed to match the \emph{asymptotic error} $\lim_{\eps \to 0} E_\eps(\Delta t)$, and using standard schemes is possible. 
For the Bloch model, this could mean solving the rate equation instead of the full problem. 
With this approach, the error will plateau for $\Delta t$ sufficiently small.

Some numerical methods are valid in both the standard and asymptotic regimes, at no additional computational cost and with the same order of convergence. 
These are called \emph{asymptotic-preserving} (AP) methods.
The term was coined in \cite{Jin1999} in the context of hyperbolic problems\footnote{%
	The notion of asymptotic preservation is problem-dependent. 
	For example, the implicit Euler method is exact for the asymptotic error when applied to the problem $\dot u = -u / \eps$, but forgoes all accuracy in the limit $\eps \to 0$ for the problem $\dot u = i u / \eps$ with a non-zero initial condition.}, 
and the development of such schemes remains active e.g. in the community of kinetic equations \cite{Crouseilles2016, Albi2020, Anandan2023, Jin2023}. 
While the previous approach of asymptotic models becomes useless when $\eps$ is not small, these methods prove the convergence of the scheme for both the standard error in the regime $\Delta t \ll \eps$ and the asymptotic error.

However, these asymptotic notions do not describe the behavior of the scheme in the \emph{intermediate regime} $\Delta t \sim \eps$, for which the error may be degraded
\footnote{%
	Consider again the implicit Euler scheme applied to $\dot u = -u / \eps$, with initial condition $u(0) = 1$. 
	After a time $\eps$, the solution is $u(\eps) = e^{-1}$, but using a time-step $\Delta t = \eps$ yields the approximation $u_1 = 1/2$.
	Even though the scheme is asymptotic preserving, here the local error $u(\Delta t) - u_1$ is large independently of the time-step, i.e.~the error is severely degraded in that regime.}. 
To encompass every regime, we consider the \emph{uniform error} $\sup_{\eps \in (0,1]} E_\eps(\Delta t)$, defining the \emph{uniform order} of the method. 
A numerical method is said to be \emph{uniformly accurate} (UA) if its computational cost is independent of $\eps$ and if its uniform order matches its standard order. 
Such methods are valid independently of the size of $\eps$ and of the regime (standard, asymptotic or intermediate).\\

Here, we consider a generic linear differential equation, with a time-dependent forcing which can be split in a quasi-periodic part\footnote{%
	A quasi-periodic function is a function generated by multiple non-resonant base frequencies, e.g.~$a_\tau = \cos(\tau) + \cos(\tau \sqrt2)$, with base (angular) frequencies $1$ and $\sqrt2$.} 
and an exponentially decaying part, both with characteristic time $\eps$. 
The aforementioned ``transitional'' model derived from the Bloch equation falls under this category. 
Our strategy consists in using asymptotic expansion techniques to perform a \emph{micro-macro decomposition}, which separates the asymptotic behavior and the error of asymptotic approximation in the macro and micro part respectively. 
This new micro-macro problem is less stiff and can be solved using standard numerical methods, with no order reduction.
This uniform accuracy on the micro-macro problem translates directly to the solution of the original problem.

Our approach uses techniques from high-order averaging, a method to perform asymptotic expansions on highly-oscillatory problems (of characteristic time $\eps$). 
This method views the solution as the composition of an \emph{average dynamics} with a near-identity rapidly-oscillating change of variable. 
This composition is accurate up to an error of size $\bigO(\eps^n)$ with arbitrary order $n$. 
We refer to Lochak-Meunier \cite{Lochak1988} and Sanders-Verhulst-Murdock \cite{Sanders2007} for textbooks on this method. 
Readers might find similarities with the methods of two-scale expansion \cite{Chartier2015a}, WKB expansion \cite{Wentzel1926, Kramers1926, Brillouin1926, Carles2021}, non-linear geometric optics \cite{Crouseilles2017}, or even normal forms \cite{Bambusi2003, Murdock2006}.

Historically, a key tool in performing these expansions was power series in $\eps$. 
They were used in \cite{Perko1969}, the first known result with periodic forcing, and in \cite{Simo1994} when extending the result to a quasi-periodic forcing. 
Even recently in \cite{Chartier2010, Chartier2012}, formal series (specifically, B-series) were used to derive analytical expressions of the mappings constructed by averaging. 
Somewhat recently, however, the authors in \cite{Castella2015} introduced a concise differential algebraic equation for the mappings, called the \emph{homological equation}. 
In the spirit of \cite{Neishtadt1984}, fixed point iterations may be applied on this equation, and the error of approximation can readily be recovered. 
In \cite{Chartier2020a} and \cite{Chartier2022}, this method is used to derive micro-macro problems in the contexts of problems with fast periodic oscillations and with stiff  relaxation respectively, allowing the use of standard numerical methods with uniform accuracy.

We exploit the fixed-point approach based on the closed homological equation of \cite{Castella2015}, which we extend to the case of linear problems with quasi-periodic forcing and added exponential decay.
Compared to their setting of non-linear problems with periodic forcing, here the quasi-resonances (often called small divisors) introduced by the quasi-periodicity degrade the regularity with each fixed-point iteration.
We define appropriate functional spaces to quantify this loss of regularity and to take into account the added exponential decay. 
In this context, we show that we may construct a micro-macro decomposition to any order $n \in \setN$.

In a second time, we study the derivatives of the thereby-obtained micro and macro variables, and show that their derivatives are uniformly bounded up to the $(n+1)$-th derivative, as opposed to the original problem for which the second derivative is degenerate. 
Because of this, the result of uniform accuracy from \cite{Chartier2020a, Chartier2022} still holds, i.e. we may solve the micro-macro problem (and therefore the original transitional problem) with uniform accuracy using a standard scheme.\\ 
The paper is structured as follows. In Section~\ref{sec:results}, we introduce the formalism surrounding the problem, as well as the assumptions we make. 
Crucially, we describe how to construct the micro-macro decomposition and state our results, which are proven in Section~\ref{sec:proofs}.
In Section~\ref{sec:num}, we present some numerical experiments. 
After briefly introducing the numerical schemes we consider, we showcase the importance of each term in the decomposition and verify the uniform accuracy result thanks to a toy model for which an analytical solution is available. 
Finally, we present the Bloch model, and the derivation of the aforementioned ``transitional'' model on which we apply the micro-macro decomposition.
We show that a naive resolution of the problem has severely degraded accuracy, while the micro-macro method converges with uniform accuracy.

\section{Setting and theoretical results}
\label{sec:results}

We wish to derive an equivalent less-stiff problem for the ordinary differential equation
\begin{equation}
	\label{eq:ode_u}
	\partial_t u^\eps(t) = a_{t/\eps} u^\eps(t), 
	\qquad
	u^\eps(0) = u_0 \in X, 
\end{equation}
in some Banach space $(X, |\cdot|)$ for positive finite times $t \in [0,T]$ with $T > 0$ independent of $\eps$, and where $a_\tau$ is a linear map from $X$ to $X$ for all $\tau \geq 0$. 
As detailed below, we assume that $a$ is the sum of a (quasi-)periodic part $a^\qp$ and an exponentially decaying part $a^\dmp$. In the sequel, we say that $a$ admits a ``sharp-flat'' decomposition. 

This derivation is conducted using the ansatz 
\begin{equation}
         \label{eq:mima-ansatz-theo}
	u^\eps(t) = \Phi^\eps_{t/\eps} e^{t A^\eps}  \big( \Phi^\eps_0 \big)^{-1} u_0 
\end{equation}
where $\Phi^\eps_\tau$ is a \emph{near-identity} map for all $\tau \geq 0$ and a non-stiff (uniformly bounded w.r.t. $\eps$) averaged field $A^\eps$ is obtained from $a_\tau$. 
In general, such maps cannot be computed exactly, therefore we seek to compute, for any given $n \in \setN$, \emph{approximate} maps $\Phi\rk n$, $A\rk n$ such that the solution $u^\eps$ of~\eqref{eq:ode_u} may be decomposed into 
\begin{equation*}
	u^\eps(t) = \Phi\rk n_{t/\eps} e^{t A\rk n} \big( \Phi\rk n_0 \big)^{-1} u_0 + w\rk n(t)
\end{equation*}
where $w\rk n = \bigO(\eps^{n+1})$. 
We call this error of approximation $w\rk n$ the \emph{micro} variable and the slow part $v\rk n(t) = e^{t A\rk n} \big( \Phi\rk n_0 \big)^{-1} u_0$ the \emph{macro} variable, which satisfy the so-called micro-macro problem 
\begin{equation*}
	\begin{cases}
		\partial_t v\rk n(t) = A\rk n v\rk n(t), & v\rk n(0) = \big( \Phi\rk n_0 \big)^{-1} u_0, \\
		\partial_t w\rk n(t) = a_{t/\eps} w\rk n(t) - \delta\rk n_{t/\eps} v\rk n(t) , \quad & w\rk n(0) = 0,
	\end{cases}
\end{equation*}
for some defect $\delta\rk n$ computed from $\Phi\rk n$ and $A\rk n$. 
Our goal is to prove that this new problem is \emph{non-stiff} up to the $(n+1)$-th derivative and can therefore be computed with \emph{uniform accuracy} up to order $n$ using standard numerical schemes. 

This section states the theoretical results of this paper.
Subsection~\ref{subsec:def} introduces the mathematical setting and the formalism necessary to state the problem. 
Subsection~\ref{subsec:hyp} describes the assumptions we make on the problem, and define the aforementioned ``sharp-flat decomposition''. 
Subsection~\ref{subsec:results} finally details our construction of the maps $\Phi\rk n$, $A\rk n$ and $\delta\rk n$ and states results on the properties of the micro-macro decomposition. 

\subsection{Definitions and notations}
\label{subsec:def}

Here we introduce some formalism associated to endormorphisms, multivariate periodic functions and exponentially decaying functions. 
We also introduce the time-average operator and some norms associated to such functions. 
Throughout the paper, the set of endomorphisms, denoted $\End(X)$, is endowed with the induced norm, denoted $|\cdot|$.

\begin{definition}[Time average and KBM mappings]
	Given a mapping $\tau \in \setR_+ \mapsto \varphi_\tau \in \End(X)$, we define the time average 
	\begin{equation}
		\label{eq:time-avg}
		\mean{\varphi}
		:= \lim_{\tau \to \infty} \frac1\tau \int_0^\tau \varphi_\sigma \D \sigma.
	\end{equation}
	A \emph{continuous} function $\tau \mapsto \varphi_\tau$ such that this limit converges is called a \emph{KBM} mapping\footnote{%
		The acronym KBM stands for Krylov, Bogoliubov and Mitropolsky, see e.g.~\cite{Sanders2007}.}. 
	We denote $\calE$ the vector space of KBM mappings from $\setR_+$ to $\End(X)$. 
\end{definition}

Among KBM mappings, we are interested in two subspaces which we introduce briefly: quasi-periodic mappings and exponentially decaying mappings.

\subsubsection{Quasi-periodic mappings}

By quasi-periodic, we mean mappings that are generated from a finite number of angular frequencies $\omega_1, \ldots, \omega_r$ and a multivariate $2\pi$-periodic map $\theta \in \setT^r \mapsto \varphi^\qp_\theta$ with $\setT = \setR / (2\pi\setZ)$. 
Denoting $\omega = (\omega_1, \ldots, \omega_r)$, the quasi-periodic map is given by $\tau \mapsto \varphi^\qp_{\omega\tau}$, i.e.~it is given by evaluating $\varphi^\qp$ along the curve $\tau \mapsto (\omega_1\tau, \ldots, \omega_r\tau)$. 
If the generating map $\varphi^\qp$ is continuous, then it coincides with its Fourier series, 
\begin{equation*}
	\forall \theta \in \setT^r, 
	\quad
	\varphi^\qp_\theta = \sum_{\alpha \in \setZ^r} e^{i \alpha\cdot\theta} \widehat\varphi^\qp _\alpha. 
\end{equation*}
Here we use the multi-index notation $\alpha = (\alpha_1, \ldots, \alpha_r) \in \setZ^r$ to obtain the phase $\alpha\cdot\theta = \alpha_1 \theta_1 + \ldots + \alpha_r \theta_r$ and Fourier coefficient $\widehat\varphi^\qp_\alpha \in \End(X)$. 
We furthermore denote $|\alpha| = |\alpha_1| + \ldots + |\alpha_r|$. 
In the sequel the number of frequencies $r$ and the vector of frequencies $\omega$ are fixed.

\begin{definition}
	We define $\calE^\qp$ the set of continuous quasi-periodic maps with frequencies $\omega$, 
	\begin{equation*}
		\calE^\qp 
		:= \left\{ \tau \mapsto \varphi^\qp_{\omega\tau},~ 
		\varphi^\qp \in C^0\big(\setT^r, \End(X) \big) \right\}. 
	\end{equation*}
	Some particularly regular maps are such that the Fourier coefficients are exponentially decreasing, which is quantified by the functional spaces, for $\kappa \geq 0$, 
	\begin{equation*}
		\calT_\kappa := \left\{ \varphi^\qp \in C^0\big(\setT^r, \End(X) \big),~ 
		\sum_{\alpha \in \setZ^r} e^{\kappa|\alpha|} \big| \widehat\varphi^\qp_\alpha \big| < \infty \right\}. 
	\end{equation*}
	The set of quasi-periodic mappings $\tau \mapsto \varphi^\qp_{\omega\tau}$ this generates is denoted $\calE^\qp_\kappa$.
\end{definition}

For $\kappa > 0$, all mappings in $\calT_\kappa$ are smooth, and for $\kappa = 0$, the mappings in $\calT_0$ are continuous. 
Additionally, if $0 \leq \kappa^- \leq \kappa^+$, then $\calT_{\kappa^+} \subset \calT_{\kappa^-}$. 
Therefore, we obtain the following inclusions
\begin{equation*}
	\calE^\qp_{\kappa^+} \subset \calE^\qp_{\kappa^-} \subset \calE^\qp_0 \subset \calE^\qp. 
\end{equation*}

\subsubsection{Exponentially decaying mappings} 

By exponentially decaying, we mean bounded mappings $\tau \mapsto \varphi^\dmp_\tau$ which are also $\bigO(e^{-\tau})$ for $\tau \to \infty$. 

\begin{definition}
	The set of exponentially decaying functions we consider is denoted $\Lexp$ and is defined by 
	\begin{equation*}
		\Lexp := \left\{ \varphi^\dmp \in L^\infty\big( \setR_+, \End(X) \big),~ 
		\sup_{\tau \in \setR_+} e^\tau |\varphi^\dmp_\tau| < \infty \right\}. 
	\end{equation*}
	A subset of this are KBM mappings, which we denote 
	\begin{equation*}
		\calE^\dmp := \Lexp \cap C^0\big( \setR_+, \End(X) \big). 
	\end{equation*}
\end{definition}

\begin{remark}
	The rate of decay (namely 1) is chosen without loss of generality, since it could be obtained with a time rescaling.
\end{remark}

\subsubsection{Norms}

We endow the above defined functional spaces with the following norms.

\begin{definition}[Norms]
	\label{def:norms}
	Remember that we denote $|\cdot|$ the induced norm on $X$. 
	For mappings $\varphi \in \calE$, $\varphi^\qp \in \calT_\kappa$ and $\varphi^\dmp \in \Lexp$, we denote
	\begin{equation*}
		\| \varphi \| := \sup_{\tau \geq 0} |\varphi_\tau|, 
		\qquad
		\| \varphi^\qp \|_\kappa 
		:= \sum_{\alpha \in \setZ^r} e^{\kappa |\alpha|} |\widehat\varphi^\qp_\alpha|, 
		\qquad
		\| \varphi^\dmp \|_\Lexp := \sup_{\tau \geq 0} e^\tau |\varphi^\dmp_\tau|.
	\end{equation*}
\end{definition}

Note that for $0 \leq \kappa^- \leq \kappa^+$ and all $\varphi^\qp \in \calE^\qp_{\kappa^+}$,
\begin{equation}
	\label{eq:kappainequalities}
	\| \varphi^\qp \| \leq \| \varphi^\qp \|_0 \leq \| \varphi^\qp \|_{\kappa^-} \leq \| \varphi^\qp \|_{\kappa^+},
\end{equation}
and for all $\varphi^\dmp \in \Lexp$, 
\begin{equation}
	\label{eq:expinequalities}
	\forall\tau\in\setR^+, \quad | \varphi^\dmp_\tau | \leq e^{-\tau} \| \varphi^\dmp \|_\Lexp.
\end{equation}

\begin{remark}
	In the entire upcoming reflexion, $\calT_\kappa$ could be replaced by the set of functions on $\setT^r$ analytic with radius everywhere greater than $\kappa$. 
	The norm $\|\cdot\|_\kappa$ would then be replaced by the infinite norm on the analytical extension of radius $\kappa$ to the complex domain, i.e. symbolically 
	\begin{equation*}
		\forall \varphi^\qp \in \calT_\kappa, 
		\quad 
		\| \varphi^\qp \|_\kappa \leq \sup_{\zeta \in \setT^r_\kappa} | \varphi^\qp_\zeta |
	\end{equation*}
	with $\setT^r_\kappa := \{ \theta + \xi,~ (\theta, \xi) \in \setT^r \times \setC^r, |\xi| \leq \kappa \}$. 
	For $\theta \in \setT^r, |\xi| \leq \kappa$, the analytical extension is defined as $\varphi^\qp _{\theta + \xi} = \displaystyle \sum\limits_{\alpha \in \setN^r} \frac{\partial^{|\alpha|} \varphi^\qp_\theta}{\partial\theta_1^{\alpha_1} \cdots \partial\theta_r^{\alpha_r}} \frac{\xi_1^{\alpha_1} \cdots \xi_r^{\alpha_r}}{\alpha_1! \cdots \alpha_r!}$.  
\end{remark}

\subsection{Assumptions on the problem}
\label{subsec:hyp}

So far, we have described the two types of KBM mappings we consider, namely the ``sharp'' and the ``flat'' parts of an aforementioned decomposition. 
The ``sharp'' part is obtained in a one-way relationship from a function on $\setT^r$. 
We start by introducing an assumption (namely a non-resonance condition on the frequencies $\omega$) which makes this a two-way relationship. 
Thanks to this, we may define the ``sharp-flat'' decomposition rigorously, allowing us to finally state how our problem fits into this setting.

\subsubsection{Non-resonance}

\begin{assumption}
	\label{hyp:non-resonance}
	The vector of angular frequencies $\omega = (\omega_1, \ldots, \omega_r)$ is strongly non-resonant in the sense that it satisfies the following Diophantine inequality
	\begin{equation}
		\label{eq:diophantine}
		\exists c_D > 0, \quad \exists \nu \geq r-1, \quad
		\forall \alpha \in \setZ^r \setminus \{ \mathbf{0} \}, \quad
		| \alpha \cdot \omega | \geq \frac{c_D}{|\alpha|^\nu},
	\end{equation}
	where $\alpha \cdot \omega = \alpha_1 \omega_1 + \ldots + \alpha_r \omega_r$ and  $\mathbf{0} = (0, \ldots, 0) \in \setZ^r$. In the mono-frequency case $r = 1$, we choose $\nu = 0$ and $c_D = |\omega|$.
\end{assumption}

\begin{remark}
	While this may seem restrictive at first glance, it is classically known (see for instance \cite{Arnold1963}, \cite[App. 8]{Arnold1989} or \cite[Chap. X]{Hairer2006}) that this condition is satisfied for almost all frequency vectors in any bounded subset of $\setR^r$. 
	In practice, we mostly require $\omega$ to be ``well-prepared'' in the sense that no frequencies are rationally dependent, which is \emph{usually} enough for the Diophantine condition to be met.
\end{remark}

This assumption implies that the set of cancelling combination is zero -- which is to say that the vector of frequencies $\omega = (\omega_1, \ldots, \omega_r)$ is such that 
\begin{equation*}
	\omega^\perp
	:= \left\{ \alpha \in \setZ^r \text{ s.t. } \alpha \cdot \omega = 0 \right\}
	= \{\mathbf{0} \}.
\end{equation*}
As such, we may apply Arnold's theorem, which states that for continuous quasi-periodic maps, the time average coincides with the phase average. 
This may be written 
\begin{equation*}
	\forall \varphi^\qp \in C^0\big(\setT^r, \End(X) \big),\ \forall \theta_0 \in \setT^r, 
	\quad 
	\lim_{T \to \infty} \frac1T \int_0^T \varphi^\qp_{\theta_0 + \omega\tau} \D\tau 
	= \frac1{(2\pi)^r} \int_{\setT^r} \varphi^\qp_{\theta} \D\theta. 
\end{equation*}
In particular, this means that we can recover the Fourier coefficients of the function $\theta \mapsto \varphi^\qp_\theta$ from the quasi-periodic function $\tau \mapsto \varphi^\qp_{\omega\tau}$. 
As such we may identify both functions and extend the notion of Fourier series to the quasi-periodic mapping of $\calE^\qp_\kappa$.

\subsubsection{``Sharp-flat'' decomposition}

Under Assumption \ref{hyp:non-resonance}, we introduce the set $\calE_\kappa$ that describes what we name  ``sharp-flat'' maps as well as an associated norm $\calN_\kappa$ that will be used all along this paper to easily handle the (quasi)-periodic and exponentially decaying maps that we consider.

\begin{definition}
	\label{def:direct_sum}
	For all $\kappa \geq 0$, we define
	\begin{equation*}
		\calE_\kappa = \calE^\qp_\kappa \oplus \calE^\dmp.
	\end{equation*}
Given a mapping $\varphi \in \calE_\kappa$, it splits uniquely as $\varphi_\tau = \varphi^\qp_{\omega\tau} + \varphi^\dmp_\tau$ and we endow the space $\calE_\kappa$ with the norm
	\begin{equation*}
		\calN_\kappa(\varphi) := \| \varphi^\qp \|_\kappa + \| \varphi^\dmp \|_\Lexp,
	\end{equation*}
	where $\|\cdot\|_\kappa$ and $\|\cdot\|_\Lexp$ are given in Definition~\ref{def:norms}.
\end{definition}

The proof that the sum is indeed direct, and that this defines a unique ``sharp-flat'' decomposition for elements in $\calE_\kappa$ is presented in Appendix~\ref{app:sec:sharp-flat-properties}. 
To summarize, the average of flat functions is always zero, therefore the Fourier coefficients of a function in $\calE^\qp_\kappa \cap \calE^\dmp$ are all zero. 
Since functions in $\calE^\qp_\kappa$ coincide with their Fourier series, this function can only be zero.

\begin{proposition}
	\label{prop:algebra}
    	The space $\calE_\kappa$ is an algebra, and if $ \varphi$, $\widetilde\varphi \in \calE_\kappa$ then $\psi = \varphi \widetilde\varphi \in \calE_\kappa$ and
    	\begin{equation*}
		\psi_\tau = \psi^\qp_{\omega \tau} + \psi^\dmp_\tau  
		\text{ with }
		\psi^\qp_\theta = \varphi^\qp_\theta \widetilde\varphi^\qp_\theta
		\text{ and }
		\psi^\dmp_\tau  = \varphi^\dmp_\tau \widetilde\varphi^\qp_{\omega \tau} 
		+ \varphi^\qp_{\omega \tau} \widetilde\varphi^\dmp_\tau  
		+ \varphi^\dmp_\tau \widetilde\varphi^\dmp_\tau. 
    	\end{equation*}
 	Furthermore, the norm is algebraic, i.e. 
    	\begin{equation*}
        		\calN_\kappa(\psi) \leq \calN_\kappa(\varphi) \calN_\kappa(\widetilde\varphi). 
       	\end{equation*}
\end{proposition}

Again, the proof is presented in Appendix~\ref{app:sec:sharp-flat-properties}.  
To finish, we present some norm inequalities that will be often used in the next parts.
Let $\varphi \in \calE_\kappa$.  
The flat part of its time average is zero, i.e. $\mean{\varphi}^\flat = 0$. Indeed, since $\mean{\varphi}$ is constant, it is therefore periodic and belongs to $\calE_\kappa$ for all $\kappa \geq 0$ with norm $\calN_\kappa(\mean{\varphi}) = | \mean{\varphi} |$. 
Consequently,
\begin{equation*}
	\calN_\kappa(\varphi - \mean{\varphi}) 
	= \| \varphi^\qp - \mean{\varphi^\qp} \|_\kappa + \| \varphi^\dmp \|_\Lexp.
\end{equation*}
Then, since $\mean{\varphi^\qp} = \widehat\varphi^\qp_0$,
	\begin{equation*}
	\| \varphi^\qp - \mean{\varphi^\qp} \|_\kappa 
	= \sum_{\alpha \in \setZ^r, \alpha \neq 0} e^{\kappa|\alpha|} \left| \widehat\varphi^\qp_\alpha \right|
	\leq \sum_{\alpha \in \setZ^r} e^{\kappa|\alpha|} \left| \widehat\varphi^\qp_\alpha \right| 
	= \| \varphi^\qp \|_\kappa. 
	\end{equation*}
It follows that
\begin{equation*}
        \calN_\kappa(\varphi - \mean{\varphi}) \leq \calN_\kappa(\varphi).
    \end{equation*}
Also, let $0 \leq \kappa_- \leq \kappa_+$ and $\varphi \in \calE_{\kappa_+}$. 
From the definition of the time average \eqref{eq:time-avg} and the inequalities \eqref{eq:kappainequalities} and \eqref{eq:expinequalities}, it clearly holds that
\begin{equation}
	\label{eq:ineq_norm_average}
	| \mean{\varphi} | \leq \| \varphi \| \leq \calN_{\kappa_-}(\varphi) \leq \calN_{\kappa_+}(\varphi).
\end{equation}
    
\subsubsection{Assumptions on the linear map}

	This functional setting being laid out, let us describe how our problem fits into it. 
	Specifically, we introduce positive constants $\mu$ and $M$ which  quantify respectively the regularity of $a^\qp$ and the size of $\tau \mapsto a_\tau$.

\begin{assumption}
	\label{hyp:analyticity_and_bounds}
	The mapping $\tau \mapsto a_\tau \in \End(X)$ is in $\calE_\mu = \calE^\qp_\mu \oplus \calE^\dmp$ for some $\mu > 0$, i.e. it can be written 
	\begin{equation*}
		a_\tau = a^\qp_{\omega\tau} + a^\dmp_\tau
	\end{equation*}
	with $a^\qp \in \calT_\mu$ and $a^\dmp \in \Lexp \cap C^0$. 
	Furthermore, we define a constant $M > 0$ such that 
	\begin{equation*}
		\| a \| \leq \calN_\mu(a) \leq M. 
	\end{equation*}
\end{assumption}

\begin{remark}
	In the mono-frequency case $r = 1$, we may replace the analyticity assumption on $a^\qp$ by a continuity assumption.
\end{remark}

Since the goal of the present work is to apply numerical schemes, it is necessary to quantify the regularity of $\tau \mapsto a_\tau$. 
Indeed, if $a \in C^q$, then a scheme of order $s > q$ will see its order reduced to $q$, even in the non-stiff regime.

\begin{assumption}
	\label{hyp:derivatives}
	The mapping $\tau \mapsto a_\tau$ is of class $C^q$ for some $q \in \setN$. 
	Additionally, there exists $C_a^{(q)} > 0$ such that 
	\begin{equation*}
		\sup_{0 \leq p \leq q} \| \partial_\tau^p a \| \leq C_a^{(q)} M.
	\end{equation*}
\end{assumption}

\subsection{Results on the decomposition}
\label{subsec:results}

We may now construct the micro-macro decomposition by performing asymptotic expansions which separate the exponential dynamics in $\calE_\mu - \mean{\calE_\mu}$ (contained in a change of variable $\Phi^\eps$) from the average dynamics in $\mean{\calE_\mu}$ (contained in a vector field $A^\eps$).
From these, we derive a micro-macro problem which can be solved with uniform accuracy. 
This section describes this construction and states our results of well-posedness and uniform accuracy. 
The proofs can be found in the next section.

\subsubsection{Homological equation}

Injecting the exact decomposition \eqref{eq:mima-ansatz-theo} into \eqref{eq:ode_u}, we obtain the following ``homological equation'' on $\Phi^\eps$ and $A^\eps$, 
\begin{equation}
	\label{eq:homological}
	\partial_\tau \Phi^\eps_\tau
	= \eps \big( a_\tau \Phi^\eps_\tau - \Phi^\eps_\tau A^\eps \big) , 
	\text{ where }
	A^\eps = \Mean{\Phi^\eps}^{-1} \mean{a \Phi^\eps},
\end{equation}
which may be rewritten by introducing a non-linear operator $\Lambda$, 
\begin{equation*}
	\partial_\tau \Phi^\eps_\tau = \eps \Lambda\big\{ \Phi^\eps \big\}_\tau 
	\text{ with }
	\Lambda\{\varphi\}_\tau = a_\tau \varphi_\tau - \varphi_\tau \mean{\varphi}^{-1} \mean{a \varphi}.
\end{equation*}
In particular if $\mean{\varphi} = \id$, then $\Lambda\{\varphi\}_\tau = a_\tau \varphi_\tau  - \varphi_\tau \mean{a \varphi}$. 

Then, we define approximations of $\Phi^\eps$ and $A^\eps$ by a fixed point iteration. 
Starting from $\Phi\rk0 = \id$, we construct them with the relations
\begin{equation}
	\label{eq:def_Phin}
	A\rk n = \Mean{\Phi\rk n}^{-1} \mean{a \Phi\rk n}, 
	\qquad
	\partial_\tau \Phi\rk{n+1}_\tau  = \eps \Lambda\big\{ \Phi\rk n \big\}_\tau.
\end{equation}

\subsubsection{Closure condition}

In order to solve the homological equation as well as its approximations, one needs to impose a closure condition on $\Phi^\eps$ (and consequently on $\Phi\rk n$). 
In this paper\footnote{%
	Another common choice is $\Phi^\eps_0 = \id$, which presents nice geometric properties but leads to more complex calculations---see e.g.~\cite{Chartier2020a}. 
	Since these properties are not needed here, we do not consider this possibility, although all our estimates remain valid with this geometric closure condition, up to some tweaking of the constants.}, 
we consider the so-called standard averaging by choosing  $\mean{\Phi^\eps} = \id$ (and consequently $\mean{\Phi\rk n} = \id$).
This yields the relations 
\begin{subequations}
	\label{eq:iter_Phi}
	\begin{align*}
		A\rk n & = \mean{a \Phi\rk n}, \\
		\Phi\rk{n+1}_\tau & = \id + \eps \int_0^\tau \Lambda\big\{\Phi\rk n\big\}_\sigma \D\sigma 
		- \eps \MEAN{\int_0^{\:\scriptscriptstyle\bullet} \Lambda \big\{ \Phi\rk n \big\}_\sigma \D\sigma}.
	\end{align*}
\end{subequations}
These iterations may be performed explicitly using symbolic calculus. 

\subsubsection{Micro-macro variables}

To motivate the interest of these approximations, let us immediately introduce the \emph{micro-macro variables} $(v\rk n, w\rk n)$ given by 
\begin{equation*}
	u^\eps(t) = \Phi\rk n_{t/\eps} v\rk n(t) + w\rk n(t)
\end{equation*}
with $v\rk n(t) = e^{t A\rk n} \big( \Phi\rk n_0 \big)^{-1}(u_0)$. 
A straightforward computation yields the following \emph{micro-macro problem}, 
\begin{subequations}
	\label{eq:mima}
	\begin{empheq}[left=\left\lbrace,right=\right.]{align}
		\label{eq:mima_v}
		\partial_t v\rk n(t) & = A\rk n v\rk n(t), 
		& & v\rk n(0) = \big( \Phi\rk n _0 \big)^{-1} (u_0), \\
		\label{eq:mima_w}
		\partial_t w\rk n(t) & = a_{t/\eps} w\rk n(t) - \delta\rk n_{t/\eps} v\rk n(t),
		& & w\rk n(0) = 0,
	\end{empheq}
\end{subequations}
where we introduced the defect $\delta\rk n$ defined as
\begin{equation}\label{eq:def_delta}
	\delta\rk n_\tau
	= \frac1\eps \partial_\tau \Phi\rk n_\tau
	- \left( a_\tau \Phi\rk n_\tau - \Phi\rk n_\tau A\rk n \right) 
	= \Lambda\{ \Phi\rk{n-1} \}_\tau - \Lambda\{ \Phi\rk n \}_\tau
\end{equation}
with the convention $\Lambda\{ \Phi\rk{-1} \} = 0$. 
This defect is of zero time average, i.e. $\mean{\delta\rk n} = 0$, and quantifies the quality of approximation in the homological equation~\eqref{eq:homological}. 
Combined with the drift $v\rk n$, it generates a source term in the equation on $w\rk n$. 

The equation \eqref{eq:mima_w} seems to be as stiff as the original equation \eqref{eq:ode_u} at first glance, but we will see in the sequel that, due to the small size of $w\rk n(0)$ and of $\delta\rk n$, $w\rk n$ remains of size $\bigO(\eps^{n+1})$ at all times. 
This initializes an induction to prove that the derivatives of $w\rk n$ are bounded up to order $n+1$.

\subsubsection{Well-posedness of the micro-macro decomposition}

In the introduction of Section~\ref{sec:results} we have claimed to construct a near-identity map $\Phi^\eps _\tau$. 
In the construction of $\Phi\rk n$, we introduce a parameter $c\in]0,1[$ and assume that $\| \Phi\rk n - \id \| \leq c$. 
This will be valid for a small enough $\eps\leq\eps_n$.
The upper bound $\eps_n$, which depends on both $n$ and $c$, will be constructed along the proof of the following theorem that sums up the well-posedness of the micro-macro decomposition.

\begin{theorem}
	\label{thm:decomp_mima}
	Under Assumptions~\ref{hyp:non-resonance} and \ref{hyp:analyticity_and_bounds}, for all $n \in \setN$, there exists $\eps_n > 0$ such that the decomposition of order $n$ exists, meaning that, for $\eps \leq \eps_n$,
	\begin{equation}
		\label{prop:asympt:eq:Phi}
		\| \Phi\rk n - \id \| \leq \frac{c\eps}{\eps_n} \leq c, 
	\end{equation}
	\begin{equation}
		\label{prop:asympt:eq:A}
		| A\rk n | \leq (1+c) M. 
	\end{equation}
	Furthermore, integrating the defect $t \mapsto \delta\rk n_{t/\eps}$ yields an error of approximation of size $\bigO(\eps^{n+1})$, which is translated through the relation 
	\begin{equation}
		\label{prop:asympt:eq:delta}
		\forall \tau \geq 0, 
		\quad 
		\left| \eps \int_0^\tau \delta\rk n_\sigma \D\sigma \right| 
		\leq \left( \frac\eps{\eps_n} \right)^{n+1}.
	\end{equation}
\end{theorem}

This theorem will be proven in Section~\ref{sec:proof_wellposedness}. 
A direct consequence of \eqref{prop:asympt:eq:Phi} is that the inverse $(\Phi\rk n_0)^{-1}$ is well-defined and may be bounded 
\begin{equation}
	\label{eq:bound_inv_Phi}
	| (\Phi\rk n_0)^{-1} | \leq 1 / (1-c).
\end{equation}
This is crucial for the definition and the boundedness of $v\rk n$. 
Moreover, a direct by-product of the proof is the following bound on the defect 
\begin{equation}
	\label{eq:bound_delta}
	\| \delta\rk n \| \leq M \left( \frac\eps{\eps_n} \right)^n.
\end{equation}

The coefficient $\eps_n$ may be chosen of the form $\eps_n = \eps_0 / (n+1)^\nu$ with $\eps_0$ depending on $c$ and the constants appearing in Assumptions~\ref{hyp:non-resonance} and \ref{hyp:analyticity_and_bounds}. 
As such, increasing the order of accuracy requires a reduction in the maximum size of $\eps$. 	
However, in the mono-frequency case $r = 1$, $\nu = 0$, this reduction does no longer appear and the iterations converge for $\eps \leq \eps_0$\footnote{%
	This is a known result of single-frequency linear averaging, available e.g. in~\cite{Castella2015} with a straightforward proof. 
	In their paper, these authors analyze the non-linear setting, and here we analyze the multi-frequency setting with an added decay, further extending the method.}.

As consequence of the well-posedness  of the micro-macro decomposition,  the micro-macro variables $v \rk n$ and  $w \rk n$ solutions to \eqref{eq:mima} are bounded in finite time as stated by the following corollary (proven in Section~\ref{sec:proof_corollary_mima}).

\begin{corollary}
	\label{corollary:mima_well_posed}
	Under Assumptions~\ref{hyp:non-resonance} and \ref{hyp:analyticity_and_bounds}, for any $n \in \setN$ and any $\eps \leq  \eps_n$ (as defined in Theorem~\ref{thm:decomp_mima}), the solutions to \eqref{eq:mima} are bounded at all times by 
	\begin{equation*}
		| v\rk n(t) | \leq e^{(1+c)t M} \frac{|u_0|}{1 - c},
		\quad
		| w \rk n (t) | 
		\leq \left( \frac\eps{\eps_n} \right)^{n+1} e^{t\|a\|} \left( 2e^{(1+c)t M} - \right) \frac{|u_0|}{1-c}.
	\end{equation*}
	Thus, for any $T > 0$, the micro-macro problem~\eqref{eq:mima} is uniformly bounded on $[0,T]$ with $v \rk n$ of size $\bigO(1)$ and $w\rk n$ of size $\bigO(\eps^{n+1})$. 
\end{corollary}

\begin{remark}
	Here we impose the order of the expansion and deduce the size of the error term from it. 
	However, the condition $\eps \leq \eps_0/(n+1)^\nu$ may be interpreted the other way around, and one may wish to choose an ``optimal'' order $n$ depending on the size of $\eps$. 
	This yields the optimal exponential bound on $w\rk{n(\eps)}$ for $\nu > r - 1$ found in~\cite{Simo1994}.
\end{remark}

\subsubsection{Well-posedness of the derivatives and uniform accuracy}
\label{sec:result_ua}

In order to analyze the order of numerical schemes based on the micro-macro problem, we also need estimates on time-derivatives of the defect. 
\begin{theorem}
	\label{thm:bound_derivatives_delta}
	Under Assumptions~\ref{hyp:non-resonance}, \ref{hyp:analyticity_and_bounds} and \ref{hyp:derivatives}, the derivatives of the defect up to order $q$ remain of size $\bigO(\eps^n)$. 
	Specifically, there exists $C^{(q)}_\delta > 0$, depending on $q$, $c$ and $n$, such that, for all $\eps \leq \eps_n$, 
	\begin{equation*}
		\sup_{0 \leq p \leq q} \| \partial_\tau^p \delta\rk n \| 
		\leq C^{(q)}_\delta M \left( \frac\eps{\eps_n} \right)^n.
	\end{equation*}	
\end{theorem}

This theorem will be proven in Section~\ref{sec:derivatives}. 
These estimates on the defect gives estimates for the micro-macro variables $v \rk n$ and $w \rk n$ limited by the order of the micro-macro decomposition $n$ and the regularity $q$ of the linear map $a$.

\begin{corollary}
	\label{corollary:mima_bounded_derivatives}
	The derivatives of the micro-macro problem~\eqref{eq:mima} are uniformly bounded up to order $\min(n,q)+1$. 
	Indeed, at fixed final time $T$ and for all $t \in [0,T]$, for all $p \in \{ 0, \ldots, \min(n,q)+1 \}$, 
	\begin{equation*}
		| \partial_t^p v\rk n(t) | = \bigO(1), 
		\qquad
		| \partial_t^p w\rk n(t) | = \bigO(\eps^{n+1-p}) .
	\end{equation*}
\end{corollary}

We now consider the discretization of the micro-macro problem~\eqref{eq:mima} on the time interval $[0,T]$. 
To simplify the presentation, we discretize uniformly this time interval introducing $t^\ell=\ell \Delta t$ for $\ell = 0, \dots, L$ where $L+1$ is the number of discretization points and $\Delta t = T/L$ the time step. 
We denote $(v^\ell, w^\ell)$ the approximate values at time $t^\ell$ of the solution of the micro-macro problem \eqref{eq:mima} for a given order $n$.  

\begin{corollary}
	\label{corollary:mima_ua}
	Using a standard stable one-step scheme of non-stiff order $s \leq \min(n,q)$, i.e. a method which exhibits order $s$ of convergence when applied to \eqref{eq:ode_u} with $\eps = 1$ for all $\Delta t \in [0,\Delta t^*]$ (with $\Delta t^* > 0$ being the stability threshold, which depends only on $M$), the micro-macro problem \eqref{eq:mima}  can be solved with uniform accuracy. 
	More precisely, we have, for all $\Delta t \in [0,\Delta t^*]$, the bound 
	\begin{equation*}
	\sup_{\eps \in (0,\eps_n]} \max_{0 \leq \ell \leq L} | \Phi\rk n_{t^\ell/\eps} v^{\ell} + w^{\ell} - u^\eps(t^\ell) | \leq C \Delta t^s,
	\end{equation*}
	where the constant $C$ is independent of $\Delta t$.
\end{corollary}

Using a Runge-Kutta integral scheme, this order may be increased by one. 
For instance, for the problem $\partial_t y(t) = b_{t/\eps} y(t)$, the Runge-Kutta scheme of order 2 is 
\begin{equation*}
	\widetilde{y}^{\ell+1/2} = y^\ell + \Delta t\ b_{t^\ell/\eps}\ y^\ell, 
	\qquad
	y^{\ell+1} = y^\ell + \Delta t\ b_{(t^{\ell} + \Delta t/2)/\eps}\ \widetilde{y}^{\ell+1/2},
\end{equation*}
and, by Runge-Kutta integral scheme of order 2, we mean the scheme
\begin{equation*}
	\widetilde{y}^{\ell+1/2} = y^\ell + \left( \int_{t^\ell}^{t^{\ell} + \Delta t / 2} b_{t/\eps} \D t \right) y^\ell, 
	\qquad
	y^{\ell+1} = y^\ell + \left( \int_{t^\ell}^{t^{\ell+1}} b_{t/\eps} \D t \right) \widetilde{y}^{\ell+1/2}.
\end{equation*}
The idea is to exploit the form of the right-hand side $b_{t/\eps} y(t)$ and to build a scheme approximating $ \int_{t^\ell}^{t^{\ell+1}} b_{t/\eps} \Big( y(t^\ell) + \int_{t^\ell}^t \partial_t y(\sigma) \D \sigma \Big) \D t$ instead of directly $\int_{t^\ell}^{t^{\ell+1}} \partial_t y(t) \D t$. 
For a given quadrature rule, the first expression may generate a better approximation.

\begin{remark}
	As noted in \cite{Chartier2022}, the initial data $v\rk n(0) = \big( \Phi\rk n_0 \big)^{-1} u_0$ may be approximated explicitly to avoid the inversion of  $\Phi\rk n_0$. 
	It can be done such that the initial condition of the micro part $w\rk n(0) = u_0 - \Phi\rk n_0 v\rk n(0)$ becomes of size $\bigO(\eps^{n+1})$ which is enough to preserve the uniform accuracy result. 
\end{remark}

\section{Proofs}
\label{sec:proofs}

We now present proofs of our theoretical results, in the same order as they are presented. 
Namely, we start with the properties of the micro-macro decomposition, i.e. of the maps $\Phi\rk n$, $A\rk n$ and $\delta\rk n$ as enounced in Theorem~\ref{thm:decomp_mima}. 
We then focus on the well-posedness of the micro-macro problem $(v\rk n, w\rk n)$ from Corollary~\ref{corollary:mima_well_posed}. 
After this, we show the boundedness of the derivatives of the defect $\delta\rk n$, and in turn of $(v\rk n, w\rk n)$.
Finally, we use all this to prove our main result of uniform accuracy.

\subsection{Well-posedness of the micro-macro decomposition}
\label{sec:proof_wellposedness}

Since the fixed-point \eqref{eq:iter_Phi} is based on successive integrations, we shall use the following lemma, which bounds an antiderivative from the integrated function.

\begin{lemma}
    	\label{lemma:integration}
    	Let $\kappa_+ > 0$ and $\psi \in \calE_{\kappa_+}$. Then, for any $\kappa_-$ such that $0 \leq \kappa_- < \kappa_+$, solutions $\varphi$ to the equation
    	\begin{equation*}
        		\partial_\tau \varphi = \psi - \mean{\psi}
    	\end{equation*}
    	satisfy the inequality
    	\begin{equation*}
        		\calN_{\kappa_-}(\varphi - \mean{\varphi}) \leq c_I(\kappa_+ - \kappa_-) \calN_{\kappa_+}(\psi - \mean{\psi})
    	\end{equation*}
    	with 
    	\begin{equation}
    		\label{eq:cI}
    		c_I(\kappa) = \begin{cases}
    		\max \left\{1, \frac1{c_D} \left( \frac\nu{\kappa e} \right)^\nu \right\} & \text{ if } \nu \neq 0, \\
		\max \left\{1, \frac1{c_D} \right\} & \text{ if } \nu = 0.
		\end{cases}
    	\end{equation} 
\end{lemma}

For $\varphi\in\calE_0$, $\mean{\partial_\tau \varphi} = 0$. 
This implies that if $\varphi \in \calE_{\kappa_-}$ and $\partial_\tau \varphi \in \calE_{\kappa_+}$, then 
\begin{equation*}
    	\calN_{\kappa_-}(\varphi - \mean{\varphi}) \leq c_I(\kappa_+ - \kappa_-) \calN_{\kappa_+}(\partial_\tau \varphi).
\end{equation*}

We also present some estimates on the nonlinear operator $\Lambda$ occurring in \eqref{eq:def_Phin} and \eqref{eq:def_delta}.

\begin{lemma}[Bounds on $\Lambda$]
	\label{lemma:bound_Lambda}
	Let $0 \leq \kappa \leq \mu$ and $\varphi$, $\widetilde\varphi \in \calE_\kappa$. 
	Let $c \in ]0,1[$. 
	If
	\begin{equation*}
		\calN_\kappa(\varphi - \id) \leq c \text{ and } \calN_\kappa(\widetilde\varphi - \id) \leq c,
	\end{equation*}
	then there exists two constants $N_c \geq 2$ and $L_c \leq N_c/c$ depending on $c$ only, such that
	\begin{align*}
		\calN_\kappa(\Lambda\{\varphi\}) & \leq N_c M, \\
		\calN_\kappa(\Lambda\{\varphi\} - \Lambda\{\widetilde\varphi \}) & \leq L_c M \calN_\kappa(\varphi - \widetilde\varphi).
	\end{align*}
\end{lemma}

The proofs of Lemmas~\ref{lemma:integration} and \ref{lemma:bound_Lambda} are postponed in Appendices~\ref{app:integration_lemma} and \ref{app:bound_Lambda}. \\

We now proceed with the proof of Theorem~\ref{thm:decomp_mima}. 
Fix $n \in \setN$ and consider $0 \leq k \leq n$. 
Owing to Lemmas~\ref{lemma:integration} and \ref{lemma:bound_Lambda}, if $\Phi\rk k$ is in some space $\calE_{\kappa_+}$ such that $\calN_{\kappa_+}(\Phi\rk k - \id) \leq c$, then it possible to bound $\Phi\rk{k+1} - \id$ on the larger space $\calE_{\kappa_-}$ for all $0 \leq \kappa_- < \kappa_+$. 
Here, we proceed by induction, by considering successive $\Phi\rk k$ on spaces $\calE_{\mu_k}$ with
\begin{equation*}
	\mu_k = \left( 1 - \frac{k}{n+1} \right) \mu,
	\qquad\text{s.t.}\qquad
	0 = \mu_{n+1} < \mu_n < \ldots < \mu_1 < \mu_0 = \mu.
\end{equation*}
We show the desired bound on $\calN_{\mu_n}(\Phi\rk n - \id)$, which implies the well-posedness of the $n$-th order change of variable. 
Additionally, we bound the approximate map $A\rk n$ and we determine the size of the defect $\delta\rk n$. 

\subsubsection{Estimates on the near-identity and average maps}

We proceed by induction to show that, for all $0 \leq k \leq n+1$, $\calN_{\mu_k}(\Phi\rk k - \id) \leq c$. 
It is clear that it holds for $k = 0$ since $\Phi\rk0 = \id$. 
Now, for $0 \leq k \leq n$, we assume that $\Phi\rk k \in \calE_{\mu_k}$ and $\calN_{\mu_k}(\Phi\rk k - \id) \leq c$. 
Owing to the identity 
\begin{equation*}
	\partial_\tau \left[\Phi\rk{k+1} - \id\right] = \partial_\tau \Phi\rk{k+1} = \eps \Lambda\{ \Phi\rk k \},
\end{equation*}
as well as Lemma~\ref{lemma:integration} using $\kappa_- = \mu_{k+1}$ and $\kappa_+ = \mu_k$ and noticing that $\kappa_+ - \kappa_- = \mu_n$, we have $\Phi\rk{k+1} \in \calE_{\mu_{k+1}}$ and 
\begin{equation*}
	\calN_{\mu_{k+1}}(\Phi\rk{k+1} - \id) \leq c_I(\mu_n) \calN_{\mu_k}(\eps \Lambda\{ \Phi\rk k \}).
\end{equation*}
By Lemma~\ref{lemma:bound_Lambda} and the induction hypothesis,
\begin{equation*}
	\calN_{\mu_{k+1}}(\Phi\rk{k+1} - \id) \leq c \eps \left( c_I(\mu_n) \frac{N_c}c M \right).
\end{equation*}
Consequently, introducing 
\begin{equation}
	\label{eq:eps_n}
	\eps_n := \frac{c}{c_I(\mu_n) N_c M},
\end{equation}
then for all $\eps\leq\eps_n$,
\begin{equation*}
	\calN_{\mu_{k+1}}(\Phi\rk {k+1} - \id) \leq c \frac\eps{\eps_n} \leq c.
\end{equation*}
In particular for $k = n$, we find
\begin{equation*}
	\| \Phi\rk n - \id \| \leq \calN_0(\Phi\rk n - \id) \leq \calN_{\mu_n}(\Phi\rk n - \id) \leq \frac{c\eps}{\eps_n} \leq c.
\end{equation*}
Moreover, proceeding as in Lemma~\ref{lemma:bound_Lambda}, we obtain
\begin{equation*}
	| A\rk n | = | \mean{a \Phi\rk n} | \leq \calN_{\mu_n}(a \Phi\rk n) \leq \calN_{\mu_n}(a) \calN_{\mu_n}(\Phi\rk n) .
\end{equation*}
For one part $\calN_{\mu_n}(a) \leq M$ since $\mu_n \leq \mu$, and for the other part, 
\begin{equation*}
	\calN_{\mu_n}(\Phi\rk n) \leq \calN_{\mu_n}(\id) + \calN_{\mu_n}(\Phi\rk n - \id) \leq 1 + c.
\end{equation*}
This finally yields 
\begin{equation*}
	| A\rk n | \leq (1+c) M.
\end{equation*}

\subsubsection{Size of the defect}

By definition, for all $0 \leq k \leq n$, 
\begin{equation*}
	\delta\rk k  = \Lambda\{ \Phi\rk{k-1} \} - \Lambda\{ \Phi\rk k \} 
	= \frac1\eps \partial_\tau \big (\Phi\rk k -  \Phi\rk{k+1} \big)
\end{equation*}
and $\delta\rk k \in \calE_{\mu_k}$, with convention $\Lambda\{ \Phi\rk{-1} \} = 0$. 
The Lipschitz property on $\Lambda$ implies that
\begin{equation*}
	\calN_{\mu_k}(\delta\rk k) \leq L_c M \calN_{\mu_k}(\Phi\rk{k-1} - \Phi\rk k).
\end{equation*}
Since $\mean{\Phi\rk{k-1} - \Phi\rk k} = 0$, and thanks to Lemma~\ref{lemma:integration}
\begin{align*}
	\calN_{\mu_k}(\Phi\rk{k-1} - \Phi\rk k) 
	& \leq c_I(\mu_n) \calN_{\mu_{k-1}}(\partial_\tau (\Phi\rk{k-1} - \Phi\rk k)) \\
	& = c_I(\mu_n) \eps \calN_{\mu_{k-1}}(\Lambda\{ \Phi\rk{k-2} \} - \Lambda\{ \Phi\rk{k-1} \}).	
\end{align*}
We recognize the definition $\delta\rk{k-1} = \Lambda\{\Phi\rk{k-2}\} - \Lambda\{\Phi\rk{k-1}\}$, hence
\begin{equation*}
	\calN_{\mu_k}(\delta\rk k) \leq L_c M c_I(\mu_n) \eps \calN_{\mu_{k-1}}(\delta\rk{k-1}).
\end{equation*}
Using the expression of $\eps_n$ \eqref{eq:eps_n} and the bound $c L_c \leq N_c$, we finally obtain
\begin{equation*}
	\calN_{\mu_k}(\delta\rk k) \leq \frac{cL_c}{N_c} \frac\eps{\eps_n} \calN_{\mu_{k-1}}(\delta\rk{k-1})
	\leq \frac\eps{\eps_n} \calN_{\mu_{k-1}}(\delta\rk{k-1}).
\end{equation*}
An immediate induction, and the fact that $\mu_0 = \mu$ leads to
\begin{equation*}
	\calN_{\mu_n}(\delta\rk k) \leq \left( \frac\eps{\eps_n} \right)^k \calN_\mu(\delta\rk0).
\end{equation*}
Now $\delta\rk0 = - \Lambda\{ \id \} = - (a - \mean{a})$, hence $\calN_\mu(\delta\rk0) \leq \calN_\mu(a) \leq M$, and as such,
\begin{equation*}
	\calN_{\mu_n}(\delta\rk n) \leq M \left( \frac\eps{\eps_n} \right)^n.
\end{equation*}
Notice that a direct by-product of this proof is the bound between two consecutive near-identity maps   
\begin{equation}
	\label{eq:estim_diff_Phi_new}
	\calN_{\mu_{n+1}}(\Phi\rk n - \Phi\rk{n+1}) 
	\leq  c_I(\mu_n) \eps \calN_{\mu_n}(\delta\rk n) 
	\leq c_I(\mu_n) \eps M \left( \frac\eps{\eps_n} \right)^n 
	\leq \frac{c}{N_c} \left( \frac\eps{\eps_n} \right)^{n+1}.
\end{equation}
A direct integration yields 
\begin{equation*}
	\eps \int_0^\tau \delta\rk n_\sigma \D\sigma = \Phi\rk n_\tau - \Phi\rk n_0 - \Phi\rk{n+1}_\tau + \Phi\rk{n+1}_0,
\end{equation*}
from which \eqref{eq:estim_diff_Phi_new} may be plugged to find
\begin{equation*}
	\sup_{\tau \geq 0} \left| \eps \int_0^\tau \delta\rk n_\sigma \D\sigma \right| 
	\leq 2 \| \Phi\rk n - \Phi\rk{n+1} \| 
	\leq 2 \calN_{\mu_{n+1}}(\Phi\rk n - \Phi\rk{n+1}) 
	\leq \frac{2c}{N_c} \left( \frac\eps{\eps_n} \right)^{n+1}.
\end{equation*}
Thanks to the bounds $N_c \geq 2$ and $c \leq 1$, we finally obtain the desired result.

\subsection{Well-posedness of the micro-macro problem}
\label{sec:proof_corollary_mima}

\begin{proof}[Proof of Corollary~\ref{corollary:mima_well_posed}]
	By boundedness of $A\rk n$ due to the estimates \eqref{prop:asympt:eq:A}, the macro part $v\rk n$ is well-defined. 
	Writing $v\rk n(t) = v\rk n(0) + \int_0^t A\rk n v\rk n(t') \D t'$, a direct application of Gronwall's lemma and the estimate \eqref{eq:bound_inv_Phi} on $(\Phi \rk n_0)^{-1}$, we obtain
	\begin{equation}
		\label{eq:bound_vn}
		|v\rk n(t) | \leq e^{(1+c)t M} \frac{|u_0|}{1 - c}.
	\end{equation}
	For the micro part, we use the integral formulation to obtain 
	\begin{equation*}
		| w\rk n(t) | \leq \|a\| \int_0^t | w\rk n(t') | \D t' + \left| \int_0^t \delta \rk n_{t'/\eps} v\rk n(t') \D t' \right|.
	\end{equation*}
	A new application of Gronwall's lemma and the fact that  $w\rk n(0) = 0$ generates 
	\begin{equation*}
		| w\rk n(t) | \leq e^{t\| a \|} \sup_{t' \in [0,t]} \left| \int_0^{t'} \delta\rk n _{t''/\eps} v\rk n(t'') \D t'' \right|.
	\end{equation*}
	Using again the integral expression of $v\rk n$, we may integrate by parts to obtain 
	\begin{align*}
		\int_0^t \delta\rk n_{t'/\eps} v\rk n(t') \D t'
		= \left( \int_0^t \delta\rk n_{t'/\eps} \D t' \right) v\rk n(0)
		- \int_0^t \left( \int_0^{t'} \delta\rk n_{t''/\eps} \D t'' \right) A\rk n v\rk n(t') \D t' \\ 
		+ \left( \int_0^t \delta\rk n_{t'/\eps} \D t' \right) \left( \int_0^t A\rk n v\rk n(t') \D t' \right).
	\end{align*}
	Applying a change of variable $\sigma \leftarrow t'/\eps$ or $\sigma' \leftarrow t'' / \eps$ in the integrals of $\delta\rk n$ followed by a direct injection of the estimates from Theorem~\ref{thm:decomp_mima} and of the estimate \eqref{eq:bound_vn}, we obtain
	\begin{align*}
		\Big| \int_0^t \delta\rk n_{t'/\eps} v\rk n(t') \D t' \Big| 
		& = \left( \frac\eps{\eps_n} \right)^{n+1} \left( 1 + 2(1+c) M \int_0^t e^{(1+c)t' M} \D t' \right) \frac{|u_0|}{1-c} \\ 
		& \leq (2 e^{(1+c)t M} - 1) \left( \frac\eps{\eps_n} \right)^{n+1} \frac{|u_0|}{1-c}.
	\end{align*}
	Finally,
	\begin{equation*}	
		| w\rk n (t) | \leq \left( \frac\eps{\eps_n} \right)^{n+1} e^{t\|a\|} \left( 2e^{(1+c)t M} - 1\right) \frac{|u_0|}{1-c}. 	
	\end{equation*}
	Thus, for any $T > 0$, the micro-macro problem~\eqref{eq:mima} is uniformly bounded on $[0,T]$ with $v\rk n$ of size $\bigO(1)$ and $w\rk n$ of size $\bigO(\eps^{n+1})$. 
\end{proof}

\subsection{Well-posedness of the derivatives} \label{sec:derivatives}

Similar to Lemma \ref{lemma:bound_Lambda}, we start by presenting some estimates on the derivatives of the nonlinear operator $\Lambda$.

\begin{lemma}[Bounds on derivatives of $\Lambda$]
	\label{lemma:bound_der_Lambda}
	Under the assumptions of Lemma~\ref{lemma:bound_Lambda} and Assumption \ref{hyp:derivatives}, there exists a constant $N^{(q)}_c$ (depending on $q$, $c$ and $C_a^{(q)}$) such that
	\begin{equation*}
		\sup_{0 \leq p \leq q} \| \partial_\tau^p \Lambda\{ \varphi \} \|  
		\leq N^{(q)}_c M \sup_{0 \leq p \leq q} \| \partial_\tau^p \varphi \|.
	\end{equation*}
	Moreover, if there exists $c^{(q)} > 0$ such that
	\begin{equation*}
		\sup_{0 \leq p \leq q} \| \partial_\tau^p \varphi \| \leq c^{(q)} 
		\text{ and } 
		\sup_{0 \leq p \leq q} \| \partial_\tau^p \tilde\varphi \| \leq c^{(q)},
	\end{equation*}	
	then there exists a constant $L^{(q)}_c$  such that
	\begin{equation*}
		\sup_{0 \leq p \leq q} \| \partial_\tau^p \big( \Lambda\{ \varphi \} - \Lambda\{ \widetilde\varphi \} \big) \|  
		\leq L^{(q)}_c M \sup_{0 \leq p \leq q} \| \partial_\tau^p \big( \varphi - \widetilde\varphi \big) \|.
	\end{equation*}	
\end{lemma}
The proof of this lemma is postponed in Appendix~\ref{app:bound_der_Lambda}. \\

We now want to establish that there exists $C_\delta^{(q)} > 0$ such that
\begin{equation*}
	\label{bound_der_delta}
	\sup_{0 \leq p \leq q} \| \partial_\tau^p \delta\rk n \| \leq C_\delta^{(q)} M \left( \frac\eps{\eps_n}\right)^n.
\end{equation*}
We have already estimated $\| \delta\rk n \|$ in \eqref{eq:bound_delta}. 
The proof for larger values of $p$ follows the same lines, using in addition Lemma~\ref{lemma:bound_der_Lambda}. 

We first need to bound $\| \partial_\tau^p \Phi\rk k \|$ for all $0 \leq p \leq q$ and for all $0 \leq k \leq n$. 
The bound for $p = 0$ is clear from \eqref{prop:asympt:eq:Phi}. 
Under the assumption that $a \in C^q$, it is easy to see from the definition of $\Lambda$ and the definition of $\Phi\rk k$ that $\Phi\rk k \in C^q$ for all $0 \leq k \leq n$. 
Owing to the identity \eqref{eq:def_Phin} and Lemma~\ref{lemma:bound_der_Lambda}, we have, for $0 < p' \leq q$,
\begin{equation*}
	\| \partial_\tau^{p'} \Phi\rk{k} \| = \eps \| \partial_\tau^{p'-1} \Lambda\{ \Phi\rk{k-1} \} \| 
	\leq \eps \sup_{0 \leq p \leq q} \| \partial_\tau^p \Lambda\{ \Phi\rk{k-1} \} \| 
	\leq \eps N_c^{(q)} M \sup_{0 \leq p \leq q} \| \partial_\tau^p \Phi\rk{k-1} \|.
\end{equation*}
Consequently, since $\Phi\rk0 = \id$, a straightforward induction gives that there exists $c^{(q)}$ such that, for all $0 \leq k \leq n$, 
\begin{equation*}
\sup_{0 \leq p \leq q} \| \partial_\tau^p \Phi\rk k \| \leq c^{(q)}.
\end{equation*}

Let us now turn to the estimation of the defect $\delta\rk k = \Lambda\{ \Phi\rk{k-1} \} - \Lambda\{ \Phi\rk k \}$ for all $0 \leq k \leq n$. 
By the previous bound on $\| \partial_\tau^p \Phi\rk k \|$ and Lemma~\ref{lemma:bound_der_Lambda}, 
\begin{equation*}
	\sup_{0 \leq p \leq q} \| \partial_\tau^p \delta\rk k \| \leq L_c^{(q)} M \sup_{0 \leq p \leq q} \| \partial_\tau^p (\Phi\rk{k-1} - \Phi\rk k) \|.
\end{equation*}
For $0 < p \leq q$, according to \eqref{eq:def_Phin},
\begin{equation*}
	\| \partial_\tau^p (\Phi\rk{k-1} - \Phi\rk k) \| 
	= \eps  \| \partial_\tau^{p-1} (\Lambda\{ \Phi\rk{k-2} \} - \Lambda\{ \Phi\rk{k-1} \}) \| 
	= \eps  \| \partial_\tau^{p-1} \delta\rk{k-1} \|.
\end{equation*}
Thus, we obtain
\begin{equation*}
	\sup_{0 \leq p \leq q} \| \partial_\tau^p \delta\rk k \| 
	\leq L_c^{(q)} M \eps \sup_{0 \leq p \leq q} \| \partial_\tau^p \delta\rk{k-1} \|
\end{equation*}
and an immediate induction leads to
\begin{equation*}
	\sup_{0 \leq p \leq q} \| \partial_\tau^p \delta\rk n\| 
	\leq \left( L_c^{(q)} M \eps \right)^n \sup_{0 \leq p \leq q} \| \partial_\tau^p \delta\rk0 \|.
\end{equation*}
Since $\delta\rk0 = - \Lambda\{ \id \} = - (a - \mean{a})$, hence 
\begin{equation*}
\sup_{0 \leq p\leq q} \| \partial_\tau^p \delta\rk0 \| \leq \sup_{0 \leq p \leq q} \| \partial_\tau^p a \| \leq C_a^{(q)} M,
\end{equation*}
we obtain the desired bound \eqref{bound_der_delta} denoting $C_\delta^{(q)} = C_a^{(q)} (\eps_n L_c^{(q)} M)^n$. 
It concludes the proof of Theorem~\ref{thm:bound_derivatives_delta}.
 
\subsection{Uniform accuracy}

\begin{proof}[Proof of Corollary~\ref{corollary:mima_bounded_derivatives}]
	For the derivatives of $v\rk n$, we perform an induction which is initialized using that $v\rk n$ is of size $\bigO(1)$ according to Corollary~\ref{corollary:mima_well_posed}. 
	Simply using 
	\begin{equation*}
		\partial_t^{p+1} v\rk n(t) = A\rk n \partial_t^p v\rk n(t), 
	\end{equation*}
	and the estimate \eqref{prop:asympt:eq:A} on $A\rk n$, we obtain, for all $p = 0, \dots, \min(n,q)+1$,
	\begin{equation}
		\label{eq:induction_derivatives_mima}
		\partial_t^p v\rk n(t) = \bigO(1).
	\end{equation}

	For the derivatives of $w\rk n$, the induction is initialized by Corollary~\ref{corollary:mima_well_posed} from which $w\rk n (t) = \bigO(\eps^{n+1})$. 
	Then, we write 
	\begin{align*}
		\frac{\partial_t^{p+1} w\rk n(t)}{\eps^{n+1-(p+1)}} 
		= \eps \sum_{p' = 0}^p \binom{p}{p'} & \partial_\tau^{p-p'} a_{t/\eps} \frac{\partial_t^{p'} w\rk n(t)}{\eps^{n+1-p'}} \\
		& - \sum_{p'=0}^p \binom{p}{p'} \eps^{p-p'} \frac{\partial_\tau^{p'} \delta\rk n_{t/\eps}}{\eps^n} \partial_t^{p-p'} v\rk n(t).
	\end{align*}
	We bound the derivatives of $a_\tau$ thanks to Assumption~\ref{hyp:derivatives}, the derivatives of $\delta\rk n$ by $\bigO(\eps^n)$ thanks to Theorem~\ref{thm:bound_derivatives_delta} and the derivatives of $v\rk n$ using \eqref{eq:induction_derivatives_mima}. 
	Thus, by induction hypothesis, $\partial_t^{p'} w\rk n(t) = \bigO(\eps^{n+1 - p'})$ for all $p' = 0, \dots, p$, every term of the sums of the right-hand side is uniformly bounded w.r.t. $\eps$, and therefore the sum is of size $\bigO(1)$, for all $p = 0, \dots, \min(n,q)+1$. 
	This concludes the induction.
\end{proof}

\begin{proof}[Proof of Corollary~\ref{corollary:mima_ua}]
	The micro-macro problem can be written as 
	\begin{equation*}
		\partial_t y(t) = f(t/\eps, y(t)),
	\end{equation*}
	with $y = (v\rk n, w\rk n)$ and $f(\tau, y(t)) = \begin{pmatrix} A\rk n & 0 \\ -\delta\rk n_\tau & a_\tau \end{pmatrix} y(t)$. 
	We use a one-step scheme of non-stiff order $s$, written in the standard form
	\begin{equation*}
		y^{\ell+1} = y^\ell + \Delta t \mathcal{F}(t^\ell/\eps, y^\ell, \Delta t)
	\end{equation*}
	such that the $s$-th order derivative of $\mathcal{F}$ with respect to the third variable is of the same order as $\partial_t^{s+1} y$ (as it is for instance the case for standard and integral Runge-Kutta schemes). 
	It is well-known that the local consistency error $e_\ell = y(t^{\ell+1}) - y(t^\ell) - \Delta t \mathcal{F}(t^\ell, y(t^\ell), \Delta t)$ is bounded by
	\begin{equation*}
		| e_\ell | \leq \left( \frac1{(s+1)!} \sup_{t \in [t^\ell, t^{\ell+1}]} | \partial_t^{s+1} y(t) | + \frac1{s!} \sup_{h \in [0, \Delta t]} | \partial_h^s \mathcal{F}(t^\ell/\eps, y(t^\ell), h) | \right) \Delta t^{s+1}.
	\end{equation*}
	Thus, by Corollary~\ref{corollary:mima_bounded_derivatives}, the scheme retains its usual order $s$ as soon as $s \leq \min(n,q)$. 
\end{proof}

\section{Numerical experiments}
\label{sec:num}

In this section, we present some numerical experiments to illustrate the previous results. 
After a brief introduction of the different schemes we use, we test our strategy on a simple scalar problem for which we know the exact solution. 
This is used to illustrate the different components of the method, e.g. the size of the micro part $w\rk n$ and of its derivatives, and to validate our result of uniform accuracy. 
We then apply the micro-macro method to an approximation of the Bloch equations.

\subsection{Numerical schemes}
\label{sec:num_scheme}

We recall that we uniformly discretize the time interval $[0,T]$ defining $t^\ell = \ell \Delta t$ for $\ell = 0, \dots, L$ where $L+1$ is the number of discretization points and $\Delta t = \frac TL$ the time step. 
In the numerical experiments, we shall use the following numerical schemes  associated to \eqref{eq:mima}:
\begin{itemize}
\item[--] {\it Explicit Euler (EE) scheme:}
\begin{equation*}
	\left \{
	\begin{aligned}
	v^{\ell+1} &= v^\ell + \Delta t  A v^\ell, \\[1mm]
	w^{\ell+1} & = w^\ell + \Delta t (a^\ell w^\ell - \delta^\ell v^\ell),
	\end{aligned}
	\right. 
\end{equation*}
\item[--] {\it Integral Explicit Euler (EEint) scheme:}
\begin{equation*}
	\left \{
	\begin{aligned}
	v^{\ell+1} & = v^\ell +\Delta t  A v^\ell, \\[1mm]
	w^{\ell+1} & = \textstyle w^\ell +  \big( \int_{t^\ell}^{t^{\ell+1}} a_{\sigma/ \eps} \D\sigma \big) w^\ell - \big( \int_{t^\ell}^{t^{\ell+1}} \delta\rk n_{\sigma/\eps}  \D\sigma \big)  v^\ell,
	\end{aligned}
	\right. 
\end{equation*}
\item[--] {\it Runge-Kutta of order 2 (RK2) scheme:}\\
\begin{equation*}
  	\left\{ 
	\begin{aligned}
    	\widetilde{v}^{\ell+1/2} & = \textstyle v^\ell + \frac{\Delta t}2 A v^\ell , \\
    	\widetilde{w}^{\ell+1/2} & = \textstyle w^\ell + \frac{\Delta t}2 (a^\ell w^\ell - \delta^\ell v^\ell), \\[2mm]
    	v^{\ell+1} & = v^\ell + \Delta t A \widetilde{v}^{\ell+1/2} , \\
    	w^{\ell+1} & = w^\ell + \Delta t \left( a^{\ell+1/2} \widetilde{w}^{\ell+1/2} - \delta^{\ell+1/2} \widetilde{v}^{\ell+1/2} \right), \\
 	 \end{aligned} 
	 \right.
\end{equation*}
\item[--] {\it Integral Runge-Kutta (RK2int) scheme:}\\
\begin{equation*}
  	\left\{ 
	\begin{aligned}
    	\widetilde{v}^{\ell+1/2} & = \textstyle v^\ell + \frac{\Delta t}2 A v^\ell, \\
    	\widetilde{w}^{\ell+1/2} & = \textstyle w^\ell + \big( \int_{t^\ell}^{t^\ell + \Delta t/2} a_{\sigma/\eps} \D\sigma \big) w^\ell 
	- \big( \int_{t^\ell}^{t^\ell + \Delta t/2} \delta\rk n_{\sigma/\eps} \D\sigma \big) v^\ell, \\[2mm]
    	v^{\ell+1} & = v^\ell + \Delta t A \widetilde{v}^{\ell+1/2}, \\
    	w^{\ell+1} & = \textstyle w^\ell + \big( \int_{t^\ell}^{t^{\ell + 1}} a_{\sigma/\eps} \D\sigma \big) \widetilde{w}^{\ell+1/2} 
	- \big( \int_{t^\ell}^{t^{\ell + 1}} \delta\rk n_{\sigma/\eps} \D\sigma \big) \widetilde{v}^{\ell+1/2},
  	\end{aligned} 
	\right.
\end{equation*}
\end{itemize}
where $(v^\ell, w^\ell)$ and $(v^{\ell+1}, w^{\ell+1})$ are respectively approximations of $\big( v\rk n(t^\ell), w\rk n(t^\ell) \big)$ and $\big(v\rk n(t^{\ell+1}), w\rk n(t^{\ell+1}) \big)$ and $A$, $a^\ell$, $a^{\ell+1/2}$, $a^{\ell+1}$, $\delta^\ell$, $\delta^{\ell+1/2}$, $\delta^{\ell+1}$ denote respectively $A\rk n$, $a_{t^\ell/\eps}$, $a_{(t^\ell + \Delta t/2)/\eps}$, $a_{t^{\ell+1}/\eps}$, $\delta\rk n_{t^\ell/\eps}$, $\delta\rk n_{(t^\ell+\Delta t/2)/\eps}$, $\delta\rk n_{t^{\ell+1}/\eps}$.

\subsection{A scalar test problem}

\subsubsection{Presentation of the problem}

In this part, we study from  a numerical point of view the problem \eqref{eq:ode_u} with a ``sharp-flat'' scalar map defined by $a_{\tau} = a^\qp_{\omega \tau} + a^\dmp_\tau$. 
We choose it such that the equation roughly behaves as the applicative problem considered in the next section with the advantage of knowing an exact solution. 
More precisely, we consider a quasi-periodic part of the form
\begin{equation*}
	a^\qp_{\omega \tau} = -1 + b^\qp_{\omega \tau}, 
	\text{ with } 
	b^\qp_{\omega \tau} = \frac1r \sum_{p=1}^r \cos(\omega_p \tau),
\end{equation*}
and an exponentially decreasing part of the form
\begin{equation*}
	a^\dmp_\tau = \gamma e^{-\tau},
\end{equation*}
where $\gamma$ is a given constant. 
In this case, the exact solution of \eqref{eq:ode_u} is given by
\begin{equation*}
	u^\eps(t) = u_0 e^{-t + \eps \big( B^\qp_{\omega t/\eps} + a^\dmp_0 - a^\dmp_{t/\eps} \big)},   
	\text{ with } 
	B^\qp_{\omega \tau} = \int_0^\tau b^\qp_{\omega \sigma} \D\sigma 
	= \frac1r \sum_{p=1}^r \frac{\sin(\omega_p \tau)}{\omega_p},
\end{equation*}
and it tends, as $\eps$ goes to $0$, to
\begin{equation*}
	u^{\rm lim}(t) = u_0 e^{-t}.
\end{equation*}

In the sequel, we consider two choices of frequencies:
\begin{itemize}
\item Mono-frequency case (1F): $r = 1$ with $\omega_1 = \pi$,
\item Multi-frequency case (3F): $r = 3$ with $\omega_1 = 1$, $\omega_2 = \pi$ and $\omega_3 = \sqrt5 \pi$.
\end{itemize}
We also take the following data : $u_0 = 1$, $T = 10$. 

The time evolution of the solution is presented in Fig.~\ref{Fig:toy_sol_ex} for the problem with no exponential decay ($\gamma = 0$), where the plot is restricted to short times ($t \in [0,2]$) and where $\eps$ is fairly large, for visualization purposes. 
In both cases 1F (left) and 3F (right), the solution oscillates around the limit behavior $u^{\rm lim}$ represented by a solid black line.
These oscillations are of small amplitude and seem essentially proportional to $\eps$, and while their quasi-periodic nature makes them seem chaotic in the case 3F, the period $2\pi \eps/\omega_1 = 2 \eps$ appears clearly in the case 1F.

\begin{figure}[h!]
\begin{center}
\includegraphics[scale=0.18, keepaspectratio=true]{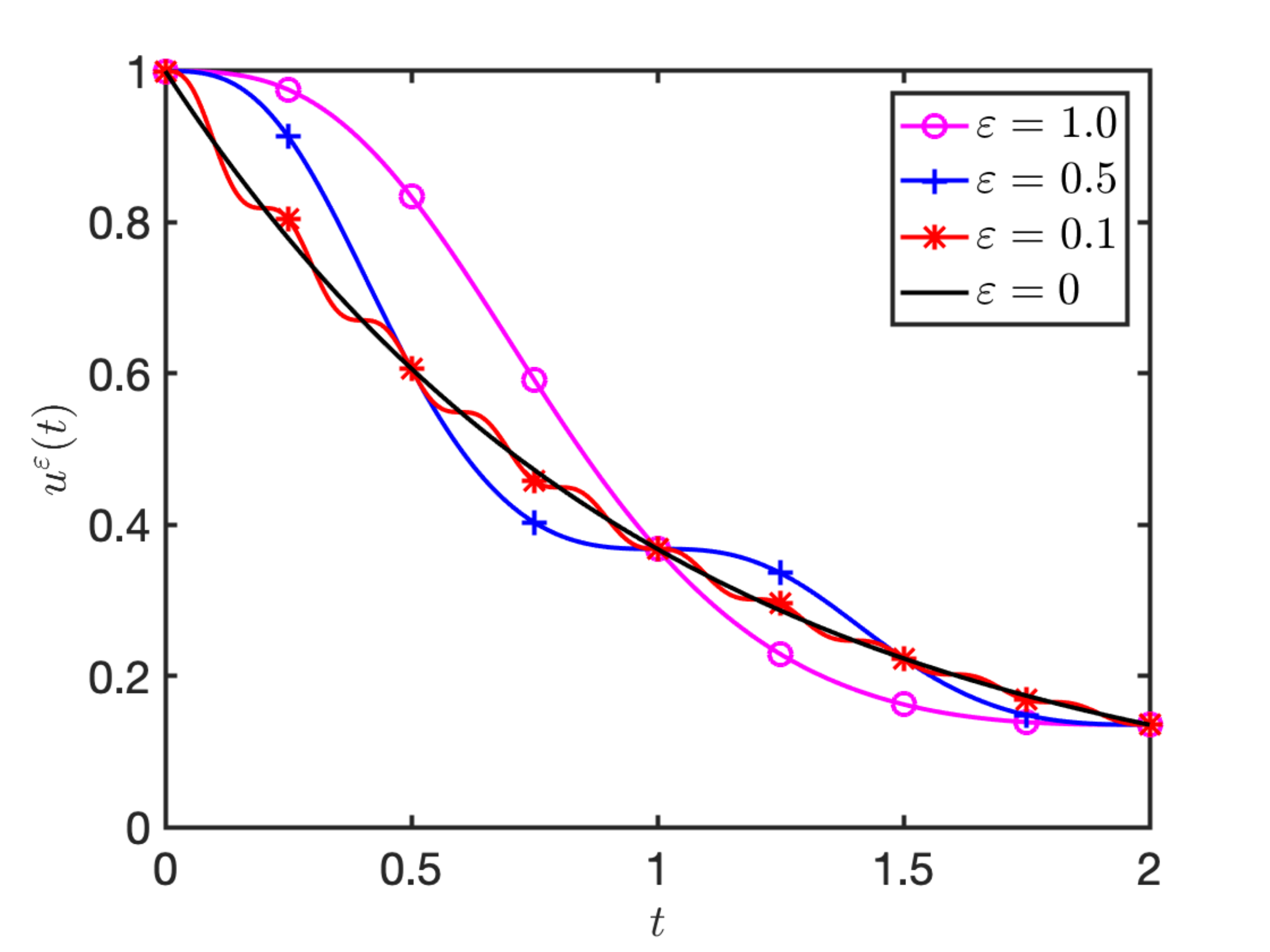}\hspace{0.5cm}
\includegraphics[scale=0.18, keepaspectratio=true]{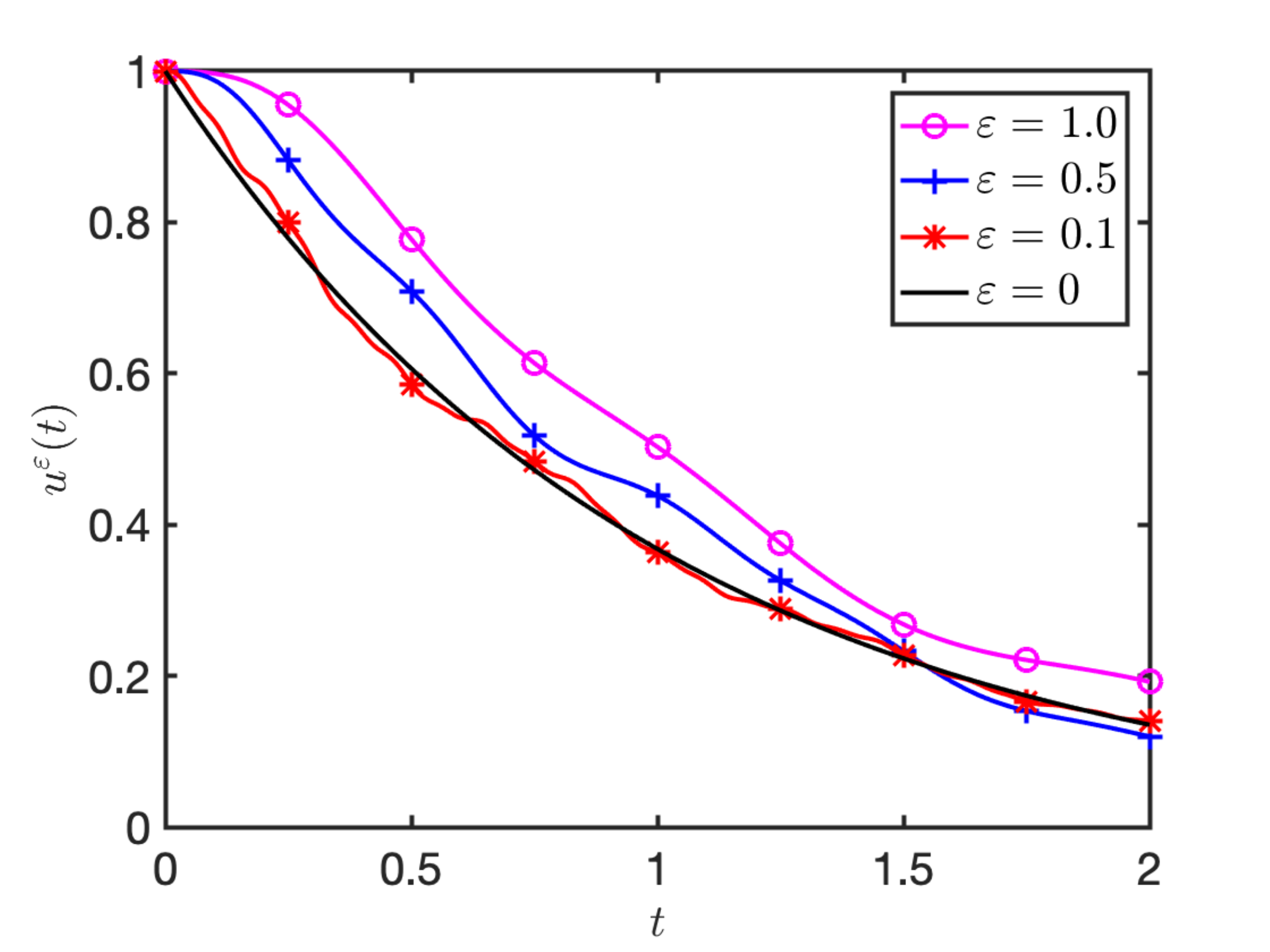}
\end{center}
\caption{Case 1F (left) and case 3F (right); time evolution of the exact solution $u^\eps$ for various $\eps$ and of the exact limit solution $u^{\rm lim}$.}
\label{Fig:toy_sol_ex}
\end{figure}

\subsubsection{Micro-macro problems}

Now, we write explicitly the different terms occurring in the micro-macro problem \eqref{eq:mima}  for $n = 1$ and $n = 2$, using the iterative relations \eqref{eq:iter_Phi} and the defect expression \eqref{eq:def_delta}.

\paragraph{Micro-Macro problem of order $1$:\\}

Since $A\rk0 = -1$ and $\Lambda\{ \Phi\rk0 \}_\tau = b^\qp_{\omega \tau} + a^\dmp_\tau$, straightforward computations yield 
\begin{equation*}
	\Phi\rk1_\tau = 1 + \eps C_\tau\rk1,
	\hspace{1cm} 	
	A\rk1 = -1
	\hspace{0.5cm} \text{ and } \hspace{0.5cm} 
 	\delta\rk1_\tau = - \eps (b^\qp_{\omega \tau} + a^\dmp_\tau) C_\tau\rk1
\end{equation*}
where we introduced
\begin{equation*}
	C_\tau\rk1 = B^\qp_{\omega \tau} - a^\dmp_\tau.
\end{equation*}
Consequently, at the first order, the solution is decomposed as $\Phi\rk1_{t/\eps} v\rk1(t) + w\rk1(t)$ where the micro-macro variables $(v\rk1, w\rk1)$ are solutions to the following problem
\begin{equation*}
	\left\{
	\begin{aligned}
 	\partial_t v\rk1(t) & =  - v\rk1(t), && v\rk1(0) = \frac{u_0}{1 + \eps C\rk1_0}, \\[1mm]
	\partial_t w\rk1(t) & =   a_\tau w\rk1(t) + \eps (b^\qp_{\omega \tau} + a^\dmp_\tau) C\rk1_\tau v\rk1(t), && w\rk1(0) = 0.
 	\end{aligned}
	\right. 
\end{equation*}

\paragraph{Micro-Macro problem of order $2$:\\}

Going on with the iterative process, further computations give
\begin{equation*}
	\Phi\rk2_\tau = \Phi\rk1_\tau + \eps^2 C\rk2_\tau,
	\hspace{1cm}
	A\rk2 = -1,
	\hspace{0.5cm} \text{ and }\hspace{0.5cm} 
	\delta\rk2_\tau = - \eps^2 (b^\qp_{\omega \tau} + a^\dmp_\tau) C\rk2_\tau
\end{equation*}
where $C\rk2_\tau = {C\rk2_{\omega \tau}}^\qp + {C\rk2_\tau}^\dmp$ is given by 
\begin{align*}
	{C\rk2_{\omega \tau}}^\qp = \frac1{r^2} \Bigg( &\sum_{p = 1}^r \frac{\sin^2(\omega_p \tau) -1/2}{2\omega_p^2} \\
	& +  \sum_{p_1 = 1}^r \sum_{\underset{p_2 \neq p_1}{p_2 = 1}}^r  \frac{\omega_{p_1} \sin(\omega_{p_1} \tau)\sin(\omega_{p_2} \tau) 
	+ \omega_{p_2} \cos(\omega_{p_1} \tau) \cos(\omega_{p_2} \tau)}{w_{p_2} (w_{p_1}^2 - w_{p_2}^2)} \Bigg)
\end{align*}
and
\begin{equation*}
	{C\rk2_\tau}^\dmp = - a^\dmp_\tau B^\qp_{\omega \tau} + \frac12 (a^\dmp_\tau)^2.
\end{equation*}
Consequently, at the second order, the solution is decomposed as $\Phi\rk2_{t/\eps} v\rk2(t) + w\rk2(t)$ where the micro-macro variables $(v\rk2, w\rk 2)$ are solutions to the following problem
\begin{equation}
	\label{eq:toy_pb_order2}
	\left\{
	\begin{aligned}
 	\partial_t v\rk2(t) & =  - v\rk2(t), && v\rk2(0) = \frac{u_0}{1 + \eps C\rk1_0 + \eps^2 C\rk2_0},  \\
	\partial_t w\rk2(t) & =  a_\tau w\rk2(t) + \eps^2 (b^\qp_{\omega \tau} + a^\dmp_\tau) C\rk2_\tau v\rk2(t), && w\rk2(0) = 0.
 	\end{aligned}
	\right. 
\end{equation}
In this specific case, $A\rk2$ is exactly equal to $\mean{a}$ but this is not the case in general. 
It gives a macro variable $v\rk2$ that differs from the limit solution $u^{\rm lim}$ only via the perturbation in the initial data $\big( \Phi\rk2_0 \big)^{-1}(u_0)$. 
Concerning the equation governing the micro variable $w\rk2$, we clearly observe that the defect $\delta\rk2$ is of size $\bigO(\eps^2)$.\\

In Figs. \ref{Fig:toy_decomp_caseA} and  \ref{Fig:toy_decomp_caseB}, we plot the micro-macro quantities of order 1 ($n = 1$) for $\eps = 0.5$, still for the pure quasi-periodic problem ($\gamma = 0$). 
They correspond respectively to the cases 1F and 3F. 
These are computed with high precision, such that no issues of numerical accuracy are considered at the moment.

On the left plots, the macro variable $v\rk1$ corresponds exactly to the limit solution $u^{\rm lim}$ since, for $\gamma = 0$, $v\rk1(0) = u_0$. 
However, the addition of the near-identity map $\Phi\rk1$ allows to incorporate the fast oscillations and to get closer to the exact solution $u^\eps$ (blue plus-marked curves in Fig. \ref{Fig:toy_sol_ex}). 
On the right, we present the micro variable $w\rk1$ that retains the information contained in the remainder. 
As expected by Corollary~\ref{corollary:mima_well_posed}, it is of size $\bigO(\eps^2)$. 
In addition, its second derivative $\partial^2_t w$ is of size $\bigO(1)$ in accordance with Corollary~\ref{corollary:mima_bounded_derivatives}. 
It confirms that the micro-macro problem can be solved with a standard scheme. 
Again, the only difference between the cases 1F and 3F concerns the almost periodicity of oscillations, both for the near-identity map and the macro variable. 
For this reason, we focus only on the more generic multi-frequency case in the sequel. 

\begin{figure}[h!]
\begin{center}
\includegraphics[scale=0.18, keepaspectratio=true]{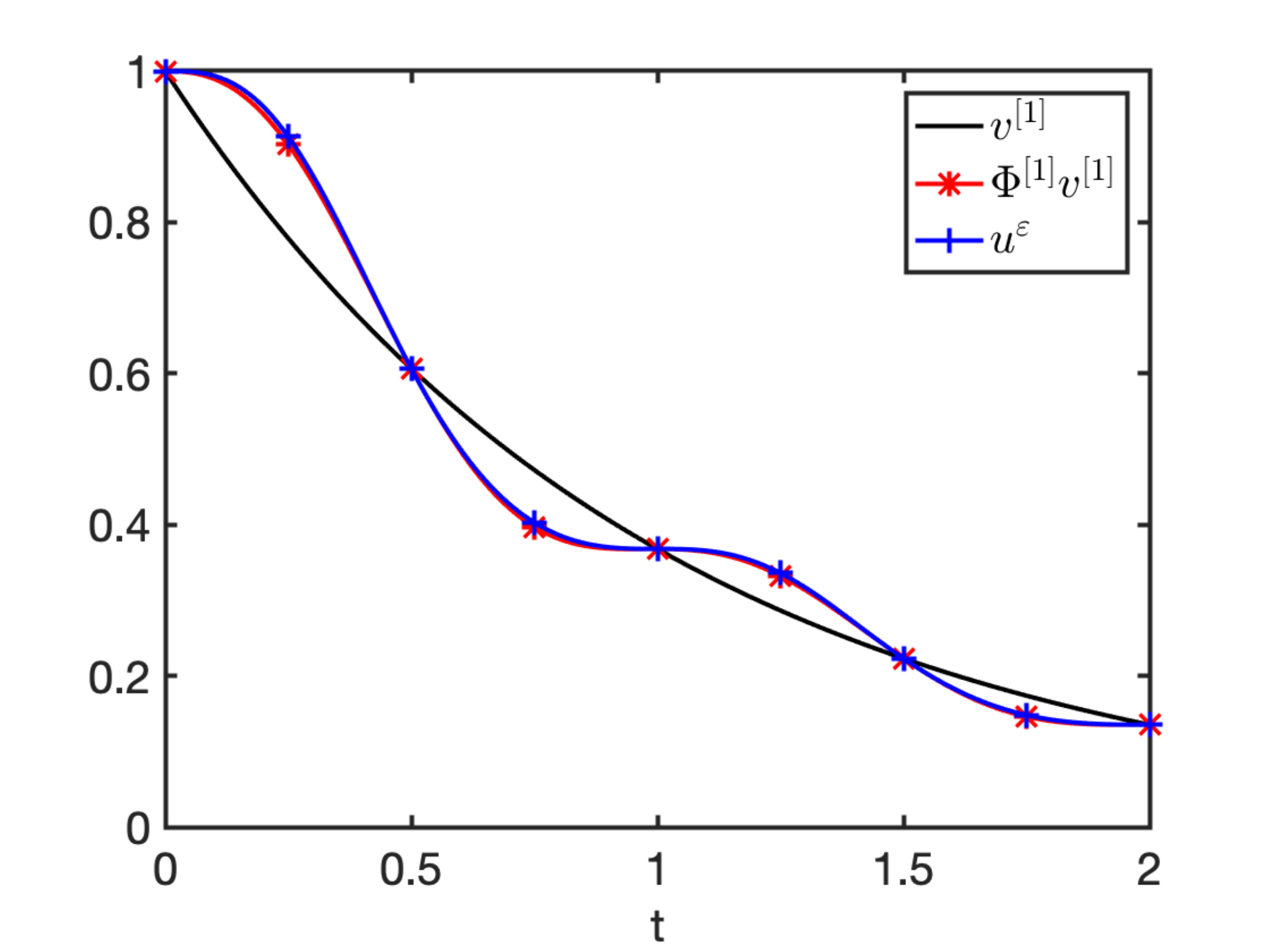}\hspace{0.5cm}
\includegraphics[scale=0.18, keepaspectratio=true]{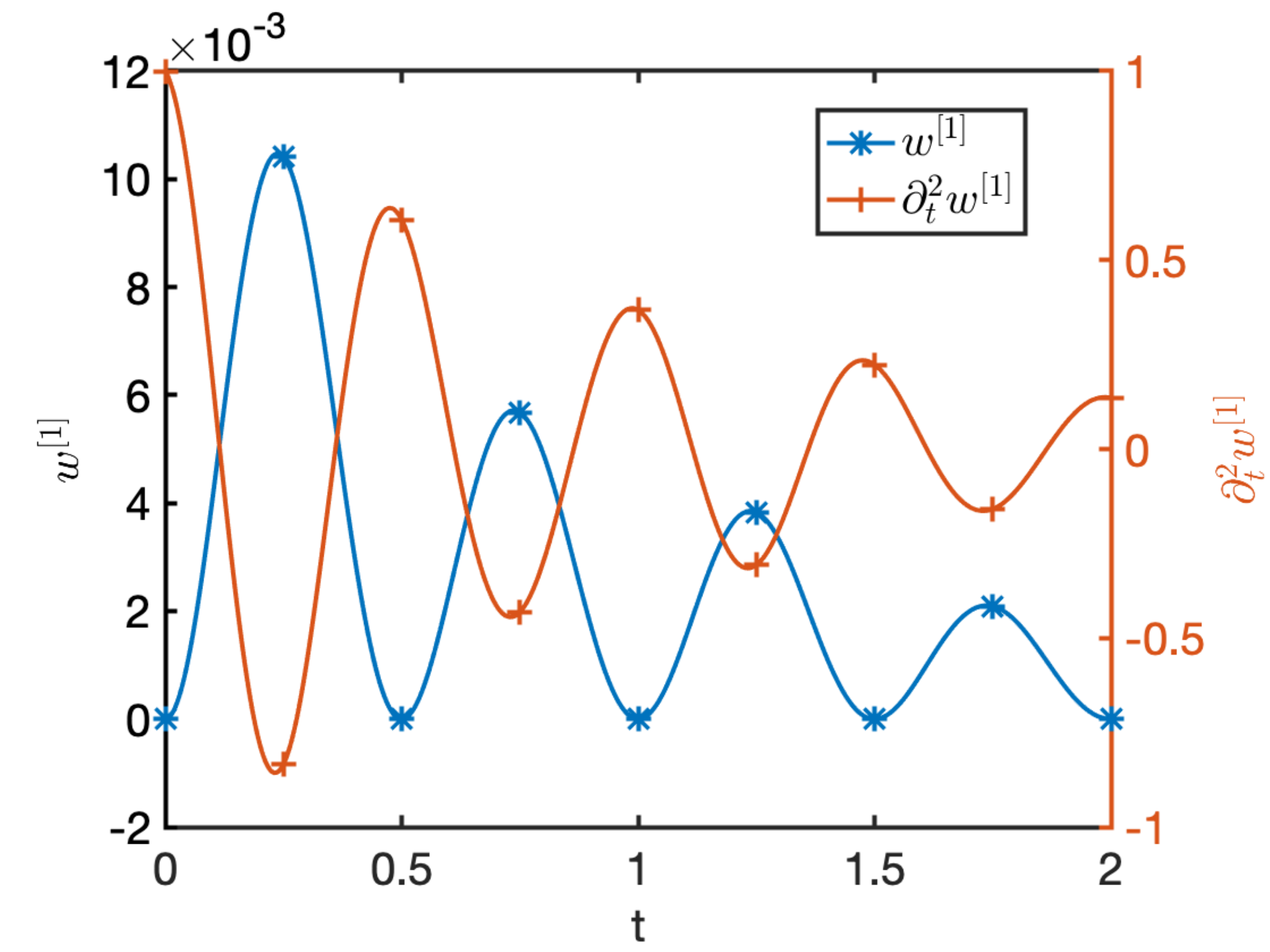}
\end{center}
\caption{Case 1F, $\eps = 0.5$; time evolution of $v\rk1$, $\Phi\rk1 v\rk1$ and $u^\eps$ (left) and of $w\rk1$ and an approximation of $\partial^2_t w\rk1$ (right).}
\label{Fig:toy_decomp_caseA}
\end{figure}

\begin{figure}[h!]
\begin{center}
\includegraphics[scale=0.18, keepaspectratio=true]{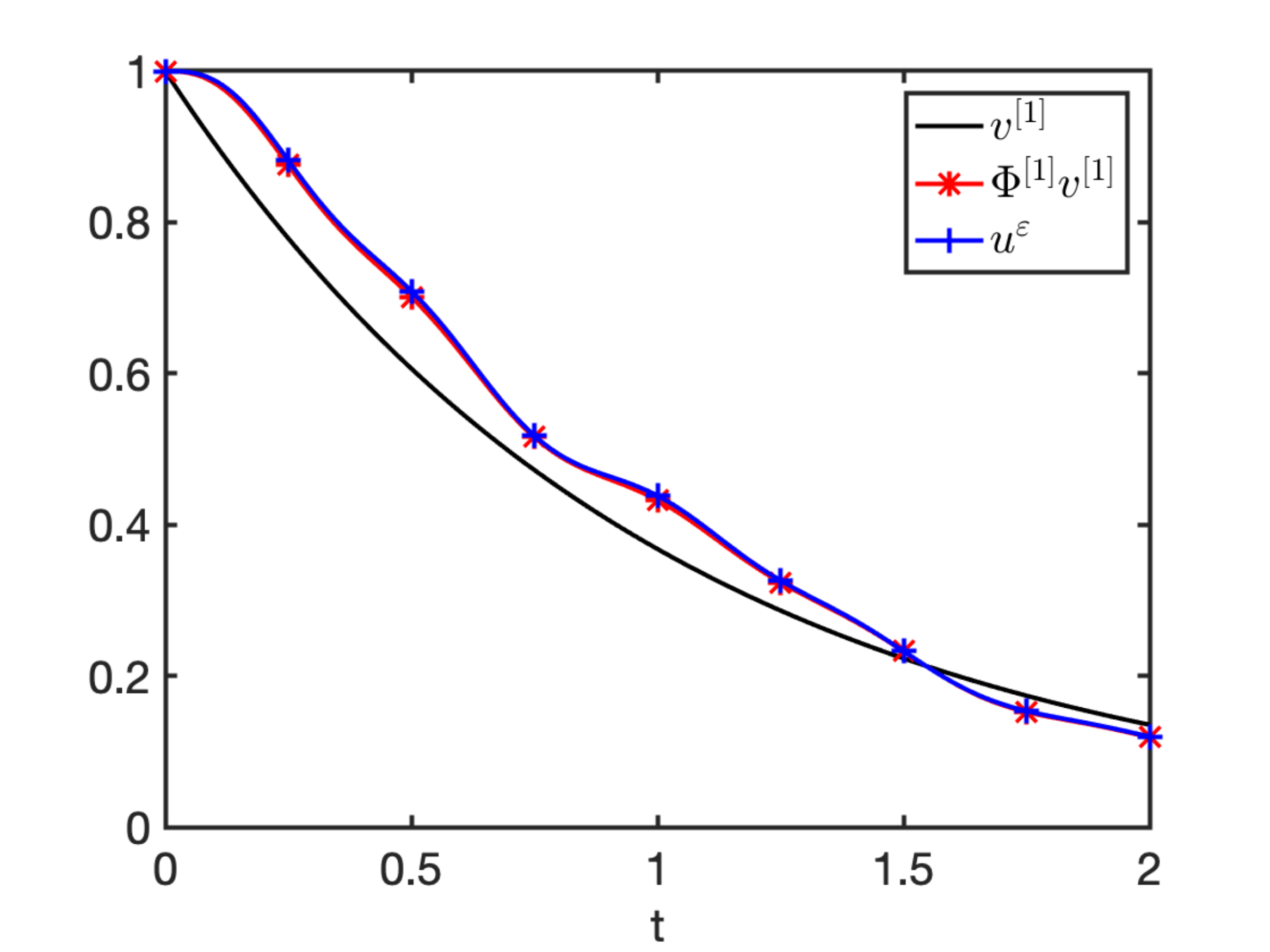}\hspace{0.5cm}
\includegraphics[scale=0.18, keepaspectratio=true]{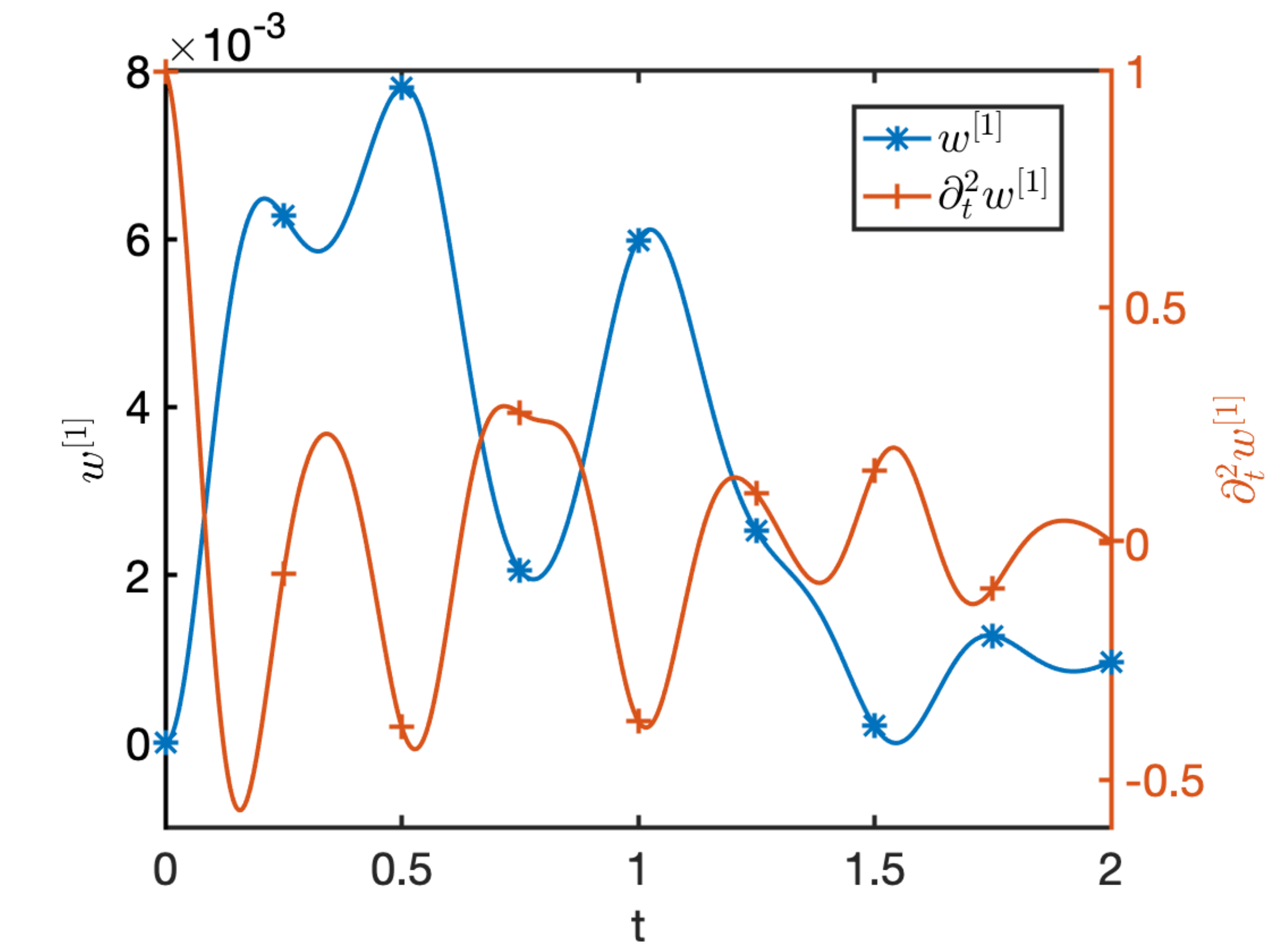}
\end{center}
\caption{Case 3F, $\eps = 0.5$; time evolution of $v\rk1$, $\Phi\rk1 v\rk1$ and $u^\eps$ (left) and of $w\rk1$ and an approximation of $\partial^2_t w\rk1$ (right).}
\label{Fig:toy_decomp_caseB}
\end{figure}

\subsubsection{Errors on the pure quasi-periodic problem ($\gamma = 0$)}

We now analyze the numerical resolution of the micro-macro problem. 
First, we focus on the quasi-periodic case choosing $\gamma = 0$ (the addition of the exponentially decaying term is studied in the next section).

To evaluate the numerical solutions, we consider the following error:
\begin{equation*}
	E(\Delta t, \eps) = \max_{0 \leq \ell \leq L} | u^\eps(t^\ell) - u^\ell |,
\end{equation*}
where $u^\ell$ is the numerical solution either  solving the stiff problem \eqref{eq:ode_u} or reconstructed from the micro-macro problem \eqref{eq:mima}.

In Fig. \ref{Fig:toy_err_eps} and Fig. \ref{Fig:toy_err_eps_without_w}, we present the errors obtained using the standard RK2 scheme. 
The numerical resolution of the stiff problem \eqref{eq:ode_u} does not yield suitable results. 
Indeed, in Fig. \ref{Fig:toy_err_eps} (left), in the standard regime $\Delta t \ll \eps$, the error for a given $\Delta t$ increases as $\eps$ decreases. 
Exiting this regime, for smaller $\eps$,  the error becomes hard to predict. 
We observe pronounced peaks for some specific values of $\eps$. 
It is known that in the case 1F, the solution has a specific behavior when the time-step resonates with the frequency of the problem, i.e. when $\omega_1 \Delta t/\eps$ is a multiple of $2\pi$. 
For such a relation between $\Delta t$ and $\eps$, a standard scheme completely fails. 
The left-hand side of Fig.~\ref{Fig:toy_err_eps} demonstrates that this phenomenon still occurs in the quasi-periodic case, even without perfect resonances.
In the right-hand side, for large $\eps$, the error decreases with order 2 as expected, but when $\eps$ decreases, we observe an order reduction with slopes closer to order $1$. 
Even worse, for small $\eps$, there are some values of $\Delta t$ for which the error is of size $\bigO(1)$.

\begin{figure}[h!]
\begin{center}
\includegraphics[scale=0.18, keepaspectratio=true]{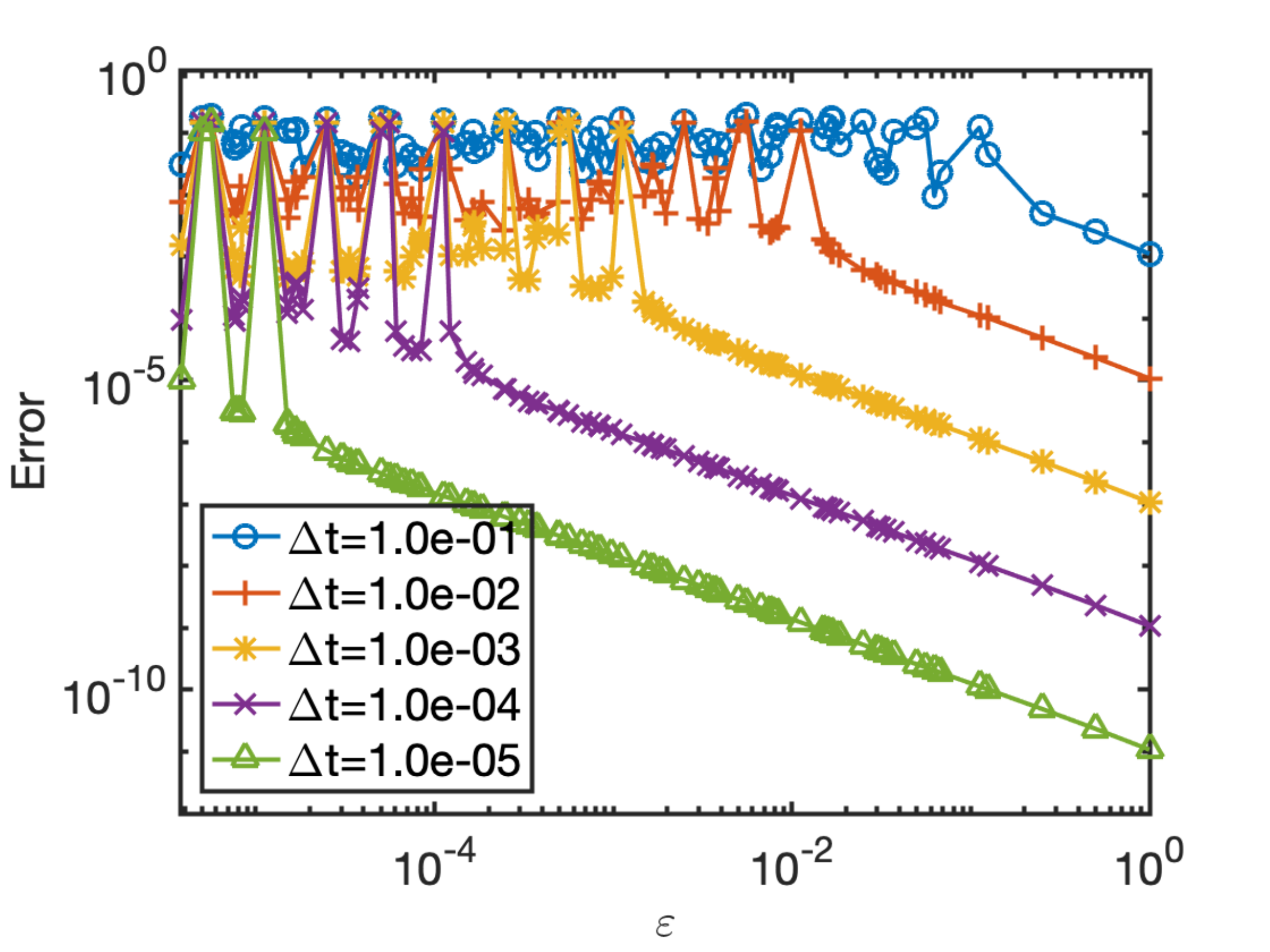}\hspace{0.5cm}
\includegraphics[scale=0.18, keepaspectratio=true]{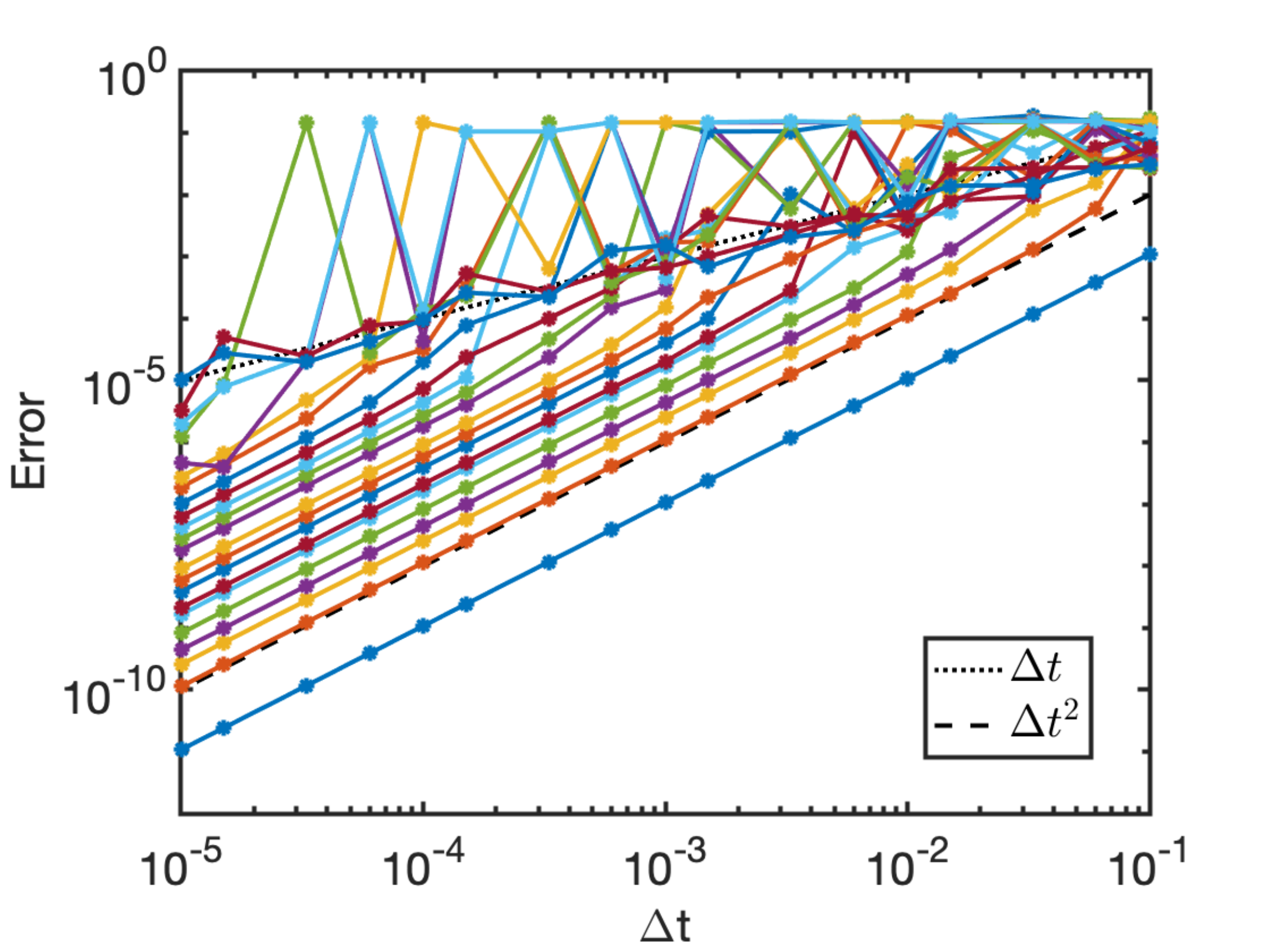}
\end{center}
\caption{Case 3F, RK2 scheme solving the stiff problem \eqref{eq:ode_u}; error with respect to $\eps$ for various $\Delta t$ (left) and with respect to $\Delta t$ for various $\eps$ (right).}
\label{Fig:toy_err_eps}
\end{figure}

On the contrary, the numerical resolution of the micro-macro problem of order 2~\eqref{eq:toy_pb_order2} gives a uniform accuracy i.e. independent of $\eps$ as observed in the left-hand side of Fig. \ref{Fig:toy_err_eps_without_w}. 
The errors associated to a given discretization step form a perfect horizontal line. 
Having computed the micro-macro variables $(v\rk2 , w\rk2)$, we may decide to build $u\rk2$ without incorporating the information of the remainder, i.e. using the relation $u\rk2 = \Phi\rk2 v\rk2$ instead of $u\rk2 = \Phi\rk2 v\rk2 + w\rk2$. 
This is plotted in the right-hand side of Fig. \ref{Fig:toy_err_eps_without_w}. 
As expected, we observe two regimes, with an error of size $\bigO(\eps^3)$ for $\eps^3$ larger than $\Delta t^2$ and of size $\bigO(\Delta t^2)$ for smaller values of $\eps$. 
It confirms that the macro variable enhanced by the near-identity map gives accuracy for small values of $\eps$ and that relevant information for large values of $\eps$ is retained by the micro variable. 

\pagebreak

\begin{figure}[h!]
\begin{center}
\includegraphics[scale=0.18, keepaspectratio=true]{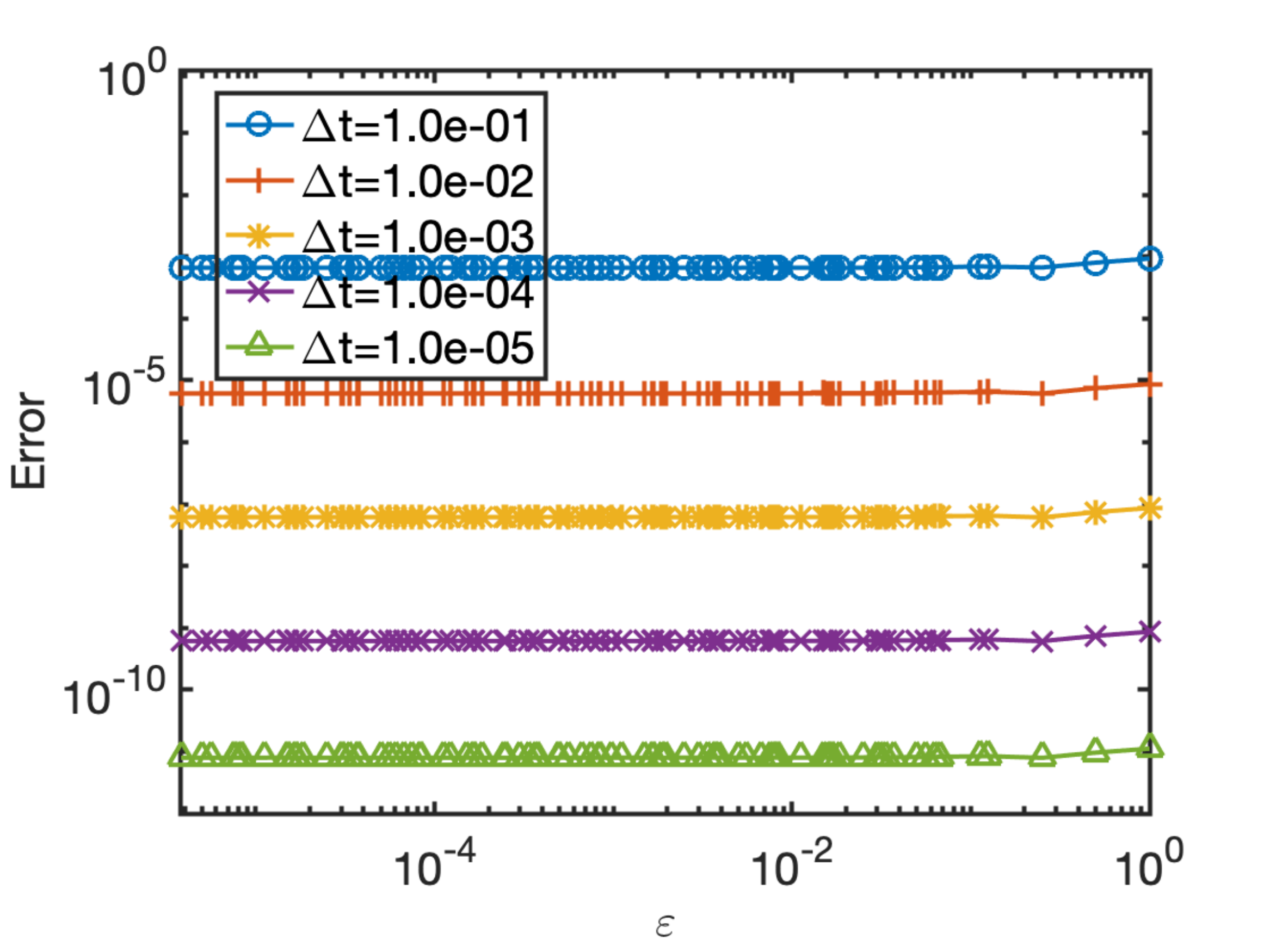}\hspace{0.5cm}
\includegraphics[scale=0.18, keepaspectratio=true]{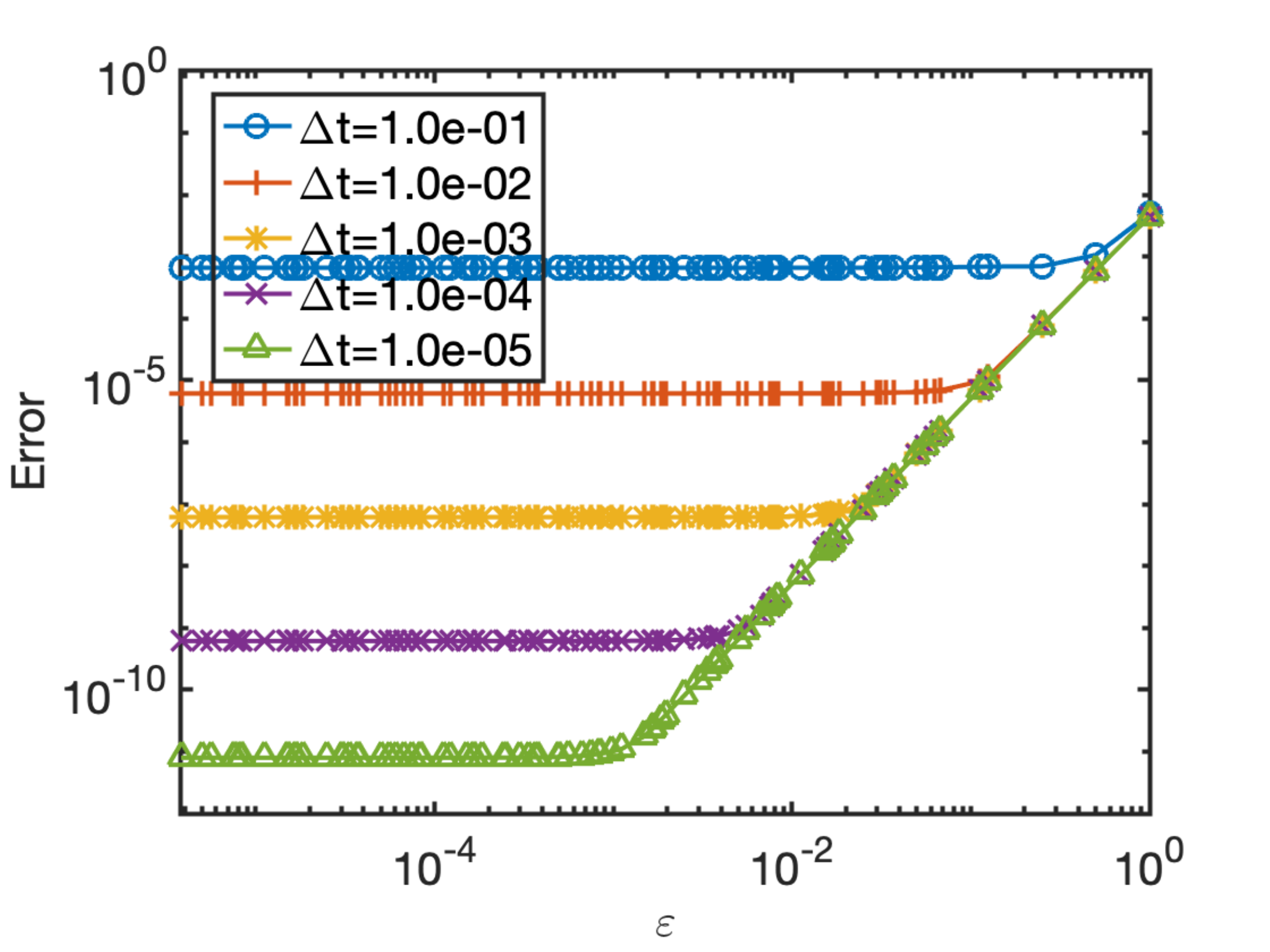}
\end{center}
\caption{Case 3F, RK2 scheme solving the macro equation \eqref{eq:mima} for $n = 2$; error with respect to $\eps$ for various $\Delta t$ defining $u\rk2 = \Phi\rk2 v\rk2 + w\rk2$ (left) and $u\rk2 = \Phi\rk2 v\rk2$ (right).}
\label{Fig:toy_err_eps_without_w}
\end{figure}

The error with respect to $\Delta t$ is presented in Figs. \ref{Fig:toy_err_dt}--\ref{Fig:toy_err_dt_other_schemes} for various values of $\eps$. In Fig. \ref{Fig:toy_err_dt}, the same value is used for the non-stiff convergence order $s$ of the scheme and for the approximation order $n$ of the micro-macro decomposition. 
The curves corresponding to the various $\eps$ are indistinguishable straight lines, in accordance with the flat lines of Fig.~\ref{Fig:toy_err_eps_without_w}. 
This illustrates once more the uniform accuracy of the method. 
Moreover, we actually get a slope of $1$ in the left figure (case $s = n = 1$) and of $2$ in the right figure (case $s = n = 2$).

\begin{figure}[h!]
\begin{center}
\includegraphics[scale=0.18, keepaspectratio=true]{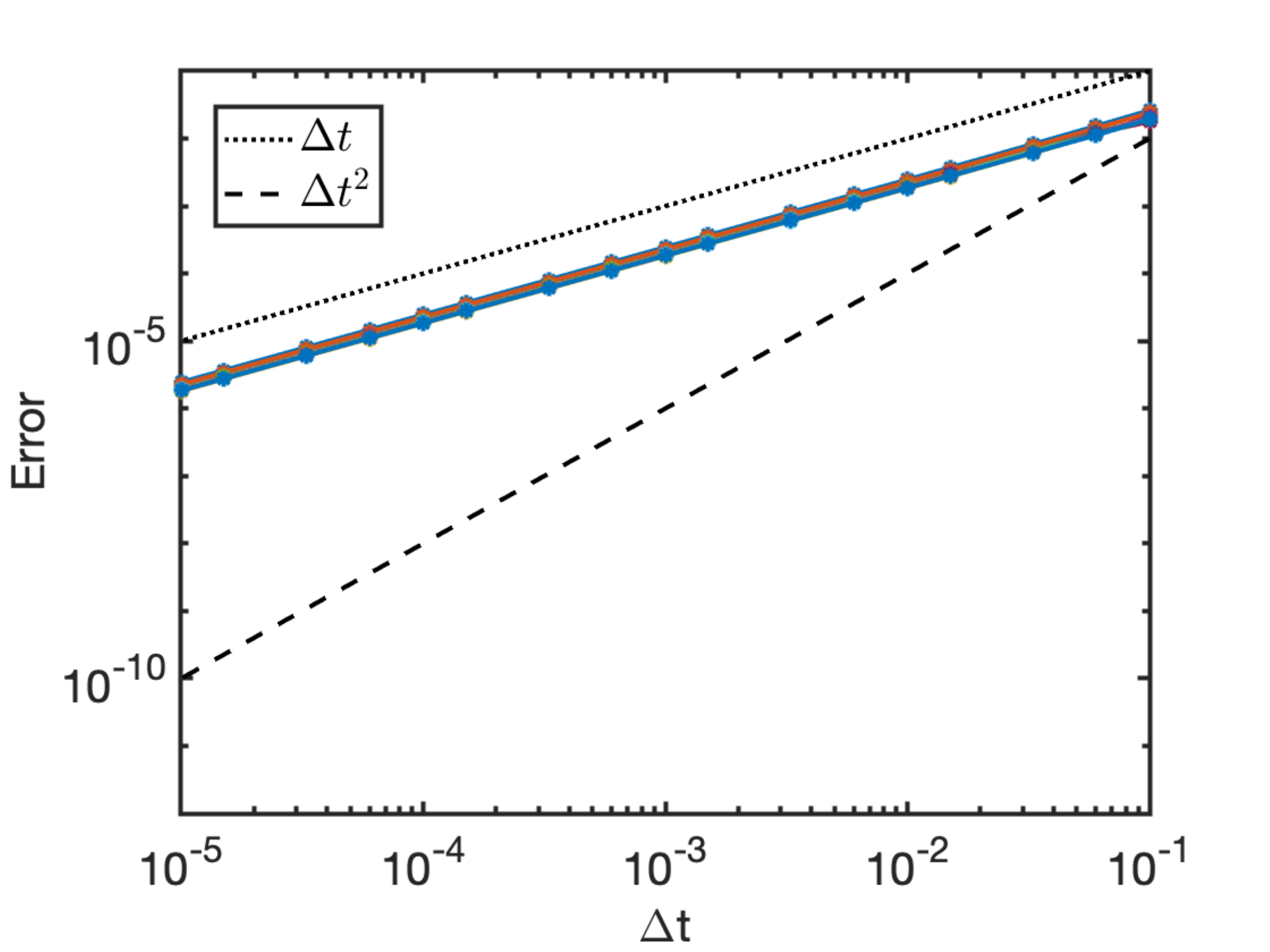}\hspace{0.5cm}
\includegraphics[scale=0.18, keepaspectratio=true]{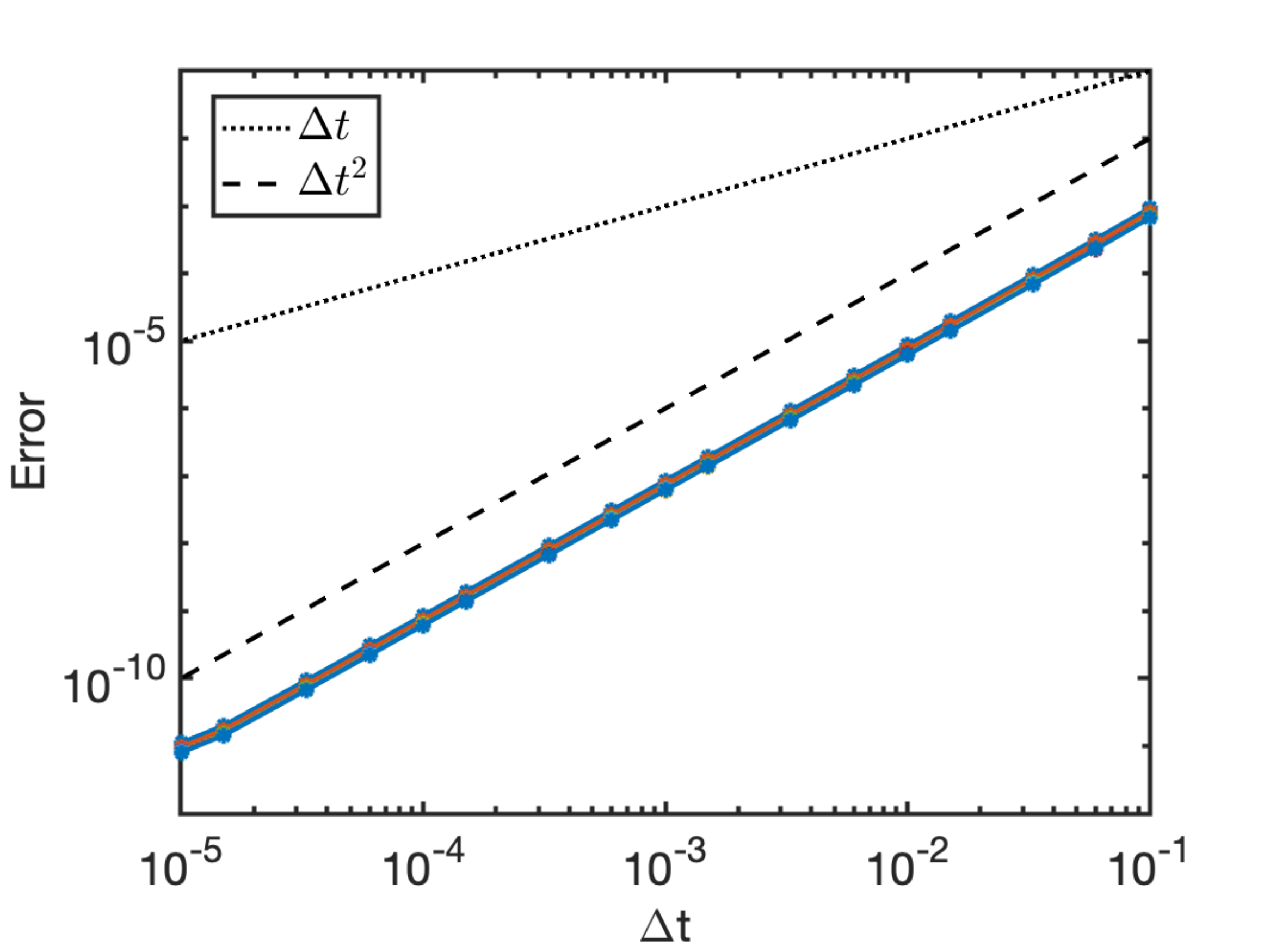}
\end{center}
\caption{Case 3F, solving \eqref{eq:mima} for $n=1$ and using EE (left) and for $n = 2$ and using RK2 (right); error with respect to $\Delta t$ for various $\eps$.}
\label{Fig:toy_err_dt}
\end{figure}

In Corollary~\ref{corollary:mima_ua}, we state that the scheme should be of the same order as that of the micro-macro expansion. 
We illustrate this in Fig.~\ref{Fig:toy_err_dt_other_schemes} (left) by using a standard scheme of order 2 but keeping the micro-macro approximation of order 1. 
At first glance, the error is of size $\bigO(\Delta t^2)$ but there exist some couples $(\Delta t, \eps)$ for which the error deteriorates. 
On the contrary, perfect straight lines fully on top of each other are observed in Fig. \ref{Fig:toy_err_dt_other_schemes} (right) when using the RK2int scheme. 
It illustrates that the convergence order may be increased by one using a Runge-Kutta integral scheme as mentioned in Section \ref{sec:result_ua}.

\begin{figure}[h!]
\begin{center}
\includegraphics[scale=0.18, keepaspectratio=true]{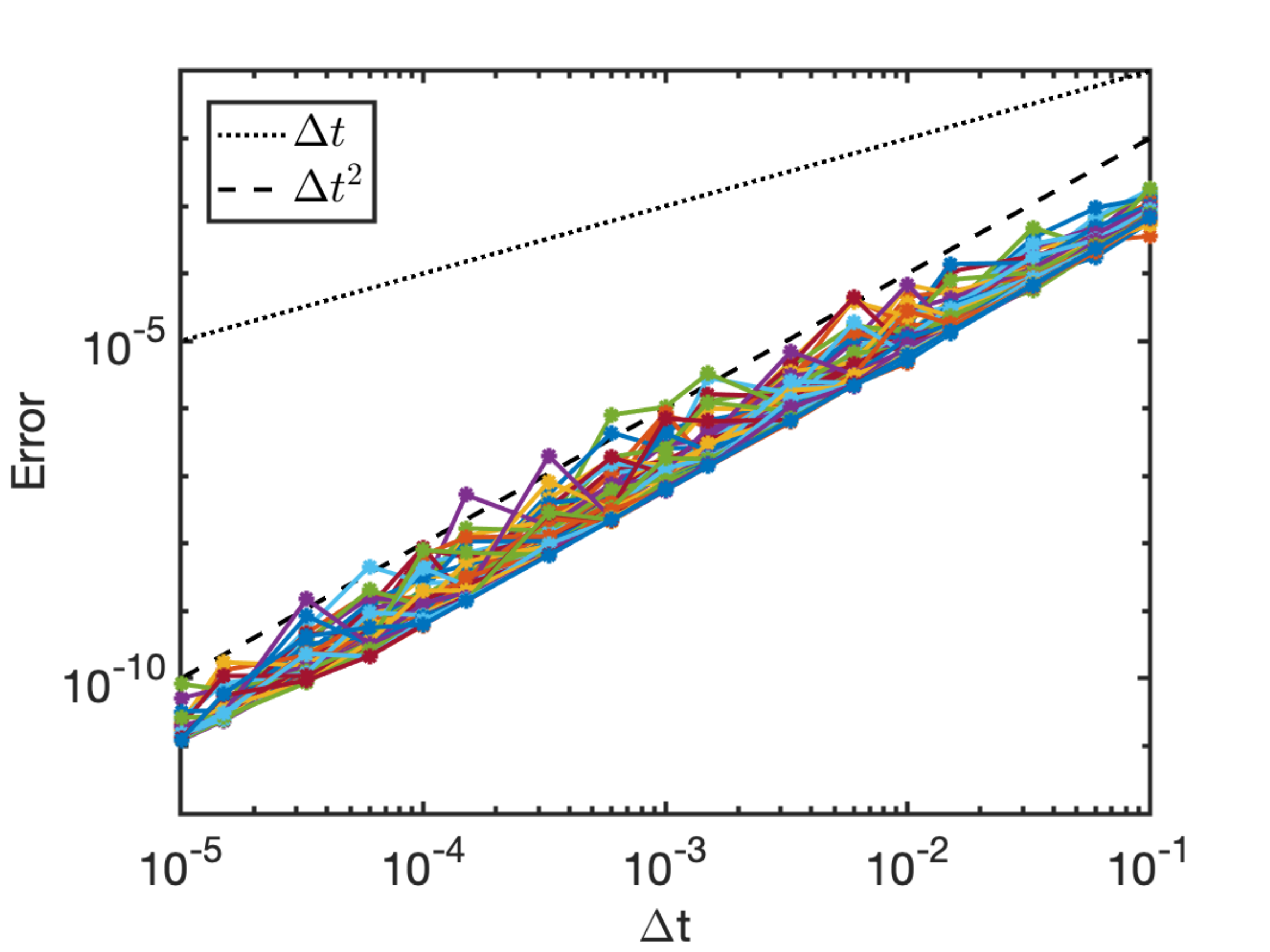}\hspace{0.5cm}
\includegraphics[scale=0.18, keepaspectratio=true]{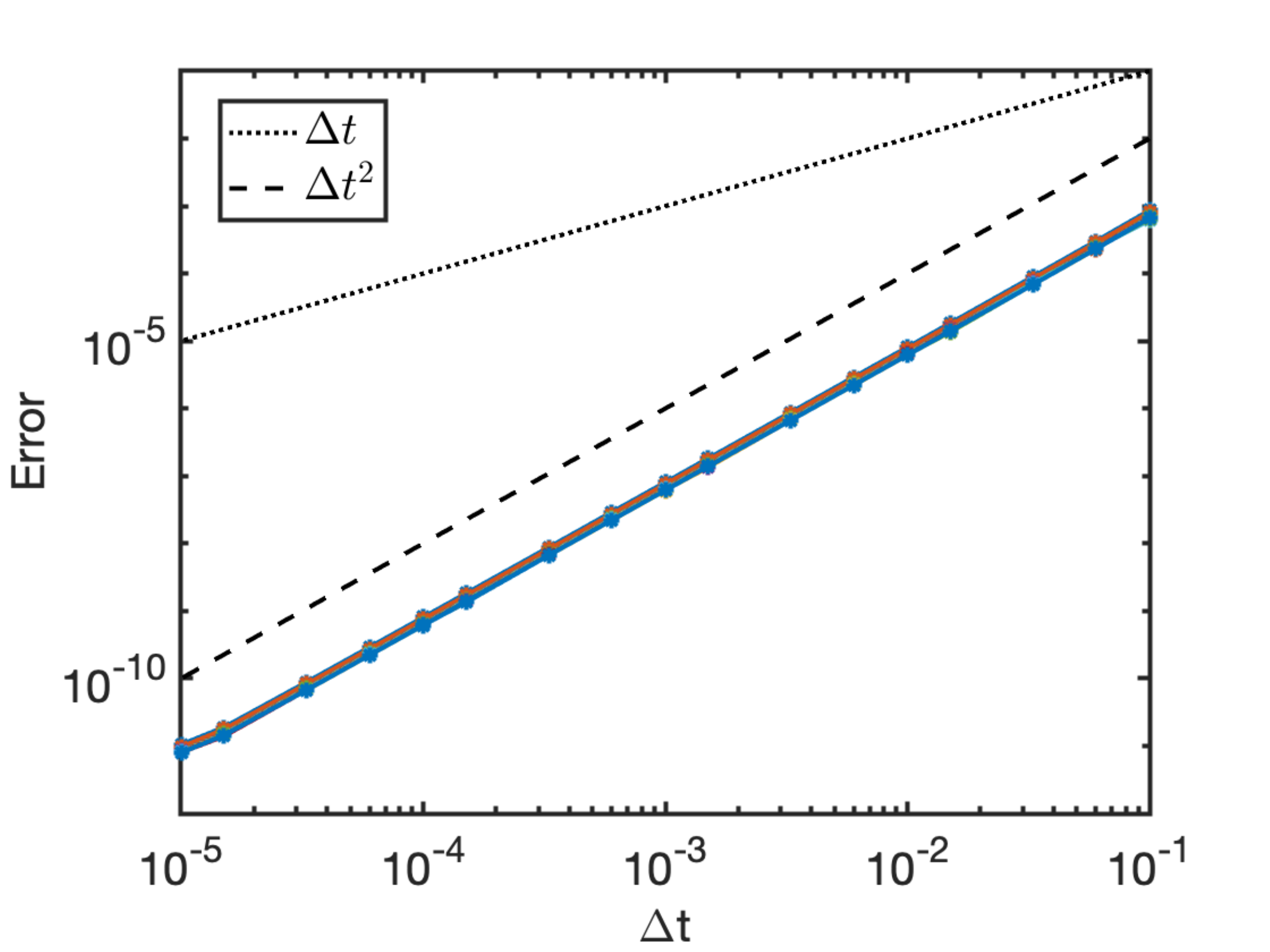}
\end{center}
\caption{Case 3F, solving \eqref{eq:mima} for $n = 1$ using RK2 (left) and RK2int (right); error with respect to $\Delta t$ for various $\eps$.}
\label{Fig:toy_err_dt_other_schemes}
\end{figure}

\subsubsection{Errors adding the decreasing term}

To finish the discussion on the scalar test problem, we finally add the exponentially decreasing part $a^\dmp_\tau$ choosing now $\gamma = 1$. 
In Fig. \ref{Fig:toy_alpha_1}, we display the error with respect to $\eps$ for various $\Delta t$ (left) and with respect to $\Delta t$ for various $\eps$ (right) solving the micro-macro problem \eqref{eq:mima} of order $n = 2$ with the RK2 scheme. 
We obtain uniform accurate results since all errors are of size $\bigO(\Delta t^2)$ independently of $\eps$. 
This last case allows to check that the flat part in the sharp-flat decomposition does not bring further numerical difficulties, in accordance with the feeling given by the proofs of Section~\ref{sec:proofs}.

\begin{figure}[h!]
\begin{center}
\includegraphics[scale=0.18, keepaspectratio=true]{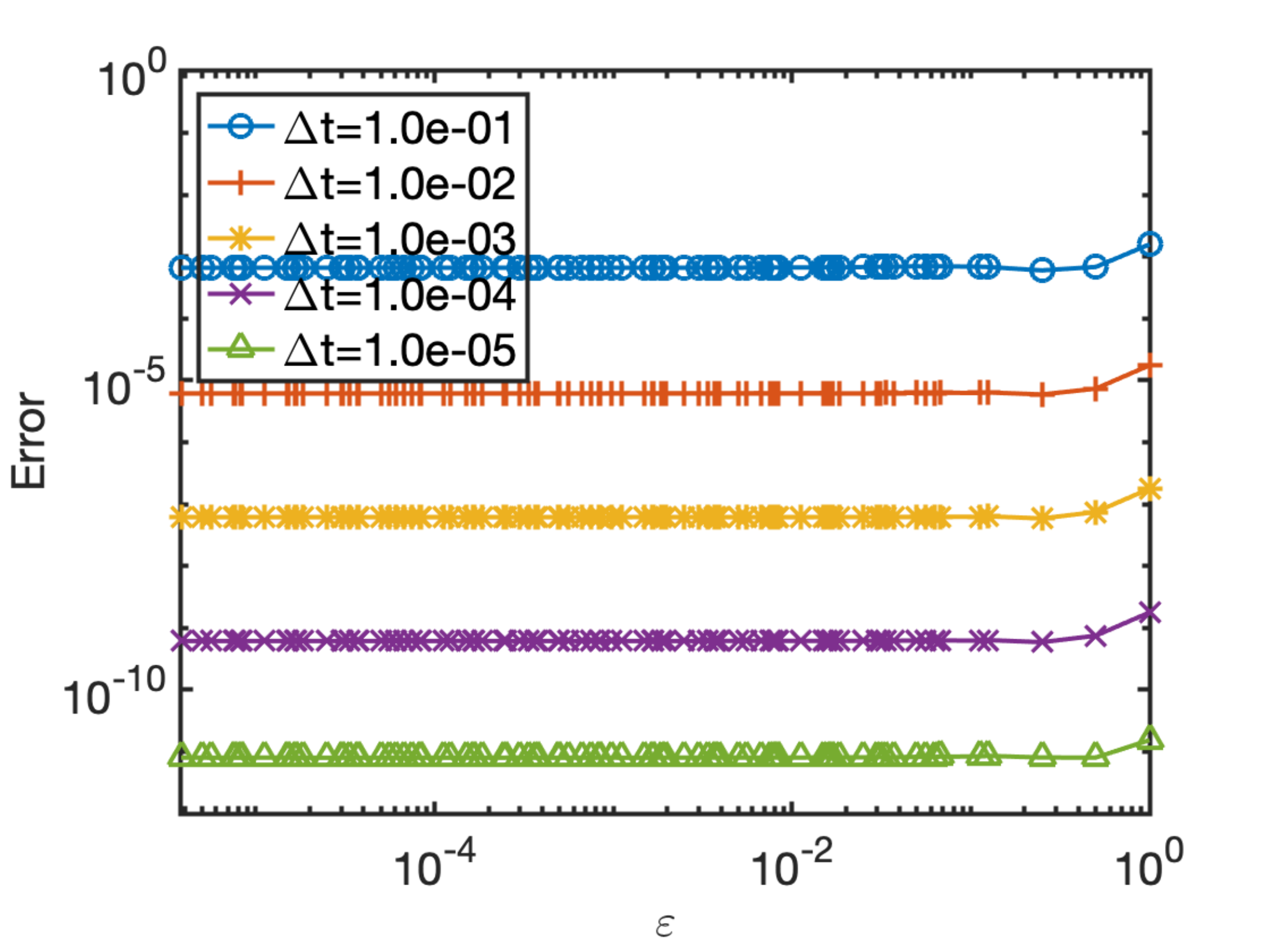}\hspace{0.5cm}
\includegraphics[scale=0.18, keepaspectratio=true]{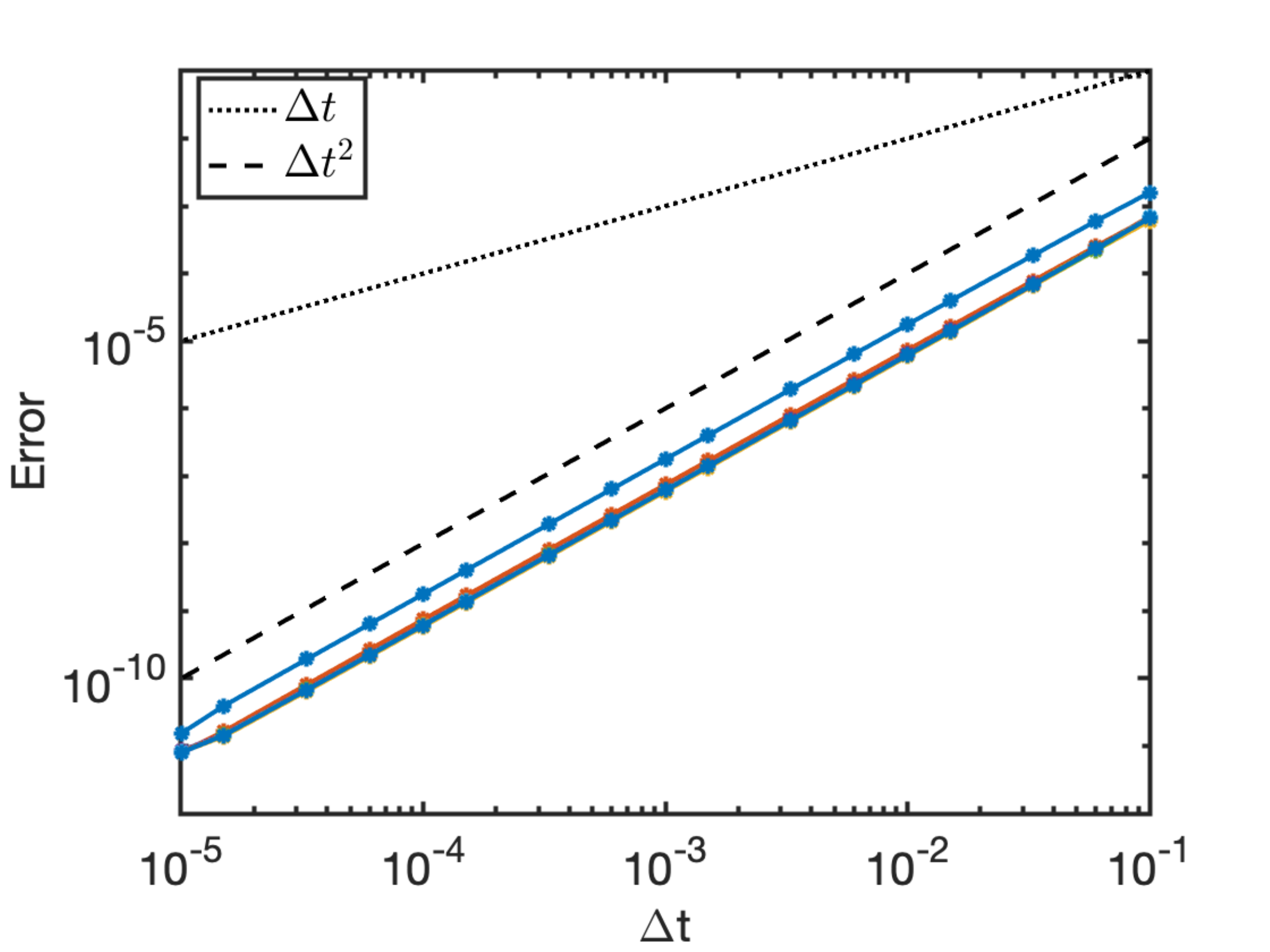}
\end{center}
\caption{Case 3F, $\gamma = 1$, solving \eqref{eq:mima} for $n = 2$ using RK2; error with respect to $\eps$ for various $\Delta t$ (left) and with respect to $\Delta t$ for various $\eps$ (right).}
\label{Fig:toy_alpha_1}
\end{figure}

\pagebreak

\subsection{Bloch model and a hierarchy of approximations}

\subsubsection{Presentation of the Bloch model}

Let us briefly present a Bloch model that governs the time evolution of the density matrix $\rho \in \mathcal{M}_n(\setC)$ associated to a quantum system described by $n$ discrete energy levels and forced by a given high frequency electromagnetic wave. 
More precisely, it corresponds to the scaled matrix equation 
\begin{equation*}
	i \eps^2 \partial_t \rho(t) = [H_0 - \eps \mathcal{V}(t/\eps^2), \rho(t)] + i Q(\rho),
\end{equation*}
where $[\cdot, \cdot]$ denotes the commutator between two matrices, $H_0 = \diag(E_1, \cdots, E_n)$ is the free Hamiltonian expressed in terms of the (scaled) energies $E_j$ associated to each energy level, $\mathcal{V}$ is the time dependent electric potential matrix and $Q$ is a relaxation term that takes into account physical phenomena involving energy-dissipating processes or collisions between particles. 
The density matrix $\rho$ is made of (non negative) diagonal quantities, denoted $\rho_{{\rm d}, j}$, called populations and representing the occupation number of the levels, and of off-diagonal quantities, denoted $\rho_{{\rm od}, jk}$ (with $j \neq k$), called coherences and describing the probability of transitions from one level to another.

A rigorous asymptotic analysis of this model (when the small parameter $\eps > 0$ goes to $0$) has been addressed in \cite{BidegarayFesquet2004a}. 
In that scaling, the evolution is considered over long times, of size $1/\eps^2$  and the influence of the electromagnetic wave is weak, of size $\eps$, and depends on the fast time scale $t/\eps^2$. 
Considering the bipolar approximation, we assume that the entries of the interaction potential matrix $\mathcal{V}$ are of the form
\begin{equation*}
	 \mathcal{V}_{jk}(\tau) = V^\qp(\tau) p_{jk},
\end{equation*}
where $p$ is a given hermitian (dipolar moment) matrix and $V^\qp$ is a given (quasi-\nobreak)pe\-riodic function that takes into account the time dependence of the wave. 
Also, we assume that the quantum system relaxes to a given equilibrium state, via relaxation coefficients $\gamma_{jk}$ which have an effect on the off-diagonal part of the density matrix only. 
More precisely we consider that $Q(\rho)_{jk} = - \gamma_{jk} \rho_{jk}$ where relaxation coefficients, uniform with respect to $\eps$, are such that $\gamma_{jk} = \gamma_{kj} > 0$ for all $j \neq k$ and $\gamma_{jj} = 0$ for all $j$. 
The basic dynamics is thus given by a (damped) high-frequency oscillation, with frequency $1/\eps^2$. Indeed, denoting $\Omega_{jk} = - i (E_j-E_k) - \gamma_{jk}$, the coefficient $\rho_{jk}$ of the density matrix is solution to the equation    
\begin{equation}
	\label{eq:bloch2}
	\partial_t \rho_{jk}(t) = \frac1{\eps^2} \Omega_{jk} \rho_{jk}(t) + \frac{i}\eps V^\qp(t/\eps^2)  [p, \rho(t)]_{jk}.
\end{equation}
Finally, at initial time $t = 0$, we assume a density matrix with vanishing coherences and non negative populations $\rho_{\rm d}^{\rm init}$. 

\begin{remark}
	Notice that, compared to the previous sections, the characteristic time is now $\eps^2$ instead of $\eps$. 
	Consequently, $\tau$ now refers to $t/\eps^2$ and not to $t/\eps$.
\end{remark}

In this paper, we propose to use the micro-macro problem introduced and analyzed in the previous sections to obtain a uniformly accurate scheme. 
However, we do not tackle the resolution of the entire Bloch model \eqref{eq:bloch2} with both coherences and populations. 
It does not enter directly  into the ``sharp-flat'' framework and its numerical resolution is beyond the scope of this paper. 
Instead, we consider an approximation (presented in the next section) that gives an equation governing the populations only. 
Indeed, in \cite{BidegarayFesquet2004a}, it has been proven that, in the limit $\eps \to 0$, the diagonal part of the density matrix solution to \eqref{eq:bloch2} tends to the solution of a rate equation in which the transition rate is an appropriate time average of the potential, while the off-diagonal part vanishes. 
Interestingly, the asymptotic analysis is based on successive approximations which, after some point, all fit into the sharp-flat framework considered here.

\subsubsection{Transformation to a ``sharp-flat'' problem}

We first transform the model \eqref{eq:bloch2} into a closed equation governing the populations $\rho_{\rm d}$ only. 
As detailed in \cite{BidegarayFesquet2004a}, this is done by writing the equation on coherences as an integral equation and keeping only the first order expansion in $\eps$ of the right-hand side. 
Since the coherences initially vanish, this depends only on the populations, and inserting it into the population equation gives a time delayed integro-differential equation. 
Finally, the delay being small, the populations tend to be the solution of the delay-free equation
\begin{equation}
	\label{eq:bloch_pop1}
  	\partial_t \rho^{\qp \dmp}_{{\rm d}, j} 
	= \sum_{l \neq j} (\Psi_{t/\eps^2})_{lj} \big(\rho^{\qp \dmp}_{{\rm d}, l}(t) - \rho^{\qp \dmp}_{{\rm d}, j}(t) \big),
\end{equation}
where we introduced the time dependent transition rate
\begin{equation}
	\label{eq:bloch_psi1}
	(\Psi_\tau)_{lj} 
	= 2 |p_{lj}|^2 \Re\Big[ V^\qp(\tau) \int_0^\tau \exp\big( \Omega_{lj} \sigma \big) V^\qp(\tau - \sigma) \D\sigma \Big].
\end{equation}
In \cite{BidegarayFesquet2004a}, it is proven that, for all $T > 0$, there exists $C > 0$, independent of $\eps$, such that
\begin{equation*}
	\| \rho_{\rm d} - \rho_{\rm d}^{\qp \dmp} \|_{L^\infty([0,T], l^1)} \leq C \eps,
\end{equation*}
with the notation $\| \rho_{\rm d} \|_{l^1} = \sum_{j=1}^n | \rho_{{\rm d}, j} |$.

As it can be seen in Appendix \ref{appendix:psi} where explicit expressions of $(\Psi_\tau)_{lj}$ are computed for a specific $r$-chromatic $V^\qp(\tau)$, the transition rate defined in \eqref{eq:bloch_psi1} is the sum of a (quasi-)periodic part $\Psi_{lj}^\qp$ and an exponentially decaying part  $\Psi_{lj}^\dmp$ as in Section~\ref{sec:results}. 
This is the reason why the problem \eqref{eq:bloch_pop1}-\eqref{eq:bloch_psi1} completed with the initial condition $\rho^{\qp \dmp}_{\rm d}(0) = \rho_{\rm d}^{\rm init}$ is described by the sharp-flat framework and can be solved with uniform accuracy using the micro-macro problem \eqref{eq:mima} as we illustrate in the sequel.

\begin{remark}
	The equation \eqref{eq:bloch_pop1} can be written
	\begin{equation*}
		\partial_t \rho^{\qp \dmp}_{\rm d}(t) = a_{t/\eps^2} \rho^{(1)}_{\rm d}(t) 
	\end{equation*}
	introducing the population vector $\rho^{\qp \dmp}_{\rm d} = (\rho^{\qp \dmp}_{{\rm d}, 1}, \cdots, \rho^{\qp \dmp}_{{\rm d}, n})^T$ and 	defining the matrix map $\tau \mapsto a_\tau$ such that 
	\begin{equation*}
		(a_\tau)_{jk} = 
		\begin{cases} 
			(\Psi_\tau)_{kj} & \text{ if } j \neq k, \\
			- \sum_{l \neq j} (\Psi_\tau)_{lj} & \text{ if } j = k.  
		\end{cases}
	\end{equation*}
	This form is used for the implementation. 
	Nevertheless, for the simplicity of the presentation, we consider in the sequel the matrix $\Psi_\tau$ instead of $a_\tau$ and we have to keep in mind that it is the matrix occurring in the rate equation.
\end{remark}

\subsubsection{Further approximations of the populations}

Before presenting some numerical results, we shortly describe the next approximation in the hierarchy analyzed in \cite{BidegarayFesquet2004a} as well as the limit problem with the averaged transition rate.

The transition rate $\Psi$ defined in \eqref{eq:bloch_psi1} can be  approximated by a rate $\Psi^\infty$ defined integrating up to $+\infty$ instead of $\tau$, i.e.
\begin{equation}
	\label{eq:bloch_psi_inf}
	(\Psi^\infty_\tau)_{lj}  
	= 2 |p_{lj}|^2 \Re\Big[ V^\qp(\tau) \int_0^{+\infty} \exp\big( \Omega_{lj} \sigma \big) V^\qp(\tau - \sigma) \D\sigma \Big].
\end{equation}
As emphasized by the explicit computations presented in Appendix \ref{appendix:psi}, it corresponds to neglecting the exponentially decaying part $\Psi_{lj}^\dmp$. 
We obtain new approximate populations, denoted $\rho^{\rm osc}_{\rm d}$, that verify the following rate equation with a quasi-periodic time dependent transition rate 
 \begin{equation}
 	\label{eq:bloch_pop_osc}
  	\partial_t \rho^{\rm osc}_{{\rm d}, j}(t) 
	=  \sum_{l \neq j} (\Psi^\infty_{t/\eps^2})_{lj} \big( \rho^{\rm osc}_{{\rm d}, l}(t)- \rho^{\rm osc}_{{\rm d}, j} (t) \big).
 \end{equation}
Finally, by averaging theory, $\Psi^\infty$ (as well as $\Psi$) can be approximated by a time independent transition rate
\begin{equation}
	\label{eq:bloch_psi_average}
	\mean{\Psi}_{lj} = \lim_{T \to +\infty} \frac1T \int_0^T (\Psi^\infty_\sigma)_{lj} \D\sigma 
	= \lim_{T \to +\infty} \frac1T \int_0^T (\Psi_\sigma)_{lj} \D\sigma
\end{equation}
leading to the limit problem
\begin{equation}
	\label{eq:bloch_pop_lim}
	\partial_t \rho^{\rm lim}_{{\rm d}, j}(t) 
	= \sum_{l \neq j} \mean{\Psi}_{lj} \big( \rho^{\rm lim}_{{\rm d}, l}(t) - \rho^{\rm lim}_{{\rm d}, j}(t) \big).
\end{equation}
Again, an explicit expression of $\mean{\Psi}$ is presented in Appendix \ref{appendix:psi}. In \cite{BidegarayFesquet2004a}, the convergence of $\rho^{\qp \dmp}_{\rm d}$ to $\rho^{\rm lim}_{\rm d}$ (and thus of $\rho_{\rm d}$ to $\rho^{\rm lim}_{\rm d}$) is proven to be in $\bigO(\eps)$ for quasi-periodic waves fulfilling the Diophantine inequality \eqref{eq:diophantine} and in $o(1)$ for more general KBM waves.

Obviously the averaged linear equation \eqref{eq:bloch_pop_lim} is not stiff and does not present any numerical difficulties. 
It is also the equation governing the zero order micro-variable $v\rk0$ when a micro-macro decomposition is used to approximate either \eqref{eq:bloch_pop1} or \eqref{eq:bloch_pop_osc}.

\subsubsection{Numerical results}

In this part, we present the numerical results obtained with the micro-macro schemes described in Section \ref{sec:num_scheme}. 
The micro-macro decomposition requires some offline symbolic computations depending on the prescribed electromagnetic wave. 
Due to the complexity of these  computations, we limit our micro-macro decomposition to order 1, thereby obtaining second order uniform accuracy with the RK2int scheme. 
It is also for this reason that we consider the equation \eqref{eq:bloch_pop1} for the case 1F and the equation \eqref{eq:bloch_pop_osc} for the case 3F.
An example of such computations is presented in Appendix~\ref{appendix:off-line_computations} for the monochromatic forcing.

For the numerical tests, we use in the sequel $n = 3$ quantum levels with scaled relative energies $E_1 = 0$, $E_2 = 2$ and $E_3 = 3$,  identical relaxation coefficients and dipolar moment coefficients ($\gamma_{jk} = 1-\delta_{jk}$ and $p_{jk} = 1 - \delta_{jk}$  where $\delta_{jk}$ is the Kronecker delta), $T = 10$ as final time and $\rho_{{\rm d}}^{\rm init} = [0, 0, 1]$ as initial population. 
The choice of frequencies $\omega_p$ in $V^\qp$ is similar to the one of the previous section (case 1F and case 3F), with $V^\qp(\tau) = \frac1r \sum_{p = 1}^r \cos(\omega_p \tau)$.

Since the exact solution of the problem is not known, we define instead a reference solution $\rho^{\rm ref}_{\rm d}$ to analyze the quality of the micro-macro solution $\rho^{\rm approx}_{\rm d}$. 
We use the EEint scheme directly applied to the equation \eqref{eq:bloch_pop1} or \eqref{eq:bloch_pop_osc} with a reference discretization step $\Delta t_{\rm ref} = 5.10^{-6}$, which yields an approximation accurate to at least $10^{-5}$. 
The error we compute is 
\begin{equation*}
	E(\Delta t, \eps) = \max_{0 \leq \ell \leq L} \| \rho^{{\rm ref}, \ell}_{\rm d} - \rho^{{\rm approx}, \ell}_{\rm d} \|_{l^1},
\end{equation*}
where $L+1$ is the number of discretization points used to compute the approximate solution ($L$ chosen as a divisor of $L^{\rm ref}$). \\

We first consider  a monochromatic wave (case 1F). 
In Fig. \ref{Fig:bloch_standard_scheme}, we present errors obtained when we solve the stiff problem \eqref{eq:bloch_pop1} with the standard EE scheme.
 As expected, it does not yield suitable results with, for a given $\Delta t$, increasing errors for decreasing $\eps$. 
 In addition, there exists some values $\eps$ for which the error is of size $\bigO(1)$ for any $\Delta t$. 
 On the contrary, when we apply the same standard EE scheme to the micro-macro problem \eqref{eq:mima} associated to \eqref{eq:bloch_pop1}, we obtain uniform accurate results as illustrated in Fig. \ref{Fig:bloch_MM_scheme}. 
 The reader may notice that for large values of $\eps$, the error is slightly degraded, in the sense that the error constant (which multiplies $\Delta t^s$ in the error) is slightly larger for $\eps > 0.1$. 
 However, this threshold is independent of the time-step $\Delta t$ and remains uniformly bounded w.r.t. $\eps$, which does not contradict our result of uniform accuracy. 
 This is furthermore verified when plotting the error with respect with $\Delta t$, which shows straight lines of slope 1. 
 The lines are not on top of each other only for values of $\eps$ larger than $0.1$ (blue, orange, yellow and purple results), but even for these values the slope remains unchanged.

\begin{figure}[h!]
\begin{center}
\includegraphics[scale=0.18]{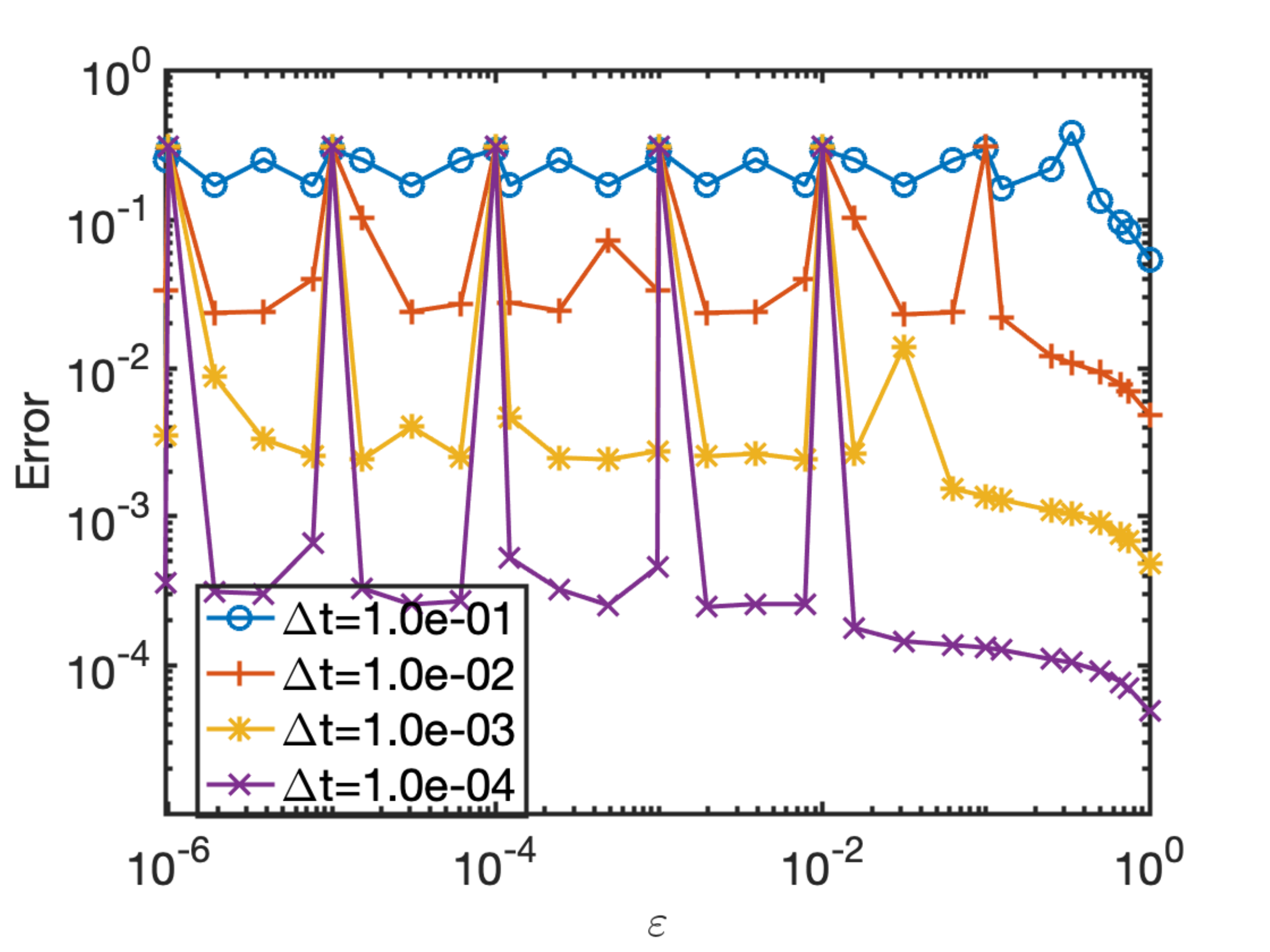}\hspace{0.5cm}
\includegraphics[scale=0.18]{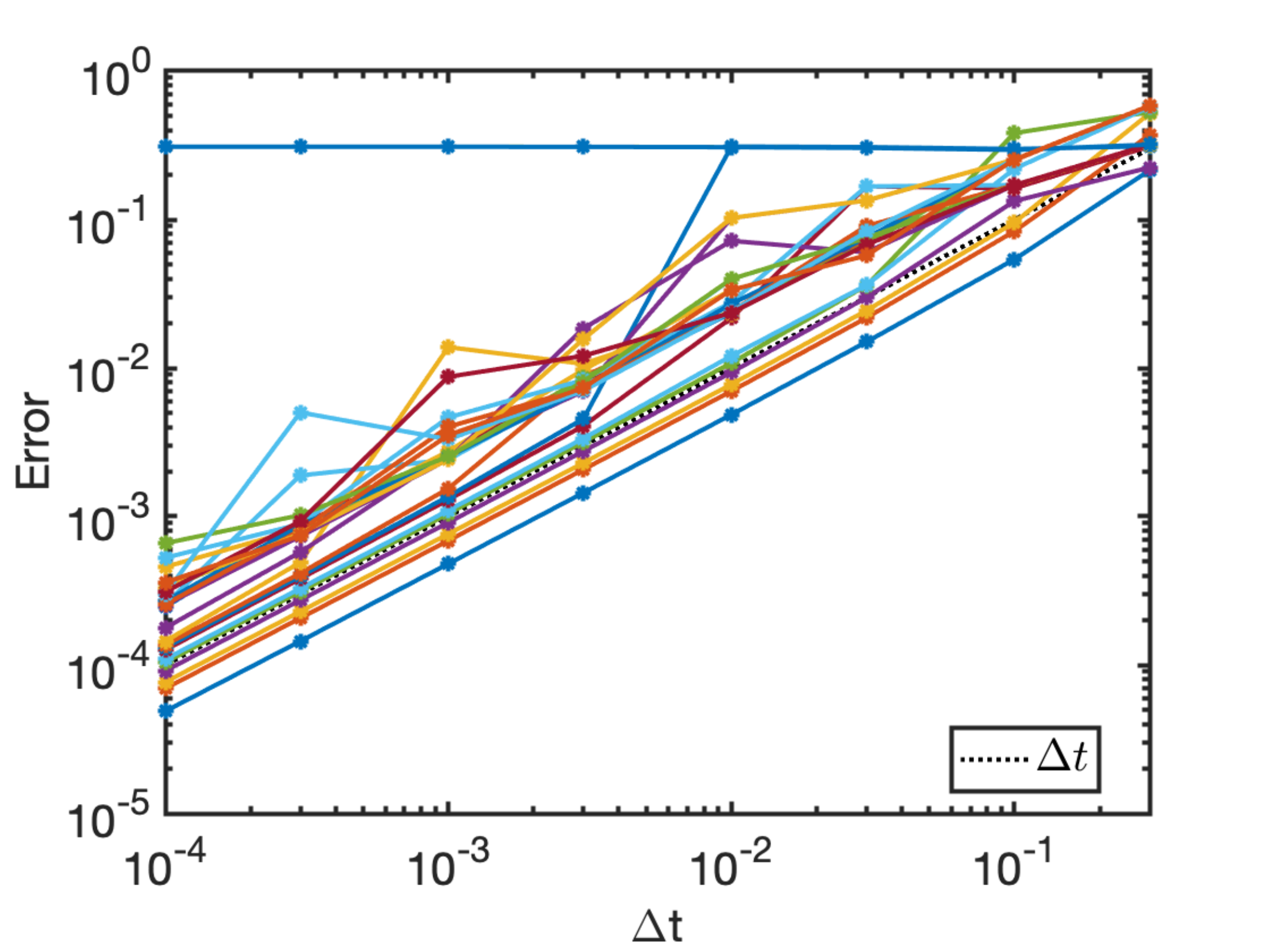}
\caption{Case 1F, solving \eqref{eq:bloch_pop1} using EE; error with respect to $\eps$ for various $\Delta t$ (left) and with respect to $\Delta t$ for various $\eps$ (right).}
\label{Fig:bloch_standard_scheme}
\end{center}
\end{figure}

\begin{figure}[h!]
\begin{center}
\includegraphics[scale=0.18]{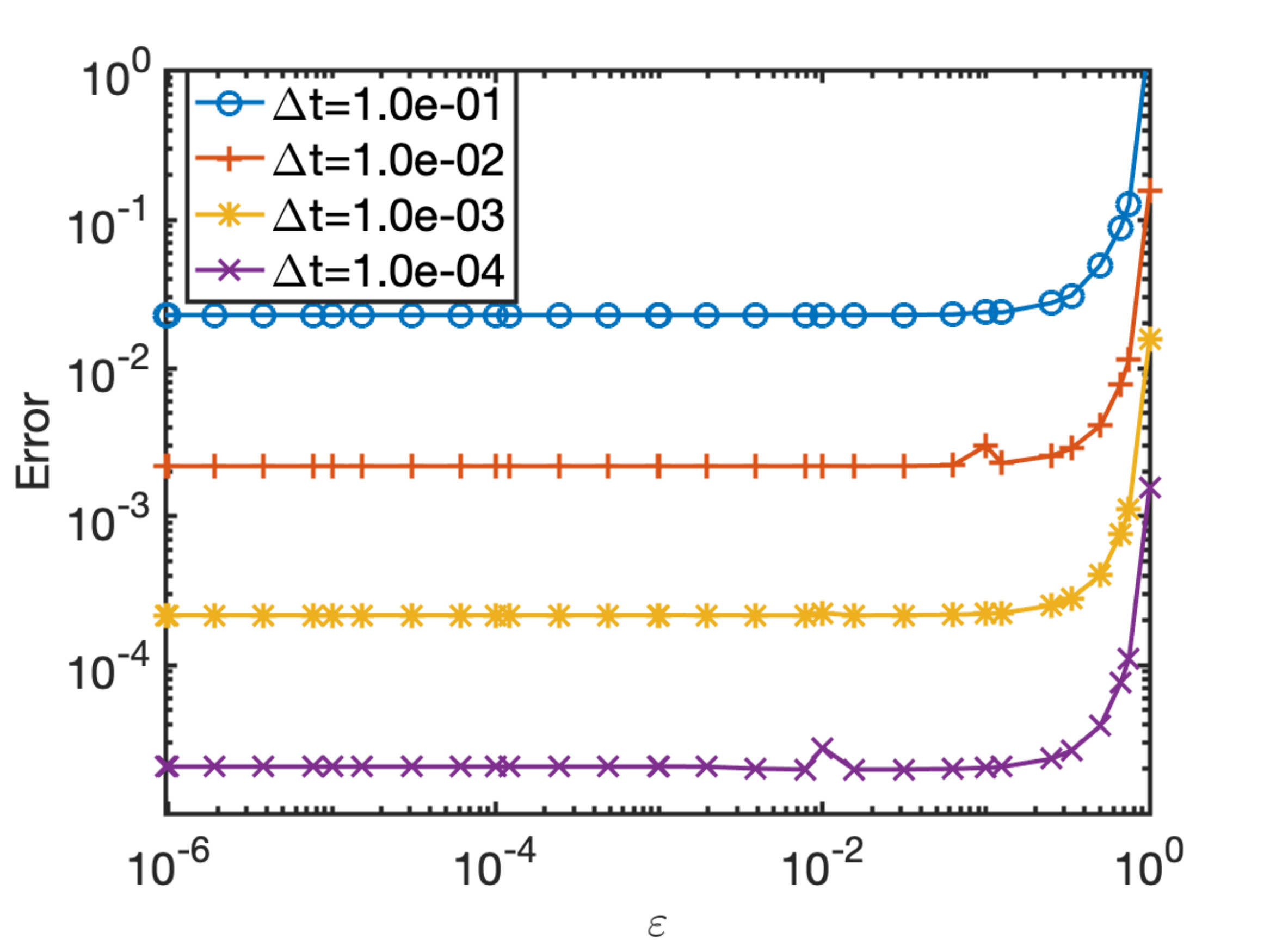}\hspace{0.5cm}
\includegraphics[scale=0.18]{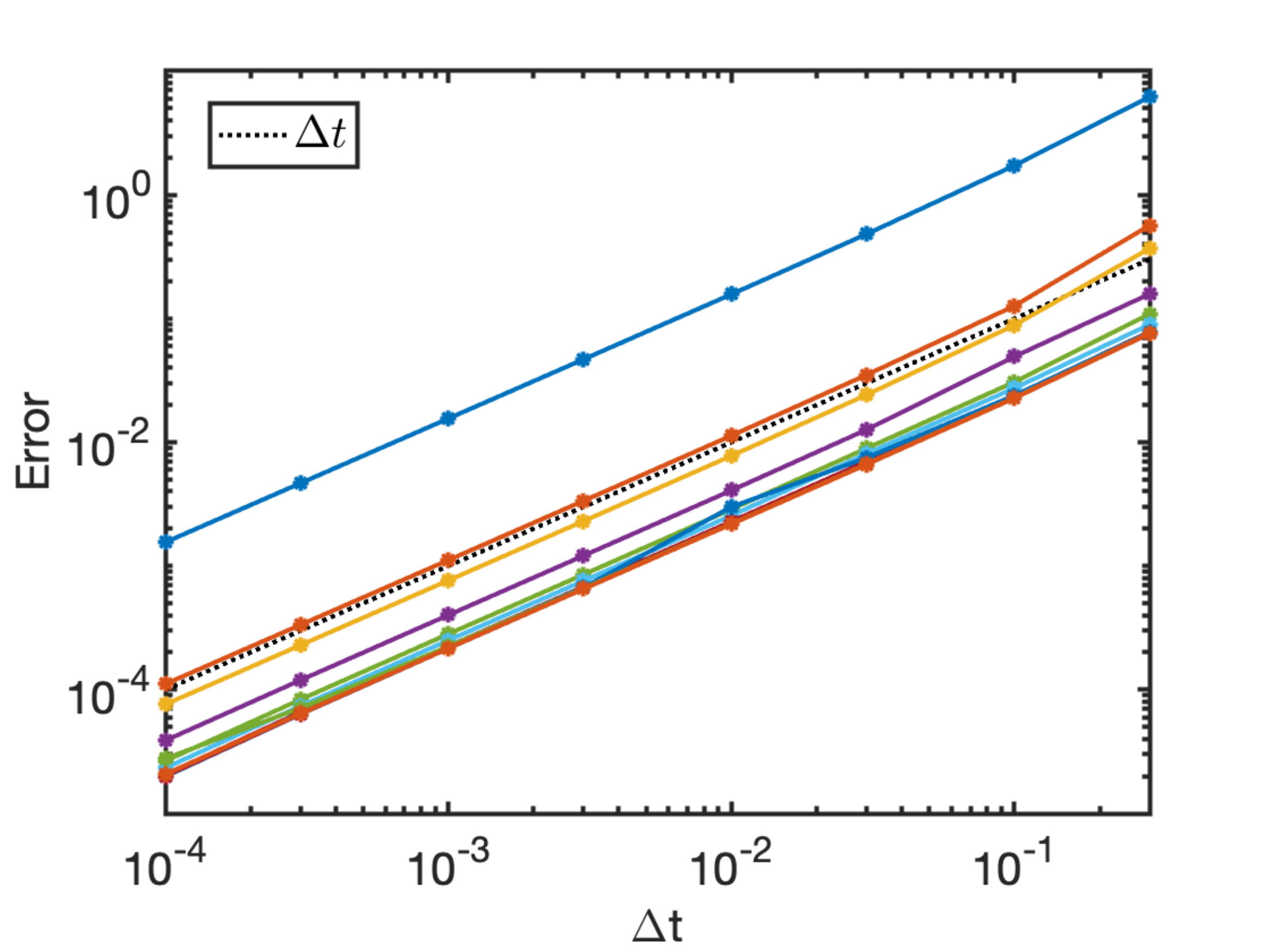}
\caption{Case 1F, solving \eqref{eq:mima} for $n = 1$ using EE; error with respect to $\eps$ for various $\Delta t$ (left) and with respect to $\Delta t$ for various $\eps$ (right).}
\label{Fig:bloch_MM_scheme}
\end{center}
\end{figure}

Then, we present in Fig. \ref{Fig:bloch_MM_RK2int_scheme} results obtained when the  micro-macro problem \eqref{eq:mima} of order 1 is solved using the RK2int scheme. 
As already discussed for the scalar test problem, it allows to increase the convergence error by one, giving uniform errors of size $\bigO(\Delta t^2)$. 
Notice that we are limited by the accuracy of the reference solution. 
This is why we do not consider discretization steps smaller than $3 \times 10^{-3}$ in this context, since even for this fairly large time step, results are slightly affected.

\begin{figure}[h!]
\begin{center}
\includegraphics[scale=0.18]{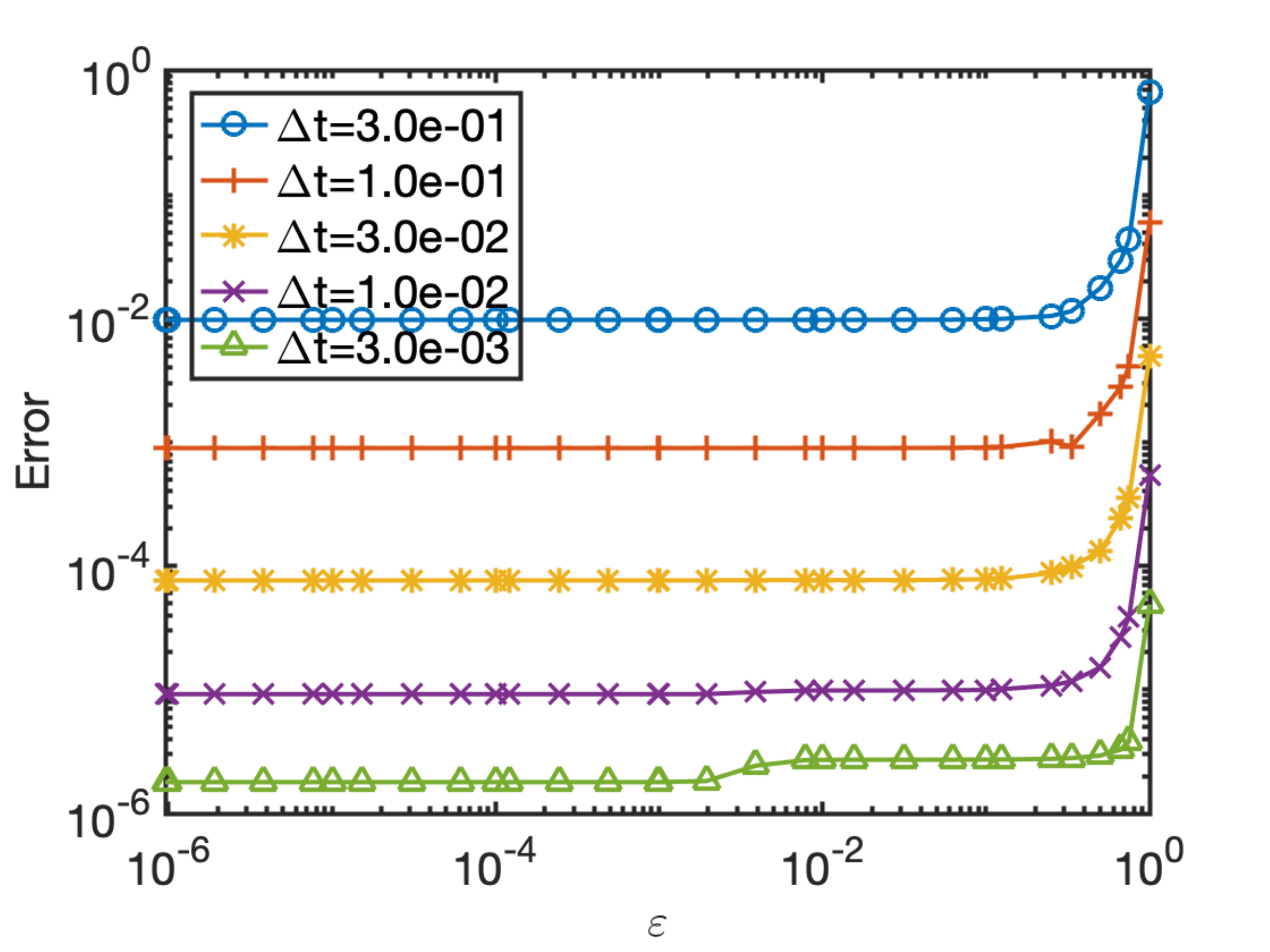}\hspace{0.5cm}
\includegraphics[scale=0.18]{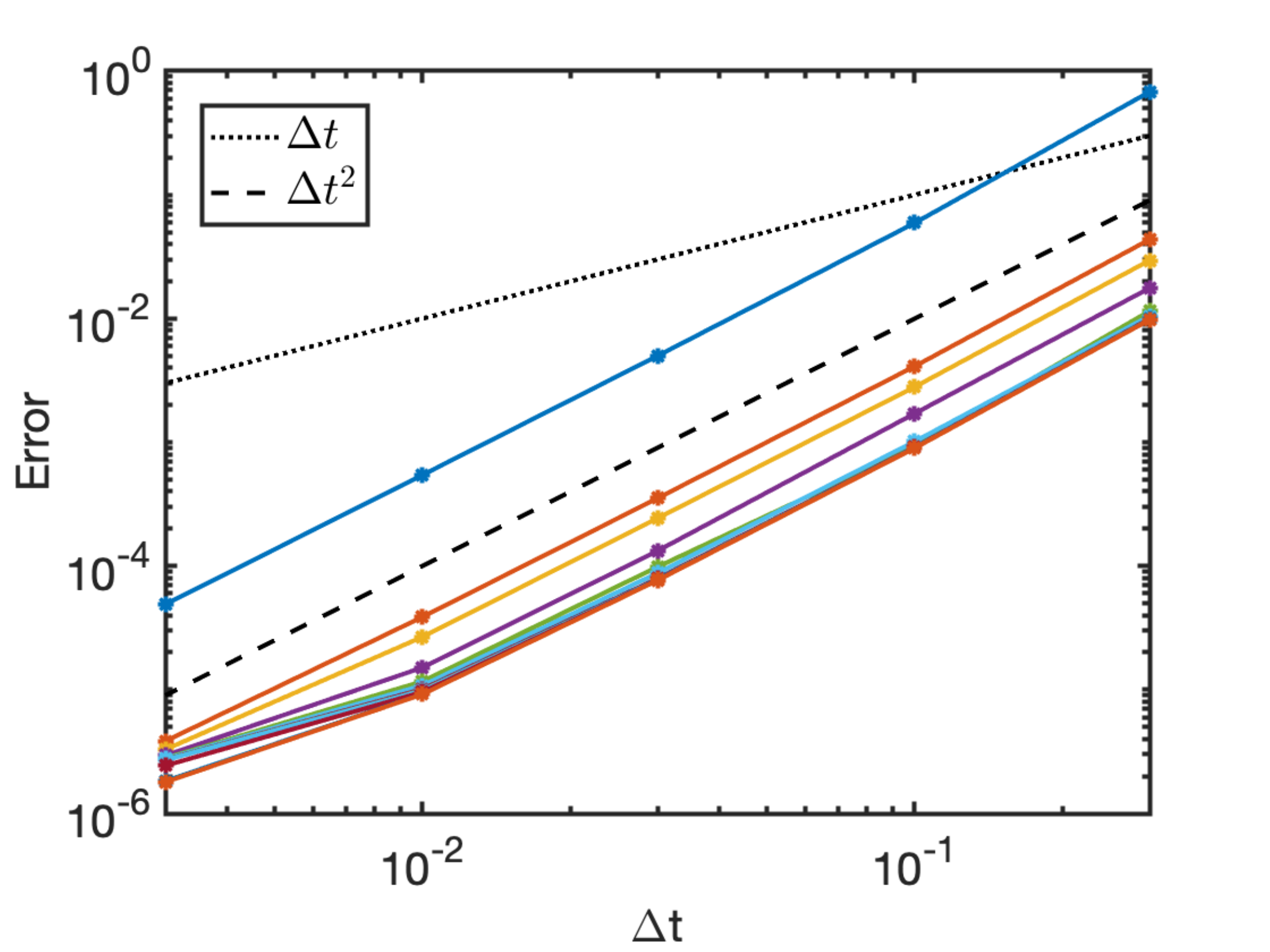}
\caption{Case 1F, solving \eqref{eq:mima} for $n=1$ using RK2int; error with respect to $\eps$ for various $\Delta t$ (left) and with respect to $\Delta t$ for various $\eps$ (right).}
\label{Fig:bloch_MM_RK2int_scheme}
\end{center}
\end{figure}

Finally, we consider a multi-chromatic wave (case 3F) for the decay-free problem~\eqref{eq:bloch_pop_osc}. 
Thanks to the in-depth study of the scalar test case, we know that the exponentially decreasing terms do not bring further numerical difficulties thus this test case remains representative of the problem \eqref{eq:bloch_pop1}. 
Results presented in Fig. \ref{Fig:bloch_MM_scheme_3F} are qualitatively similar to the previous ones. 
It confirms that our approach is suitable to describe the long time evolution of populations for a quantum system forced by a quasi-periodic electromagnetic wave.

\begin{figure}[h!]
\begin{center}
\includegraphics[scale=0.18]{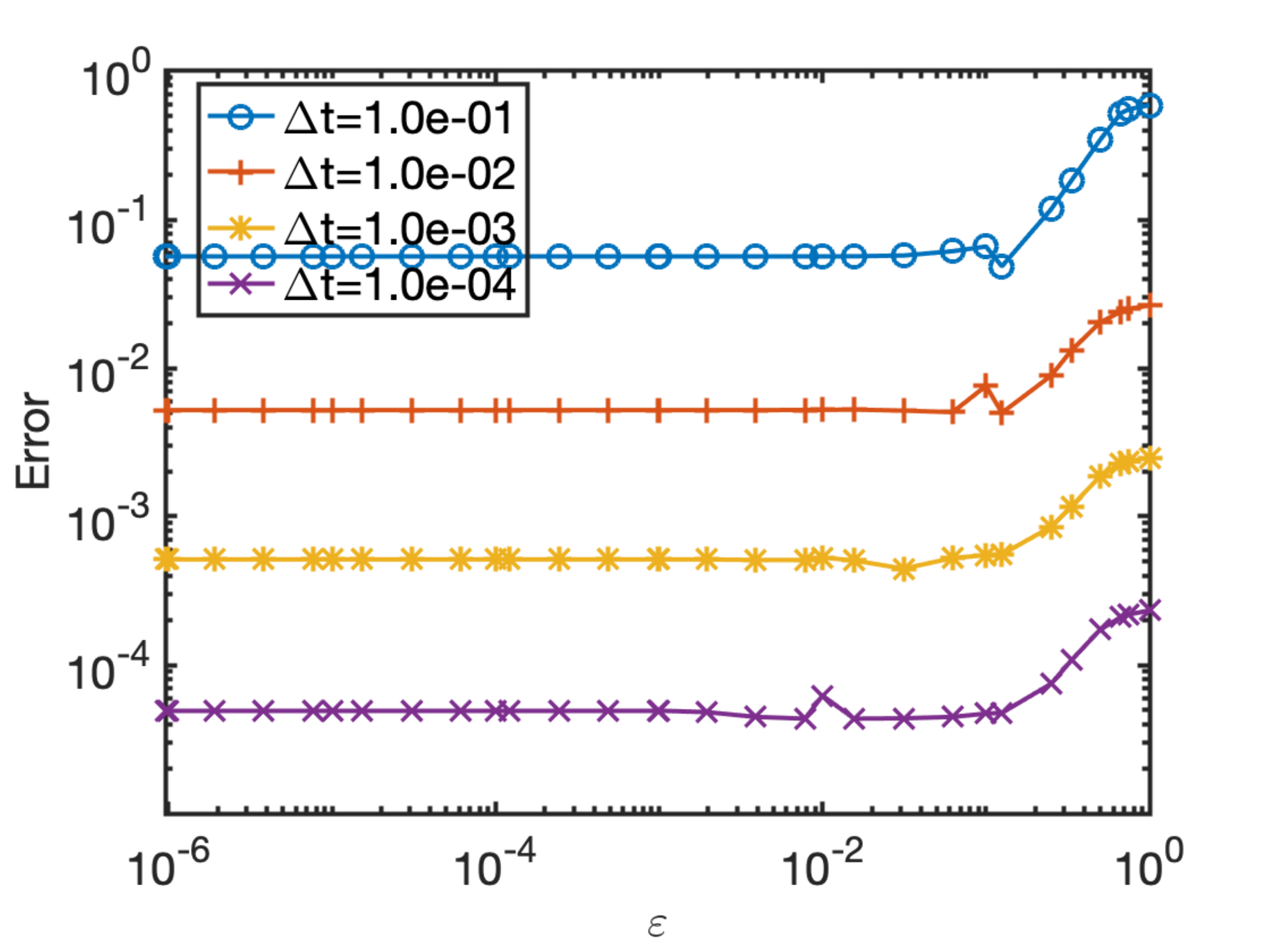}\hspace{0.5cm}
\includegraphics[scale=0.18]{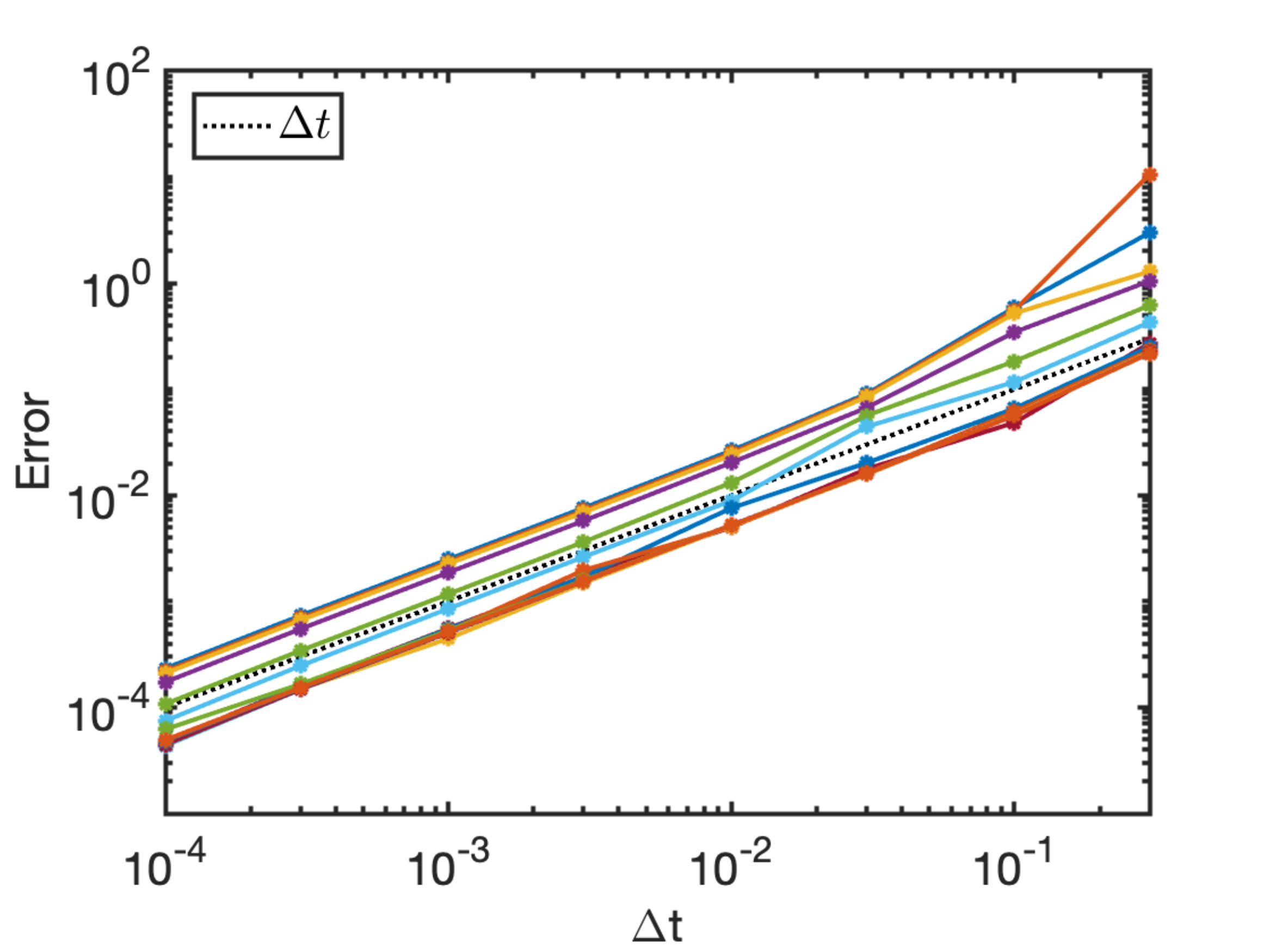}
\caption{Case 3F, solving \eqref{eq:mima} for $n = 1$ using EE; error with respect to $\eps$ for various $\Delta t$ (left) and with respect to $\Delta t$ for various $\eps$ (right).}
\label{Fig:bloch_MM_scheme_3F}
\end{center}
\end{figure}

\section{Conclusion}

This paper was concerned with a generic linear differential equation, with a time-dependent forcing which can be split in a quasi-periodic part and an exponentially decaying part. 
Adapting averaging techniques, we performed a micro-macro decomposition, which was proven to be well-posed. 
We then obtained suitable estimates on time-derivatives of the micro and macro variables meaning that the micro-macro problem can be solved with uniform accuracy using a standard scheme. 

Using a toy problem for which the exact solution is known, we illustrated the different components of this approach e.g. the size of the micro-part and of its derivatives, thereby validating the uniform accuracy results. 
Then, we successfully applied it to a transitional model derived from the Bloch model.

A continuation of this work is to propose an approach to solve numerically the original Bloch model \eqref{eq:bloch2} governing both populations and coherences. 
We believe the information learnt in this paper on the populations given by the transitional model can be enriched using predicted coherences in the population equation and then correcting them with an appropriated integral scheme.

\bibliographystyle{alpha}
\bibliography{biblio.bib}

\appendix

\section{Auxiliary proofs}

In this appendix, we present proofs of some technical results used in Sections~\ref{sec:results} and \ref{sec:proofs}.

\subsection{Sharp-flat decomposition}
\label{app:sec:sharp-flat-properties}

\begin{proof}[Proof of the direct sum defined in Definition \ref{def:direct_sum}]
	To prove that $\calE^\qp_\kappa$ and $\calE^\dmp$ are in direct sum, we show that any function $\varphi \in \calE^\qp_\kappa \cap \calE^\dmp$ is necessarily zero.
	Since $\varphi \in \calE^\qp_\kappa$ it can be represented as Fourier series and for all $\tau \geq 0$,
	\begin{equation*}
		\varphi_\tau = \varphi^\qp_{\omega\tau} = \sum_{\alpha \in \setZ^r} e^{i(\alpha \cdot \omega)\tau} \widehat\varphi_\alpha
		\text{ where } 
		\widehat\varphi_\alpha = \frac1{(2\pi)^r} \int_{\setT^r} e^{-i\alpha \cdot \theta} \varphi^\qp_\theta \D\theta.
	\end{equation*}
	Thanks to Arnold's theorem,
	\begin{equation*}
		\widehat\varphi_\alpha = \lim_{T \to \infty} \frac1T \int_0^T e^{-i\alpha \cdot (\omega\tau) } \varphi^\qp_{\omega\tau} \D\tau
		= \lim_{T \to \infty} \frac1T \int_0^T e^{-i(\alpha \cdot \omega)\tau)} \varphi_\tau \D\tau.
	\end{equation*}
	Since we also suppose that $\varphi \in \calE^\dmp$, $\| \varphi \| \leq e^{-\tau} \| \varphi \|_\Lexp$ and therefore
	\begin{equation*}
		\left| \frac1T \int_0^T e^{-i(\alpha \cdot \omega)\tau) } \varphi_\tau \D\tau \right| 
		\leq \frac1T \int_0^T e^{-\tau} \D\tau \| \varphi \|_\Lexp \leq \frac1T \| \varphi \|_\Lexp,
	\end{equation*}
	which tends to $0$ as $T$ tends to infinity. 
	Hence, $\widehat\varphi_\alpha = 0$ for all $\alpha$ and $\varphi = 0$.
\end{proof}

\begin{proof}[Proof of Proposition \ref{prop:algebra}]
	Let $\varphi$ and $\widetilde\varphi$ in $\calE_\kappa$. 
	They are uniquely decomposed as $\varphi_\tau = \varphi^\qp_{\omega\tau} + \varphi^\dmp_\tau$ and $\widetilde\varphi_\tau = \widetilde\varphi^\qp_{\omega\tau} + \widetilde\varphi^\dmp_\tau$. 
	So, the product $\psi = \varphi \widetilde\varphi$ verifies
	\begin{equation*}
		\psi_\tau = (\varphi \widetilde\varphi)_\tau 
		= \varphi^\qp_{\omega\tau} \widetilde\varphi^\qp_{\omega\tau} + \varphi^\qp_{\omega\tau} \widetilde\varphi^\dmp_\tau 
		+ \varphi^\dmp_\tau \widetilde\varphi^\qp_{\omega\tau} + \varphi^\dmp_\tau \widetilde\varphi^\dmp_\tau.
	\end{equation*}
	We introduce $\psi^\qp$ and $\psi^\dmp$ such that $\psi^\qp_{\omega \tau} = \varphi^\qp_{\omega\tau} \widetilde\varphi^\qp_{\omega\tau}$ and $\psi^\dmp_\tau = \varphi^\qp_{\omega\tau}\widetilde\varphi^\dmp_\tau + \varphi^\dmp_\tau \widetilde\varphi^\qp_{\omega\tau} + \varphi^\dmp_\tau \widetilde\varphi^\dmp_\tau$.	
	We first estimate
	\begin{align}
		| e^\tau \psi^\dmp_\tau |
		& \leq \| \varphi^\qp \| | e^\tau \widetilde\varphi^\dmp_\tau| + | e^\tau \varphi^\dmp_\tau | (\| \widetilde\varphi^\qp \| + \| \widetilde\varphi^\dmp \|) \nonumber \\
		& \label{eq:prodflat}
		\leq \| \varphi^\qp \| \| \widetilde\varphi^\dmp \|_\Lexp + \| \varphi^\dmp \|_\Lexp (\| \widetilde\varphi^\qp \| + \| \widetilde\varphi^\dmp \|) < +\infty.
	\end{align}
	Hence $\psi^\dmp \in \calE^\dmp$.
	Now we compute the Fourier series
	\begin{align}
		\sum_{\alpha \in \setZ^r} e^{i\alpha \cdot \theta} \widehat{\varphi \widetilde\varphi}_\alpha
		& = \sum_{\alpha \in \setZ^r} e^{i\alpha \cdot \theta} (\widehat\varphi \star \widehat{\widetilde\varphi})_\alpha 
		= \sum_{\alpha \in \setZ^r} e^{i\alpha \cdot \theta} \sum_{\beta \in \setZ^r} \widehat\varphi_\beta \widehat{\widetilde\varphi}_{\beta - \alpha} \nonumber \\
		& \label{eq:prodseries}
		= \sum_{\beta \in \setZ^r} e^{i\beta \cdot \theta} \widehat\varphi_\beta \sum_{\gamma \in \setZ^r} e^{i\gamma \cdot \theta} \widehat{\widetilde\varphi}_\gamma 
		= \varphi^\qp_\theta \widetilde\varphi^\qp_\theta.
	\end{align}
	By the uniqueness of the decomposition, we therefore have $\psi \in \calE$ and $\psi_\tau = \psi^\qp_{\omega \tau} + \psi^\dmp_\tau$. 
	Since we have only used the fact that $\varphi$ and $\widetilde\varphi$ are in $\calE$, $\calE$ is an algebra. 
	Let us now prove that $\calE_\kappa$ is an algebra. 
	To this aim we estimate $\calN_\kappa(\psi)$. 
	From equation \eqref{eq:prodseries}
	\begin{align*}
		\| \psi^\qp \|_\kappa = \| \varphi^\qp \widetilde\varphi^\qp \|_\kappa 
		& \leq \sum_{\gamma \in \setZ^r} e^{\kappa|\gamma|} \sum_{\alpha, \beta \in \setZ^r, \alpha + \beta = \gamma} \left| \hat\varphi_\alpha \widehat{\widetilde\varphi}_\beta \right| \\
		& \leq  \sum_{\alpha, \beta \in \setZ^r} e^{\kappa|\alpha|} e^{\kappa|\beta|} \left| \hat\varphi_\alpha \right| | \widehat{\widetilde\varphi}_\beta |
		= \| \varphi^\qp \|_\kappa \| \widetilde\varphi^\qp \|_\kappa.
	\end{align*}
	From \eqref{eq:prodflat}, we also immediately have
	\begin{equation*}
	 	\| \psi^\dmp \|_\Lexp
	 	\leq \| \varphi^\qp \|_\kappa \| \widetilde\varphi^\dmp \|_\Lexp + \| \varphi^\dmp \|_\Lexp (\| \widetilde\varphi^\qp \|_\kappa + \| \widetilde\varphi^\dmp \|_\Lexp).
	\end{equation*}
	Gathering the above estimates
	\begin{equation*}
		\calN_\kappa(\psi) \leq \calN_\kappa(\varphi) \calN_\kappa(\widetilde\varphi).
	\end{equation*}
	This in particular implies that $\calE_\kappa$ is an algebra.
\end{proof}

\subsection{Integration in sharp-flat spaces} 
\label{app:integration_lemma}

\begin{proof}[Proof of Lemma~\ref{lemma:integration}.]
	Thanks to the sharp-flat decomposition we can write
	\begin{equation*}
		\psi_\tau - \mean{\psi} = \sum_{\alpha \neq 0} e^{i(\alpha \cdot \omega)} \widehat\psi_\alpha + \psi^\dmp_\tau.
	\end{equation*}
	Integrating this
	\begin{equation*}
		\varphi_\tau = \mean{\varphi} + \sum_{\alpha \neq 0} \frac{e^{i(\alpha \cdot \omega)}}{i(\alpha \cdot \omega)} \widehat\psi_\alpha + \int_{+\infty}^\tau \psi^\dmp_\sigma \D\sigma,
	\end{equation*}
        where $\mean{\varphi}$ is the integration constant, considering that the sequel has zero mean. 
        This yields the sharp-flat decomposition for $\varphi$:
        	\begin{equation*}
		\varphi^\qp_{\omega\tau} = \mean{\varphi} + \sum_{\alpha \neq 0} e^{i(\alpha \cdot \omega)} \frac{\widehat\psi_\alpha}{i(\alpha \cdot \omega)} 
		\text{ and } \varphi^\dmp_\tau = \int_{+\infty}^\tau \psi^\dmp_\sigma \D\sigma.
	\end{equation*}
	We first estimate
	\begin{equation*}
		\| \varphi^\dmp \|_\Lexp 
		= \sup_{\tau \geq 0} e^\tau \left| \int_{+\infty}^\tau \psi^\dmp_\sigma \D\sigma \right|
		\leq \sup_{\tau \geq 0} \int_\tau^{+\infty} e^{(\tau - \sigma)} \| \psi^\dmp \|_\Lexp \D\sigma = \| \psi^\dmp \|_\Lexp.
	\end{equation*}
	Then 
	\begin{equation*}
		\| \varphi^\qp - \mean{\varphi} \|_{\kappa_-} 
		= \sum_{\alpha \neq 0} e^{\kappa_ - |\alpha|} \frac{| \widehat\psi_\alpha |}{\alpha \cdot \omega}
        		= \sum_{\alpha \neq 0} \frac{e^{-(\kappa_+ - \kappa_-)|\alpha|}}{\alpha \cdot \omega} e^{\kappa_ + |\alpha|} | \widehat\psi_\alpha |.
	\end{equation*}
	Using the Diophantine condition
	\begin{equation*}
		\| \varphi^\qp - \mean{\varphi} \|_{\kappa_-} \leq c_\nu(\kappa_-) \sum_{\alpha \neq 0} e^{\kappa_ + |\alpha|} | \widehat\psi_\alpha |
		=  c_\nu(\kappa_-) \| \psi^\qp - \mean{\psi} \|_{\kappa_+}, 
	\end{equation*}
	where we have defined
	\begin{equation*}
		c_\nu(\kappa) = \frac1{c_D} \sup_{x > 0} x^\nu e^{-\kappa x}
		= \begin{cases}
			\frac1{c_D} \left( \frac\nu{\kappa e} \right)^\nu & \text{ if } \nu \neq 0, \\
			\frac1{c_D} & \text{ if } \nu = 0.
		\end{cases} 
	\end{equation*}
	Gathering the above estimates and using the definition of $c_I$ \eqref{eq:cI}
	\begin{align*}
		\calN_{\kappa_-}(\varphi - \mean{\varphi}) & = \| \varphi - \mean{\varphi} \|_{\kappa_-} + \| \varphi^\dmp \|_\Lexp \\
		& \leq \max(1,c_\nu(\kappa_+ - \kappa_-)) \left( \| \psi^\qp - \mean{\psi} \|_{\kappa_+} + \| \psi^\dmp \|_\Lexp \right) \\
		& = c_I(\kappa_+ - \kappa_-) \calN_{\kappa_+}(\psi - \mean{\psi}).
	\end{align*}
\end{proof}

\subsection{Properties of the $\Lambda$ operator}
\label{app:bound_Lambda}

\begin{proof}[Proof of Lemma~\ref{lemma:bound_Lambda}.]
	By definition
	\begin{equation*}
		\Lambda\{ \varphi \}_\tau = a_\tau \varphi_\tau - \varphi_\tau \mean{a \varphi}.
	\end{equation*}
	The algebraic properties of Proposition~\ref{prop:algebra} and inequalities \eqref{eq:ineq_norm_average} ensure the direct bound
	\begin{equation*}
		\calN_\kappa(\Lambda\{ \varphi \}) 
		\leq \calN_\kappa(a) \calN_\kappa(\varphi) + \calN_\kappa(\varphi) | \mean{a\varphi} |
		\leq  \left( 1 + \calN_\kappa(\varphi) \right) \calN_\kappa(a) \calN_\kappa(\varphi).
	\end{equation*}
	Since $\calN_\kappa(\varphi - \id) \leq c$,
	\begin{equation*}
		\calN_\kappa(\varphi) \leq \calN_\kappa(\id) + \calN_\kappa(\varphi - \id) \leq 1 + c.
	\end{equation*}
	Gathering theses estimates, we obtain
	\begin{equation*}
		\calN_\kappa(\Lambda\{ \varphi \}) \leq \left(1 + (1+c) \right) M (1+c) = (2+c) (1+c) M.
	\end{equation*}
	Last 
	\begin{equation*}
		\Lambda\{ \varphi \} - \Lambda\{ \widetilde\varphi \}
		= a \left( \varphi - \widetilde\varphi \right) 
		- \varphi \Mean{a \left( \varphi - \widetilde\varphi \right)}
		- \left( \varphi - \widetilde\varphi \right) \mean{a \widetilde\varphi},
	\end{equation*}
	and this can be bounded as
	\begin{equation*}
		\calN_\kappa(\Lambda\{ \varphi \} - \Lambda\{ \widetilde\varphi \})
		\leq \left( 1 + (1+c) + (1+c) \right) M \calN_\kappa(\varphi - \widetilde\varphi) 
		= (3+2c) M \calN_\kappa(\varphi - \widetilde\varphi) . 
	\end{equation*}
	Setting $N_c = (2+c) (1+c)$ and $L_c = 3+2c$, we obtain the estimates of Lemma~\ref{lemma:bound_Lambda}.
\end{proof}

\begin{remark*}
	The same type of estimates are also valid with no closure assumption. 
	In this case, we find that $\mean{\varphi}^{-1}$ can be bounded from above by $1/(1-c)$, and the constants should be $N_c = 2(1+c)/(1-c)$ and $L_c = 4/(1-c)^2$.
\end{remark*}

\subsection{Properties of the derivatives of the $\Lambda$ operator}
\label{app:bound_der_Lambda}

\begin{proof}[Proof of Lemma~\ref{lemma:bound_der_Lambda}.]
	Using the Leibniz's product rule
	\begin{equation*}
		\partial_\tau^p (a_\tau \varphi_\tau) = \sum_{p' = 0}^p \binom{p}{p'} \partial_\tau^{p'} a_\tau \partial_\tau^{p-p'} \varphi_\tau,
	\end{equation*}
	we have
	\begin{equation*}
		\| \partial_\tau^p (a\varphi) \| 
		\leq \Big( \sup_{0 \leq p \leq q} \| \partial_\tau^p a \Big) \Big( \sup_{0 \leq p \leq q} \| \partial_\tau^p \varphi \| \Big) \sum_{p' = 0}^p \binom{p}{p'}  
		\leq 2^q C_a^{(q)} M \sup_{0 \leq p \leq q} \| \partial_\tau^p \varphi \|.
	\end{equation*}
	Recalling also that $|\mean{a \varphi}| \leq M (1+c)$, we obtain
	\begin{equation*}
		\sup_{0 \leq p \leq q} \| \partial_\tau^p (\Lambda\{ \varphi \}) \| 
		= \sup_{0 \leq p \leq q} \| \partial_\tau^p (a \varphi - \varphi \mean{a \varphi} \| 
		\leq (2^qC_a^{(q)} + 1+c) M \sup_{0 \leq p \leq q} \| \partial_\tau^p \varphi \|.
	\end{equation*}
	Moreover, assuming also that $\displaystyle \sup_{0 \leq p \leq q} \| \partial_\tau^p \varphi \| \leq c^{(q)}$ and $\displaystyle \sup_{0 \leq p \leq q} \| \partial_\tau^p \widetilde\varphi \|\leq c^{(q)}$, 
	\begin{align*}
		\sup_{0 \leq p \leq q} \| \partial_\tau^p (\Lambda\{ \varphi \} - \Lambda\{ \widetilde\varphi \}) \|
		& = \sup_{0 \leq p \leq q} \left\| \partial_\tau^p \left( a (\varphi - \widetilde\varphi) - \varphi \Mean{a (\varphi - \widetilde\varphi)} 
		- (\varphi - \widetilde\varphi) \mean{a \widetilde\varphi} \right) \right\| \\
		& \leq (2^q C_a^{(q)} + c^{(q)} + 1 + c) M \sup_{0 \leq p \leq q} \| \partial_\tau^p (\varphi - \widetilde\varphi) \|. 
	\end{align*}
	Setting $N^{(q)}_c = 2^qC_a^{(q)} + 1+c$ and $L^{(q)}_c = N^{(q)}_c + c^{(q)}$, we obtain the estimates of Lemma~\ref{lemma:bound_der_Lambda}.
\end{proof}

\section{Bloch computations}

In this appendix, we present some computations to clarify the description and the implementation of the rate equations derived from the Bloch model, considering that the quantum system is forced by a (quasi-)periodic wave.

\subsection{Explicit expressions of the transition rates} 
\label{appendix:psi}

We consider a $r$-chromatic wave of the form
\begin{equation*}
	V^{\qp}(\tau) = \frac{E_0}{r}\sum_{p = 1}^r \cos(\omega_p \tau).
\end{equation*}
Using the computation
\begin{equation*}
	\int_0^\tau e^{\Omega \sigma} \cos(\omega (\tau - \sigma)) \D\sigma
	= \frac{\Omega e^{\Omega \tau}}{\Omega^2 + \omega^2} 
	+ \frac\omega{\Omega^2 + \omega^2} \sin(\omega \tau) 
	- \frac\Omega{\Omega^2 + \omega^2} \cos(\omega \tau)
\end{equation*}
and introducing the coefficients
\begin{equation}
	\label{eq:appendix_SR}
	R(\tau, \omega, \Omega) = \Re \left( \frac{\omega e^{\Omega \tau}}{\omega^2 + \Omega^2} \right) 
	\text{ and }
	S(\tau, \omega, \Omega) = - \Re \left( \frac{\Omega e^{\Omega \tau}}{\omega^2 + \Omega^2} \right),
\end{equation}
we obtain expressions for $\Psi$ and $\Psi^\infty$ defined respectively in \eqref{eq:bloch_psi1} and \eqref{eq:bloch_psi_inf}. 
They  read
\begin{equation*}
	\Psi_\tau = \Psi^\qp_{\omega \tau} + \Psi^\dmp_\tau 
	\text{ and }
	\Psi^\infty_\tau = \Psi^\qp_{\omega \tau}
\end{equation*}
with
\begin{equation*}
	(\Psi^\qp_{\omega \tau})_{lj} 
	= \frac{2 E_0^2}{r^2} |p_{lj}|^2 \sum_{p_1 = 1}^r \sum_{p_2 = 1}^r \cos(\omega_{p_1} \tau) 
	\Big( \sin(\omega_{p_2} \tau) R(0, \omega_{p_2}, \Omega_{lj}) + \cos(\omega_{p_2} \tau)  S(0, \omega_{p_2}, \Omega_{lj}) \Big),
\end{equation*}
and
\begin{equation*}
	(\Psi^\dmp_\tau)_{lj}  
	= - \frac{2 E_0^2}{r^2} |p_{lj}|^2 \sum_{p_1 = 1}^r \sum_{p_2 = 1}^r \cos(\omega_{p_1} \tau)  S(\tau, \omega_{p_2}, \Omega_{lj}).
\end{equation*}
We remark that the term $e^{\Omega \tau}$ appears only in the expression of $\Psi^\dmp_\tau$ meaning that the frequencies of the term $\Psi^\qp_{\omega \tau}$ are those of the electromagnetic wave, the eigenfrequencies of the quantum system acting only on the amplitude coefficients. 
Then, the average transition rate defined in \eqref{eq:bloch_psi_average} may be cast as
\begin{equation*}
	\mean{\Psi}_{lj} = \frac{E_0^2}{r^2} |p_{lj}|^2 \sum_{p = 1}^r S(0, \omega_p, \Omega_{lj}).
\end{equation*}
Notice that 
\begin{equation*}
	S(0, \omega, \Omega_{lj})
	= \frac{\gamma_{lj}}2 \left( \frac1{\gamma_{lj}^2 + (\omega + E_l - E_j)^2} + \frac1{\gamma_{lj}^2 + (\omega - E_l + E_j)^2} \right).
\end{equation*}
This explicit expression emphasizes that some resonances can occur between the high-frequency oscillations of the electromagnetic wave (carried by $\omega$) and that of the quantum system (carried by the eigenfrequencies $E_l - E_j$).

\subsection{Off-line computations for the Bloch micro-macro problem}\label{appendix:off-line_computations}

For the simplicity of the presentation, we consider in this appendix a mono-chromatic wave $V^\qp(\tau) = E_0 \cos(\omega \tau)$ and compute  coefficients needed for the implementation of the micro-macro scheme of order 1 associated to equation \eqref{eq:bloch_pop_osc}. 
The treatment of the $r$-chromatic wave and the addition of the exponentially decreasing terms associated to equation \eqref{eq:bloch_pop1} only make more complex the expressions and do not bring further difficulties. 

For this restricted case, we simplify notations introduced in \eqref{eq:appendix_SR} denoting
\begin{equation*}
	R_{lj} = R(0, \omega, \Omega_{lj}) \text{ and } S_{lj} = S(0, \omega, \Omega_{lj}).
\end{equation*} 
The transition rate reads
\begin{equation*}
	(\Psi^\infty_\tau)_{lj} = (\Psi^\qp_{\omega \tau})_{lj}  
	= 2 E_0^2 |p_{lj}|^2 \Big( \cos(\omega \tau) \sin(\omega \tau) R_{lj} + \cos^2(\omega \tau) S_{lj} \Big).
\end{equation*}
and its average is given by
\begin{equation*}
	\mean{\Psi}_{lj} =  E_0^2 |p_{lj}|^2 S_{lj}.
\end{equation*}
Then, in order to obtain the near-identity map at first order $\Phi\rk1$, we need to compute the integral
\begin{equation*}
	(\Upsilon^\infty_\tau)_{lj}
	:= \int_0^\tau \Big( (\Psi^\infty_\sigma)_{lj} - \mean{\Psi}_{lj} \Big) \D\sigma 
	=  \frac{E_0^2 |p_{lj}|^2}\omega \Big( \sin^2(\omega \tau) R_{lj} + \sin(\omega \tau) \cos(\omega \tau) S_{lj} \Big)
\end{equation*}
as well as its average
\begin{equation*}
	\mean{\Upsilon^\infty}_{lj} = \frac{E_0^2 |p_{lj}|^2}{2\omega} R_{lj}.
\end{equation*}
Finally, in order to compute $A\rk1$ as well as $\Lambda\big\{ \Phi\rk1 \big\}$, we consider the product
\begin{align*}
	(\Psi^\infty_\tau)_{lj} (\Upsilon^\infty_\tau)_{ki}
	= \frac{2 E_0^4 |p_{lj}|^2 |p_{ki}|^2}\omega \Big( 
		& \cos(\omega \tau) \sin^3(\omega \tau) R_{lj} R_{ki} + \cos^3(\omega \tau) \sin(\omega \tau) S_{lj} S_{ki} \\
		& + \cos^2(\omega \tau) \sin^2(\omega \tau) \big( R_{lj} S_{ki} + S_{lj} R_{ki} \big) \Big).
\end{align*}
Since $\mean{\cos(\omega \tau) \sin^3(\omega \tau)} = \mean{\cos^3(\omega \tau) \sin(\omega \tau)} = 0$ and $\mean{\cos^2(\omega \tau) \sin^2(\omega \tau)} = \frac18$, the average of this product is
\begin{equation*}
	\mean{\Psi^\infty_{lj} \Upsilon^\infty_{ki}} = \frac{E_0^4 |p_{lj}|^2 |p_{ki}|^2}{4\omega} \big( R_{lj} S_{ki} + S_{lj} R_{ki} \big).
\end{equation*}

\end{document}